\documentclass[11pt,a4paper,reqno]{amsart}

\pdfoutput=1

 \usepackage[margin=3cm]{geometry}
\usepackage{dsfont}

\usepackage{amsmath, amsthm,   stmaryrd,bbm,graphicx,mathtools,enumerate}

\usepackage{mathrsfs}
\usepackage[english]{babel}
\usepackage[tracking,spacing,kerning,babel]{microtype}
\usepackage{color}
 
\usepackage{framed} 
\usepackage{wrapfig} 

\usepackage{mathalfa}

\usepackage{upgreek}
 
\usepackage{multicol}
\usepackage[mathscr]{}

\usepackage{moresize}

\usepackage[final]{hyperref}   
\usepackage[dvipsnames]{xcolor}
\hypersetup{
	colorlinks=true,
	pdfpagemode=UseNone,
    citecolor=Green,
    linkcolor=OrangeRed,
    urlcolor=black,
	pdfstartview=FitW
} 

\usepackage[tt=false]{libertinus}       
\usepackage[T1]{fontenc} 
\usepackage{textcomp}
\usepackage[amsthm]{libertinust1math}
\usepackage{microtype} 
\newtheoremstyle{mine}
{\baselineskip}
{\baselineskip}
{\itshape}
{
}
{\bfseries}
{.}
{.5em}
{#1 #2\ifx#3\relax\else~(#3)\fi}

\theoremstyle{mine}

\newtheorem{thm}{Theorem}[section]
\newtheorem{prop}[thm]{Proposition}
\newtheorem{lem}[thm]{Lemma}

\theoremstyle{remark}
\newtheorem{rem}{Remark}

\colorlet{shadecolor}{blue!10}

\renewcommand{\bar}{\widebar}

\renewcommand{\epsilon}{\varepsilon}

\def\tM{\mathbf{M}}
\def\M{{\mathcal M}}
\def\W{{\mathcal W}}
\def\D{{\mathcal D}}
\def\R{{\mathbb R}}

\def\P{{\bf P}}
\def\E{{\bf E}}
\def\Q{{\bf Q}}

\def\N{{\mathbb N}}
\def\T{{\mathbb T}}
\def\L{{\mathcal L}}

\def\C{{\mathcal C}}

\def\en{\mathcal{E}}
 
\def\d{\, \mathrm{d}}

\def\cb{{\bf c}}
\def\nf{\mathcal{F}}

\def\mS{{\underline{S}}}
\def\MS{{\overline{S}}}

\newtheorem{condition}{Assumption}
\newcommand{\ind}[1]{\mathbf{1}_{\left\{ #1 \right\}}}
\renewcommand{\epsilon}{\varepsilon}

\allowdisplaybreaks[1]

\numberwithin{equation}{section}

\title[Conditioned BRW]
{\bf Branching random walk conditioned on large martingale limit
}
 
 \author{Xinxin Chen, Lo\"ic de Raph\'elis, and Heng Ma}

\address[Xinxin Chen]
{Beijing Normal University, School of Mathematical Sciences, China   
}
\email{xinxin.chen(at)bnu(dot)edu(dot)cn}

\address[Lo\"ic de Raph\'elis]
{Institut Denis Poisson-C.N.R.S. UMR 7013-Universit\'e d'Orl\'eans, France
}
\email{loic.de.raphelis(at)univ-orleans(dot)fr}

\address[Heng Ma]{Heng Ma: School of Mathematical Sciences, Peking University,  P.R. China. }
\email{hengmamath(at)gmail(dot)com} \urladdr{\url{hengmamath.github.io}}

\begin{document}

\begin{abstract}
We consider a branching random walk in the non-boundary case where 
the additive martingale  $W_n$ converges a.s. and in mean to some non-degenerate limit $W_\infty$. We first establish the joint tail distribution of $W_\infty$ and the global minimum of this branching random walk. Next, conditioned on the event that the minimum is atypically small or conditioned on very large $W_\infty$, we study the branching random walk viewed from the minimum and obtain the convergence in law  in the vague sense. As a byproduct, we also get the right tail of the limit of derivative martingale. 
\end{abstract}

 \keywords{Branching random walk, additive martingale, derivative martingale, locally finite point process.} 
 
\maketitle

\section{Introduction: Models and Results}
\label{Intro}
\subsection{Branching random walk}
Let us first introduce a branching random walk on the real line, whose reproduction law is given by
a point process $\L$ on $\R$.
 The construction is as follows. 
We start with one individual at time $0$, which is called the root, denoted by $\rho$ and positioned at $V(\rho)=0$. At time $1$, the root gives birth to some children whose positions constitute a point process distributed as $\L$. These children form the first generation. Recursively, for any $n\geq0$, at time $n+1$, every vertex $u$ of the $n$-th generation produces its children independently of the other vertices so that the displacements of its children with respect to its position are distributed as $\L$. All children of the vertices of the $n$-th generation form the $(n+1)$-th generation.  
We hence get the genealogical tree $\T$ which is a Galton-Watson tree. For any vertex $u\in\T$, let $V(u)$ denote its position and $|u|$ denote its generation with $|\rho|=0$. Denote by $\P$ the law of the branching random walk $(\T, (V(u))_{u\in\T})$. For any $a\in\R$, the law of $(\T, (V(u)+a)_{u\in\T})$ is denoted by $\P_a$. Let $\E$ and $\E_a$ be the corresponding expectations, respectively.

An important tool in the analysis of branching random walk is  the log-Laplace transform of $\L$ defined by
\[
\psi(\theta):= \ln \E\left[\sum_{|u|=1}e^{-\theta V(u)}\right]= \ln \E\left[\int e^{-\theta x}\L(dx)\right]\in(-\infty,+\infty], \forall \theta \in\R.
\]
Note that $\psi(0)>0$ is equivalent to say that the Galton-Watson tree $\T$ is supercritical. In this case, the survival event $\{\T=\infty\}$ is of positive probability. We only work in this case throughout this paper. For every $\theta\in \R$ such that $\psi(\theta)<\infty$, define
\[
W_n(\theta):=\sum_{|u|=n}e^{-\theta V(u)-n\psi(\theta)}, \forall n\geq0.
\]
Obviously, $(W_n(\theta))_{n\geq0}$ is a non-negative $\P$-martingale with respect to the filtration of sigma-fields $\{\nf_n:=\sigma((u, V(u)); |u|\leq n)\}_{n\geq0}$, thus converges a.s. to some non-negative limit $W_\infty(\theta)$.  We usually call it the additive martingale or Biggins' martingale in honor of Biggins' seminal contribution on it (see for instance \cite{Biggins77}). When the Galton-Watson tree $\T$ is supercritical, the limit is positive a.s. on $\{\T=\infty\}$ if and only if $\{W_n(\theta)\}$ is uniformly integrable, according to \cite{Biggins77} and \cite{Lyons97}. 

Suppose that $D_\psi:=\{\theta: \psi(\theta)<\infty\}$ contains an open set. By taking derivative with respect to $\theta\in D^o_\psi$, one also gets that 
\[
D_n(\theta)=-\sum_{|u|=n}(V(u)+n\psi'(\theta))e^{-\theta V(u)-n\psi(\theta)},\ \  n\in\N,
\]
forms a $\P$-martingale with respect to the natural filtration  $\{\nf_n\}_{n\geq0}$. This is known as the derivative martingale.

 In this work, we only consider the parameter $\theta=1$ and study the martingales $W_n=W_n(1)$ and $D_n=D_n(1)$. The model is said to be in the boundary case in the sense of \cite{BK05} if $\psi(1)=\psi'(1)=0$. We work in the non-boundary case.
More precisely, through out this paper, we make the following assumptions.
\begin{condition}\label{cond1}
$\psi(0)>0$
and
\begin{align}\label{hyp1}
\E\left[\sum_{|u|=1}e^{-V(u)}\right]=& \, 1, \\
  \E\left[\sum_{|u|=1}|V(u)|e^{-V(u)}\right]<\infty \text{ and }&  \E\left[\sum_{|u|=1}V(u)e^{-V(u)}\right] > 0.
\end{align}
\end{condition}

\begin{condition}\label{cond2}
There exists 
\begin{equation*}
  \kappa:= \inf\{\theta>1: \psi(\theta)=0\}\in (1,\infty).
\end{equation*} 
\end{condition} 
\begin{condition}\label{cond3}
The support of $\L$ is non-lattice. 
\end{condition}
\begin{condition}\label{cond4}
There exists some  $\delta_0  >0 $ such that 

\begin{equation}\label{hyp1+}
\sup_{ \theta\in [1-\delta_0,\kappa+\delta_0]}\psi(\theta)<\infty,\textrm{ and } \E\left[\left(\sum_{|u|=1}(1+|V(u)|)e^{-V(u)}\right)^{\kappa+\delta_0}\right]<\infty.
\end{equation}
\end{condition}

 The convexity of $\psi$ implies that $\psi(\kappa)=0$ in Assumption \ref{cond2}. Assumptions \ref{cond1} and \ref{cond4} show that $\psi(1)=0>\psi'(1)$. So,
\[
W_n=\sum_{|u|=n}e^{-V(u)}, \textrm{ and } D_n = \sum_{|u|=n}(-V(u)-\psi'(1)n)e^{-V(u)}.
\]

It is well-known that under assumption \ref{cond1}, and mild moment conditions, the additive martingale $W_n$ converges in $L^1$ to some non-degenerate limit $W_\infty$ (see, e.g., \cite{Lyons97}. \cite{Lyons97}).
In Proposition \ref{Dcvg} we will show that 
under the assumptions \ref{cond1}, \ref{cond2}, \ref{cond3} and \ref{cond4},  the martingales $\{W_n\}$ and $\{D_n\}$ are bounded in $L^p$ for any $p \in (1,\kappa)$.   Consequently,   there exists a non-degenerate limit $D_\infty$ satisfying
\begin{equation*}
 \lim_{n \to \infty} D_n  =  D_{\infty}  \text{ a.s. and in $L^{p}$  for any $p \in (1,\kappa)$.}
\end{equation*} 
 Our  interest lies in the tail behavior of this martingale limit $D_\infty$, especially the right tail which is closely related to the right tail of $W_\infty$.  In fact, it is known from Theorem 2.2 in \cite{Liu00} that under Assumptions \ref{cond1}-\ref{cond4}, there exists a constant $C_0\in(0,\infty)$ such that 
\begin{equation}\label{tailW}
\P(W_\infty\geq x)=(C_0+o_x(1)) x^{-\kappa}, \textrm{ as } x\rightarrow\infty.
\end{equation}
We will also study the behavior of the branching random walk conditioned on large $W_\infty$.

The martingale limits gain many interests in the study of branching random walks. For the additive martingales, Biggins \cite{Biggins77} has determined the necessary and sufficient condition for the non-degeneracy of $W_\infty$. Lyons gave in \cite{Lyons97} a shorter proof of Biggins' result by using the so-called Lyons' change of measure which will be used in this work as well. The additive martingale also appears in the study of Mandelbrot's cascades and has been studied by Kahane and Peyri\`ere \cite{KP76}. Liu \cite{Liu00} investigated the moments, exponential moments and tail probabilities of the additive martingale limit. In particular, the tail \eqref{tailW} can be viewed as an application of Kesten's theorem on random difference equation, see \cite{Kest73}, \cite{Kevei16}. The derivative martingale, in the boundary case where $\psi'(1)=\psi(1)=0$ and the additive martingale vanishes, plays an important role in the study of the minimal position and the extremal point process, see \cite{ABK13}, \cite{ABBS13} etc. Further, in the boundary case, the Cauchy tail of derivative martingale is proved by Buraczewski, Iksanov and Mallein\cite{BIM21}. Recently, in the context of Gaussian multiplicative chaos, in \cite{LRV22}, the non-boundary derivative martingale, and the tail probabilities of its limit, are raised. For branching Wiener process, the left tail of non-boundary derivative martingale limit is partially answered in \cite{BV23}. We will give the right tail in this work.

To study $W_\infty$ and $D_\infty$, let us introduce the global minimum of the branching random walk which is defined by
\[
\tM:=\inf_{u\in\T}V(u) \leq  0
\]
In fact, it is known in \cite{Big76} that Assumption \ref{cond1} implies that on $\{\T=\infty\}$, a.s.,
\begin{equation}\label{Infgrow}
\inf_{|u|=n}V(u)\to +\infty,\textrm{ as } n\to\infty.
\end{equation}
This means that $\tM$ is $[0,\infty)$-valued, $\P$-a.s. in our setting. We first show that on $\{W_\infty\ge x\}$ with large $x\gg1$, the global minimum $\tM$ is around $-\ln x$ with high probability. So the law conditioned on $\{W_\infty \ge x\}$ is linked with the law conditioned on $\{\tM\le -\ln x\}$.

\begin{thm}\label{OptimalStrategy}
 Suppose that the assumptions \ref{cond1}, \ref{cond2}  and \ref{cond4} are  fulfilled; and $\delta_0$ is given in \eqref{hyp1+}.  
There exists some constant $C_*$ such that for any $x \geq 1$ and $z \geq 0$, 
  \begin{equation*} 
    \P( W_{\infty}> x, |\tM + \ln x |> z  ) \leq C_* e^{- (\kappa \wedge \delta_0) z} x^{-\kappa} .
  \end{equation*}
\end{thm}

Let us take $u^*\in\T$ such that $V(u^*)=\tM$, if there exist several choices, one chooses $u^*$ at random among the youngest ones. Then $|u^*|$ is the first generation where $\tM$ is achieved. In addition, let 
\[
\W^{\tM}:= W_\infty e^{\tM} \ \textrm{ and }\ \D^\tM:=e^\tM D_\infty.
\]
Here, recall that $\psi$ is a convex and smooth function on $(1-\delta_0, \kappa+\delta_0)$. So, 
\[
\psi'(1)<0<\psi'(\kappa), \ \psi''(\kappa)>0.
\]
Our first result is about the following convergence in law conditioned on very negative $\tM$. 
\begin{thm}\label{BRWcvg}
Suppose that the assumptions \ref{cond1},\ref{cond2}, \ref{cond3} and \ref{cond4} are all fulfilled. Then there exists a constant $c_\tM\in(0,\infty)$ such that as $z\rightarrow\infty$, 
\begin{equation}\label{tailM}
\P(\tM\leq -z)\sim c_\tM e^{-\kappa z}.
\end{equation}
Further, conditionally on $\{\tM\leq-z\}$, the following convergence in law holds as $z\rightarrow\infty$: 
\begin{equation*} 
\left(\sum_{u\in\T}\delta_{V(u)-\tM},  \frac{\sqrt{\psi'(\kappa)}}{\sqrt{z}}(|u^*|-\frac{z}{\psi'(\kappa)}),\frac{\D^\tM}{|u^*|}, \mathcal{W}^{\tM}, \tM+z \right)\Rightarrow (\en_\infty, G,(\psi'(\kappa)-\psi'(1))Z , Z, -U)   ,
\end{equation*}
where $U$ is an exponential random variable with mean $\kappa^{-1}$, $G$ is a Gaussian random variable with mean zero and variance $\frac{\psi''(\kappa)}{\psi'(\kappa)^2}$, $Z$ is positive random variable, $\en_\infty$ is a locally finite point process on $\mathbb{R}_+$. Moreover, $U$, $G$ and $(Z,\en_\infty)$ are independent. And the convergence in law of point processes holds in the vague sense.
\end{thm}

\begin{rem}
The fact that $Z>0$, $\P$-a.s. follows from Lemma \ref{BRWtailM}. Moreover, we could deduce from Lemma \ref{bdcondmomWD} that $\E[Z^\kappa]=C_0/c_\tM$ with $C_0$ in \eqref{tailW}. 
\end{rem}

\begin{rem}
Note that in the boundary case where $\psi(1)=\psi'(1)=0$, the additive martingale vanishes and the derivative martingale $D_n$ converges a.s. to some positive limit $D_\infty$ on $\{\T=\infty\}$ under some suitable moment condition (See for instance \cite{Aid13}). The global minimum $\tM$ is still a well-defined random variable in this case and \cite{Mad20} obtained the convergence in law of $e^{-x}D_\infty$ conditioned on $\tM\le -x$ and deduced the right tail of $D_\infty$ under suitable moment conditions. 
\end{rem}

Next, we deduce the right tail of $D_\infty$ with the help of $\tM$. Let 
\begin{equation}
\gamma(a) :=c_\tM \E \left[(1\wedge \frac{Z}{a})^\kappa \right],\quad \forall a>0. 
\end{equation}
Note that $\gamma: (0,\infty)\to (0,\infty)$ is non-increasing and continuous. Moreover, $\gamma(0+)=c_\tM$ and $\lim_{a\to0+}\gamma(\frac1a)a^{-\kappa}=c_\tM \E[Z^\kappa]= C_0$ with $C_0$ in \eqref{tailW}.

As a consequence of Theorem \ref{BRWcvg}, we get the following tail probabilities.
\begin{thm}\label{thmtailWM} Under the assumptions \ref{cond1}, \ref{cond2}, \ref{cond3} and \ref{cond4}, for any $a\ge 0$,
\begin{equation}
\P\left(W_\infty \geq ax, e^{-\tM }\geq x\right)
\sim \gamma(a)x^{-\kappa},\textrm{ as }x\rightarrow\infty.
\end{equation}
Moreover, there exists some constant $c_D:=C_0 [\frac{\psi'(\kappa)-\psi'(1)}{\psi'(\kappa)}]^\kappa$ such that
\begin{equation}\label{tailD}
\P(D_\infty \ge x)\sim c_D \frac{(\ln x)^{\kappa}}{x^\kappa},  \textrm{ as }x\rightarrow\infty.
\end{equation} 
\end{thm}

Furthermore, we analyze the  behavior of these quantities  conditioned on $\{W_\infty \ge x\}$. In particular, we have the following result. 

\begin{thm}\label{CondBRW}
Suppose that the assumptions \ref{cond1},\ref{cond2}, \ref{cond3} and \ref{cond4} are all fulfilled. Conditionally on $\{W_\infty \ge x\}$, the following convergence in law holds as $x\to\infty$:
\begin{equation}
\left(\sum_{u\in\T}\delta_{V(u)-\tM}, \frac{W_\infty}{x}, \frac{D_\infty}{x\ln x},\tM+\ln x \right)\Longrightarrow (\widehat{\en}_\infty, e^U, \frac{\psi'(\kappa)-\psi'(1)}{\psi'(\kappa)}e^U, \ln \widehat{Z}-U)
\end{equation}
where $U$ has exponential distribution with mean $\kappa^{-1}$, $(\widehat{\en}_\infty, \widehat{Z})$ is independent of $U$ and has the following distribution:
\[
\P\left((\widehat{\en}_\infty, \widehat{Z})\in \cdot\right)=\frac{1}{\E[Z^\kappa]}\E[Z^\kappa\ind{(\en_\infty, Z)\in \, \cdot}].
\]
\end{thm}

\begin{rem}
The weak convergence result for $(\sum_{u\in\T}\delta_{V(u)-\tM}, \tM+\ln x, \frac{W_\infty}{x})$ conditioned on $W_\infty \ge x$, still holds if the assumption \eqref{cond4} is weakened by $\psi(1-\delta_0)+\E[(W_1)^{\kappa+\delta_0}]<\infty$.
\end{rem}

Throughout this paper, $f(x)\lesssim g(x)$ if there exists some constant $c>0$ such that $f(x)\le c g(x)$. Moreover, if $c$ depends on some parameter $\alpha$, we write $f(x)\lesssim_\alpha g(x)$. And $f(x) \asymp g(x)$ means that $f(x)\lesssim g(x)$ and $g(x) \lesssim f(x)$. We use $\{c_n\}_{n\ge1}$ to denote positive constants. We write $f(x)=o_x(1)g(x)$ if $\frac{f(x)}{g(x)}$ goes to zero as $x\to\infty$ and $f(x)\sim g(x)$ if $\frac{f(x)}{g(x)}\to 1$ as $x\to\infty$. We will use $C_{(\cdot)}$ and $(c_i)_{i\ge0}$ to represent positive constants.

We finish this section with a short review of the organization of the paper. In Section \ref{RW}, we introduce Lyons' change of measure and spinal decomposition, as well as some preliminary results on random walks. In Section \ref{mainthm}, we prove Proposition \ref{Dcvg} and Theorems \ref{BRWcvg}, \ref{thmtailWM} and \ref{CondBRW}. At the end, there is an appendix where we give the proof of some technical lemmas.

\subsection*{Acknowledgement.} We would like to thank Pascal Maillard for useful discussions. The first author is supported by the National Key R\&D program of China No. 2022YFA1006500.

\section{Lyons' Change of Measure and Spinal Decomposition}\label{RW}

Recall that under $\P$, the branching random walk is constructed by use of the point process $\L$. Let us introduce a probability measure $\Q^{\theta,*}$ of a branching random walk with a spine: $\{(V(u); u\in\T), (w_n, V(w_n))_{n\geq0}\}$. First, fix $\theta>0$ such that $ \psi(\theta)<\infty$. As $\E[\int e^{-\theta x}\L(\d x)]=e^{\psi(\theta)}$, let $\widehat{\L}$ be a point process with Radon-Nykodim derivative $\int e^{-\theta x - \psi(\theta)}\L(\d x)$ with respect to the law of $\L$. We use $\widehat{\L}$ and $\L$ to construct  $\{(V(u); u\in\T), (w_n, V(w_n))_{n\geq0}\}$ under $\Q^{\theta,*}_a$ for any $a\in\R$ as follows. 

\begin{enumerate}[(i)]
\item For the root $\rho$, let $V(\rho)=a$ and $w_0=\rho$. $w_0$ gives birth to its children according to the point process $\widehat{\L}$ (i.e., the relative positions of its children with respect to $V(w_0)$ are distributed as $\widehat{\L}$).
\item For any $n\geq0$, suppose that the process with the spine $(w_k)_{0\leq k\leq n}$ has been constructed up to the $n$-th generation. All individuals of the $n$-th generation, except $w_n$, produce independently their children according to the law of $\L$. Yet, the individual $w_n$ produces its children, independently of the others, according to the law of $\widehat{\L}$. All the children of the individuals of the $n$-th generation form the $(n+1)$-th generation, whose positions are denoted by $V(\cdot)$. And among the children of $w_n$, we choose $w_{n+1}=u$ with probability proportional to $e^{-\theta V(u)}$.
\end{enumerate}
 
 We denote by $\Q^{\theta}_a$ the marginal distribution of $(\T, (V(u), u\in\T))$ under $\Q^{\theta,*}_a$. For simplicity, write $\Q^{\theta,*}$ and $\Q^{\theta}$ for $\Q_0^{\theta,*}$ and $\Q_0^\theta$ respectively. In particular, when $\theta=1$, we write $\Q^*$ and $\Q$ for $\Q^{1,*}$ and $\Q^{1}$. The corresponding expectations are denoted by $\E_{\Q^{\theta, *}}$ and $\E_{\Q^{\theta}}$.
 
 For the genealogical tree $\T$ and for two vertices $u,v\in\T$, write $u\leq v$ if $u$ is an ancestor of $v$ and write $u<v$ if $u\leq v$ but $u\neq v$. Let $c(u)$ denote the set of children of $u$ and let $\T_u:=\{v\in\T\vert v\ge u\}$ denote the subtree rooted at $u$. If $u\neq \rho$, let $\overleftarrow{u}$ be the parent of $u$ and let
 \[
 \Delta V(u)= V(u)-V(\overleftarrow{u})
 \]
 be the displacement of $u$. Moreover, let $\Omega(u)$ be the set of siblings of $u$.
 Let us state the following proposition given by Lyons \cite{Lyons97}.  

 \begin{prop}\label{BRWchangeofp}  Let $n \geq 0$ and $a \in \mathbb{R}$.  Let $\mathcal{F}_n$ be the sigma-field generated by $\{(u,V(u)); |u|\leq n\}$.
 \begin{enumerate} 
\item We have  
\[
\frac{\d \Q^{\theta}_a }{\d\P_a} \mid_{\mathcal{F}_n}=e^{\theta a} W_n(\theta) =\sum_{|u|=n}e^{-\theta V(u) +\theta a-n \psi(\theta)}.  
\]

\item For any vertex $u\in\T$ at the $n$-th generation,
\[
\Q^{\theta,*}_a(w_n=u\mid\mathcal{F}_n)=\frac{e^{-\theta V(u)-n \psi(\theta)}}{W_n(\theta)}.
\]
\item Under $\Q^{\theta,*}_a$, $\{V(w_{n})-V(w_{n-1}), \sum_{u\in\Omega(w_n)}\delta_{\Delta V(u)}\}_{n\ge1}$ are i.i.d. random variables. 
\end{enumerate}
\end{prop}

In particular, when $\theta=1$, it is known that $W_n$ converges to $W_\infty $ in $L^1(\P)$. Then we have
\[
\d\Q=\d \Q^1= W_\infty\d \P,
\]
where $W_\infty$ is also $\Q$-a.s. limit of $W_n$.

\vspace{6pt}

\textbf{The probability $\mathbf{Q}^{\theta,*}_{\vert n} \otimes \mathbf{P}$.} 
For any $n \in \mathbb{N}$, we introduce the probability law $\mathbf{Q}^{\theta,*}_{\vert n} \otimes\mathbf{P}$ for the marked branching random walk $\{(u,V(u))_{u\in\T}; (w_k, V(w_k))_{0\le k\le n}\}$, under which 
up to time $n$ $\{(u, V(u))_{|u|\le n}, (w_k, V(w_k))_{0\le k\le n}\}$ is distributed as a branching random walk under $\mathbf{Q}^{\theta,*}$ and after the time $n$ every alive particle at time $n$ will branch according to the original point process $\mathcal{L}$ under $\mathbf{P}$. By Proposition \ref{BRWchangeofp} and branching property, we have the following corresponding result.

\begin{prop}\label{BRWchangeofp2}
  For any $n\geq0$,  we have 
 \[
 \frac{\d (\Q^{\theta}_{\vert n} \otimes \P )  }{\d\P }  = W_{n}(\theta)=\sum_{|u|=n}e^{-\theta V(u)  -n \psi(\theta)}, 
 \]
and for any individual $u\in\T$ of the $n$-th generation,
 \[
  (\Q^{\theta,*}_{\vert n} \otimes \P ) (w_{n}= u | \mathcal{F}_{\infty}) =\frac{e^{-\theta V(u)-n \psi(\theta)}}{W_n(\theta)}.
 \]  
 \end{prop} 

\subsection{Many-to-One Lemma and Renewal theory.}
\label{lemBRW}

Recall that $\sup_{\theta\in(1-\delta_0, \kappa+\delta_0)}\psi(\theta)<\infty$  for some $\delta_0>0$ and $\psi(1)=\psi(\kappa)=0$.     We have the well-known many-to-one lemma.

\begin{lem}[Many-to-One]\label{ManytoOne}
Suppose the assumptions \ref{cond1}-\ref{cond4}.  
For any $n \geq 1, a \in \mathbb{R}$ and any measurable function $g: \mathbb{R}^n \rightarrow \mathbb{R}_{+}$, and for any $\theta \in(1-\delta_0, \kappa+\delta_0)$, 
\begin{equation*}
\mathbf{E} \left[\sum_{|z|=n} g\left(V\left(z_1\right), \cdots, V\left(z_n\right)\right)\right]=\mathbf{E} \left[ e^{\theta S^{(\theta)}_n + n \psi(\theta)} g\left(S^{(\theta)}_1, \cdots, S^{(\theta)}_n\right)\right],
\end{equation*}
where $(S^{(\theta)}_n :n\geq 1) $ is a random walk with i.i.d. increments such that 
\begin{equation*}
\mathbf{E}\left[S^{(\theta)}_1\right]  = - \psi'(\theta) \text{ and }   \mathbf{E}\left[e^{-\lambda S^{(\theta)}_{1} } \right]=\exp\{ \psi(\lambda+\theta)-\psi(\theta)  \}  
\end{equation*}
 for every $\lambda \in \mathbb{R}$ with $ \psi(\lambda+\theta)<\infty $. 
In particular, under $\Q^{\theta,*}_a$, $(V(w_i); 1\leq i\leq n)$ has the same distribution as $(S^{(\theta)}_i; 1\leq i\leq n)$ under $\P_a$ where $\P_a(S^{(\theta)}_0=a)=1$.  

\end{lem}

For simplicity we write  $S_{n}$ for $S^{(1)}_{n}$. Write $\kappa'=\kappa-1\in(0,\infty)$.
Immediately, we deduce from this lemma that
\begin{equation}\label{rwschange}
\E_a\left[g(S_1,\cdots,S_n)\right]=\E_a\left[e^{\kappa'(S^{(\kappa)}_n-a)}g(S^{(\kappa)}_1,\cdots,S^{(\kappa)}_n)\right].
\end{equation}
This means that the law of $(S_n^{(\kappa)})$ can be obtained from some Girsanov-type change of measure on $(S_n)$. Moreover, note that
\[
\E[S_1]=-\psi'(1)>0,\textrm{ and } \E[S_1^{(\kappa)}]=-\psi'(\kappa)<0.
\]

Next, let us state some classic results on random walks $(S_n)_{n\geq0}$ and $(S_n^{(\kappa)})_{n\geq0}$.

\subsubsection{Renewal theory for one-dimensional random walk.}
\label{Renewaltheory}

For the random walk $(S_n^{(\kappa)})_{n\geq0}$, we define the renewal measures $U_s^{(\kappa),\pm}$ corresponding to the strict ascending/descending ladder process by 
\begin{equation}
U_s^{(\kappa),\pm}([0,x]):=\E\left[\sum_{k=0}^{\tau^{(\kappa),\mp}-1}\ind{\pm S^{(\kappa)}_k\leq x}\right], \forall x\geq0.
\end{equation}
with $\tau^{(\kappa),+}:=\inf\{k\geq1: S^{(\kappa)}_k\geq0\}$ and $\tau^{(\kappa),-}:=\inf\{k\geq1: S^{(\kappa)}_k\leq 0\}$. As usual, we set the strict renewal functions to be $R_s^{(\kappa),\pm}(x):=U_s^{(\kappa),\pm}([0,x])$. Then it is known that 
there exist constants $C_s^{(\kappa),\pm}\in(0,\infty)$ such that for any $h>0$, as $x\rightarrow\infty$, 
\begin{equation}
  R^{(\kappa),+}_s(x)\rightarrow C_s^{(\kappa),+}\textrm{ and } U^{(\kappa),-}_s((x-h,x])\rightarrow C_s^{(\kappa),-}h. \label{eq-renewal-bound}
\end{equation} 
For any $0\leq j< n$, let
\[
\MS_{[j,n]}^{(\kappa)}:=\max_{j\leq k\leq n}{S_k}^{(\kappa)}, \mS_{[j,n]}^{(\kappa)}:=\min_{j\leq k\leq n}S_k^{(\kappa)}.
\]
Then we can rewrite $  R_s^{(\kappa),-}(x) $ as 
\[ 
  R_s^{(\kappa),-}(x) =U_s^{(\kappa),-}([0,x])= \sum_{n \geq 0} \P\left(   \MS^{(\kappa)}_{[1,n]}<0, S^{(\kappa)}_n \leq - x  \right). 
\] 
Thus for  for $I(x) = ( - x -1, - x ]$, we have
\begin{align}
&\sum_{n\ge 0} \E\left[ e^{\kappa S^{(\kappa)}_n+\kappa x}\ind{\MS^{(\kappa)}_{[1,n]}<0, S^{(\kappa)}_n\in I(x)} \right] \notag \\
&= \int_{[x,x+1)} e^{\kappa (x-y) } U_s^{ (\kappa), -}( \d y )  
\overset{x \to \infty}{\longrightarrow}  C_s^{( \kappa ), -} \frac{ 1- e^{-\kappa } }{\kappa} \label{RWeScvg}
\end{align} 
 In fact, we will see that in this sum, the main contribution comes from $n\approx \frac{x}{\psi'(\kappa)}$.
Let  
\begin{equation}\label{DefJx}
  J(x):=\left[\frac{x}{\psi'(\kappa)}-b(x)\sqrt{x}, \frac{x}{\psi'(\kappa)}+b(x)\sqrt{x}\right].
\end{equation} 
with $b(x)=o(x)$ and $b(x)\gg1$. We have the following Lemma which generalizes \eqref{RWeScvg}.

\begin{lem}\label{Rfunctioncvg}
Under the assumptions of Lemma \ref{ManytoOne}, 
for any  bounded continuous  function $\phi_0:\R\to\R_+$, and for any $a \geq 0$,   there exists some constant $\cb(a)>0$ such that,  
\begin{equation}\label{RWsumecvg}
\lim_{x\to\infty}\sum_{n\in J(x)} \phi_0(\sqrt{\frac{\psi'(\kappa)}{x}}(n-\frac{x}{\psi'(\kappa)}))\E_{-a}\left[e^{\kappa S^{(\kappa)}_n+\kappa x}\ind{\MS^{(\kappa)}_{[1,n]}<0, S^{(\kappa)}_n\in I(x)}\right] 
=\E[\phi_0(G)]\cb(a).
\end{equation}
where $G$ is a centred Gaussian random variable with variance $\frac{\psi''(\kappa)}{\psi'(\kappa)^2}$.

In addition, there exists a constant $c_0>0$ such that for any $x \geq 0$ and $a\ge0$,
\begin{equation}\label{RWeSbd}
   \sum_{k\geq0} \P_{-a}\left(  \MS^{(\kappa)}_{[1,k]}<0, S^{(\kappa)}_k\in I(x) \right) \leq  c_0(1+a) 
  \end{equation} 
\end{lem}

%
The proof of this Lemma is postponed to Appendix \ref{App}.

\section{Proof of Theorem \ref{OptimalStrategy}}

We begin by establishing a rough estimate on the tail of $\tM$.

\begin{lem}\label{BRWroughbd}
Under the assumptions of Theorem \ref{BRWcvg}, there exists $0<c_{1}\leq 1$ such that 
\begin{equation}\label{BRWrhbdM}
 c_{1} e^{-\kappa x} \leq \P(\tM\leq -x) \leq e^{-\kappa x}, \forall x\geq1.
\end{equation} 
\end{lem}

The next Proposition considers the moments of the martingale limits $W_\infty$ and $D_\infty$; and will be used in the proof of the Key Lemma \ref{lem-highmoments-martlim} in this paper. Since its proof is similar and  simpler than that of Lemma \ref{lem-highmoments-martlim},
we defer it to Appendix \ref{App}.

\begin{prop}\label{Dcvg}
  Suppose that the assumptions \ref{cond1}, \ref{cond2},  and \ref{cond4} hold. For any $p \in (1,\kappa)$, the sequences $W_n$ and $D_n$ are $L^{p}$-bounded martingales.Moreover, each sequence converges in $L^p$ to its limit   at an exponential  rate.
 \end{prop}

In the following, we define 
\[
\tM_n := \inf_{|u|\le n} V(u), \  \forall n\ge0.
\]

\begin{lem}\label{lem-highmoments-martlim}
  Under the Assumptions \ref{cond1}, \ref{cond2}, \ref{cond4} for any $\delta  \in (0, \delta_0]$ with $\delta_0$ in Assumption \ref{cond4}   we have 
  \begin{equation}\label{eq-bound-W}
    \E[ W_{n} ^{\kappa+\delta} \ind{\mathbf{M}_n \geq -x } ] \leq C_{\eqref{eq-bound-W}}(\delta)    e^{\delta x }  \quad   \forall \, n\geq 1 , x \geq 0, 
  \end{equation}
with some constant $C_{\eqref{eq-bound-W}}(\delta)$ depending on $\delta, \kappa$ and the BRW (see \eqref{eq-kappa+delta<2}, \eqref{eq-bound-W-k+delta>2} and \eqref{eq-constant-induction}). Similarly,   for $\delta \in (0,  \delta_0)$,  we have     
\begin{equation}\label{eq-desired-bound-D}
   \E[ |D_{n}|^{\kappa+\delta} \ind{\mathbf{M}_n \geq -x} ] \leq C_{\eqref{eq-desired-bound-D}}(\delta) e^{\delta x } (1+x)^{\kappa+\delta}  \quad  \forall \, n \geq 1, x \geq 0  
\end{equation} 
with $C_{\eqref{eq-desired-bound-D}}(\delta)$ depending on $\delta, \kappa$ and the BRW (see \eqref{eq-D-bound-k+delta<2} and   \eqref{eq-D-bound-k+delta>2}).
\end{lem}

We are now ready to prove our first theorem using the Lemmas above, whose proofs will be presented in Section \ref{sec4-proof}.

\begin{proof}[Proof of Theorem \ref{OptimalStrategy} assuming Lemmas \ref{BRWroughbd} and \ref{lem-highmoments-martlim}]
  First, by Lemma \ref{BRWroughbd}, we immediately obtain
\[  \P(  \tM \leq - \ln x - z  ) \leq   e^{- \kappa \ln x - \kappa z} =e^{-\kappa z} x^{-\kappa} .\]
Observe that since $\mathbf{M} \leq 0$, $     \tM \geq -\ln x + z   $ holds only if $z \leq \ln x$. In this case 
 by use of Markov's inequality,  for any $\delta>0$ we have 
 \[
  \P( W_{\infty}>x,  \tM \geq - \ln x + z  ) \leq  x^{- (\kappa+\delta)}  \E[ W_{\infty} ^{\kappa+\delta} \ind{\mathbf{M}  \geq - (\ln x- z) } ].
 \] 
 It follows from \eqref{eq-bound-W} and Fatou's lemma that  
\[  \E[ W_{\infty} ^{\kappa+\delta} \ind{\mathbf{M}  \geq - (\ln x- z) } ] \leq \liminf_{n \to \infty} \E[ W_{n} ^{\kappa+\delta} \ind{\mathbf{M}_{n} \geq -(\ln x- z) } ] \leq   C_{\eqref{eq-bound-W}}(\delta)    e^{\delta (\ln x- z) }  \]
Setting $\delta= \kappa \wedge \delta_0$ and  combining the above  with   previous inequalities yields 
\[  \P( W_{\infty}>x,  |\tM + \ln x| > z  ) \leq  x^{-\kappa} \left( e^{-\kappa z}+  C_{\eqref{eq-bound-W}}(\kappa) e^{- (\kappa \wedge \delta_0)z } \right) .\]
The desired result follows by taking $C_{*}:= 1+ C_{\eqref{eq-bound-W}}(\kappa)$.
\end{proof}

\subsection{Conditional moments of martingale limits: proof of Lemma \ref{lem-highmoments-martlim}}  
We begin by establishing a key moment estimate for the associated random walk.

\begin{lem}\label{lem-Sum-RW-bound}
  Given $a >0$ and $p \geq 1$ satisfying $a p> (\kappa-1)$.  For any $x \geq 0$,
     \begin{equation}\label{eq-lem-Sum-RW-bound}
       \E \left[ \left( \sum_{k=0}^{\infty} e^{-  a [S_{k}^{(1)} +x]} 1_{\{ S_{k}^{(1)} \geq -x\}}   \right)^{p} \right] \leq C_{\eqref{eq-lem-Sum-RW-bound}}(a,p) e^{-(\kappa-1)x} . 
     \end{equation}
   \end{lem}

  \begin{proof}[Proof of Lemma \ref{lem-Sum-RW-bound}]

    For each $j \geq 0$, let $\mathtt{L}_{x,j}:= \sum_{n \geq 0} 1_{ \{ S_{n}^{(1)}+x \in [j,j+1) \} }$.  We have 
    \begin{equation*}
      \sum_{n=0}^{\infty}   e^{- a (S^{(1)}_{n}+x) }1_{\{  S^{(1)}_{n} +x  \geq 0 \}}  = \sum_{j \geq 0} \sum_{n=0}^{\infty}   e^{- a (S^{(1)}_{n}+x) }1_{\{  S^{(1)}_{n}+x \in [j,j+1)  \}} \leq \sum_{j \geq 0} e^{-a j} \mathtt{L}_{x,j}.
    \end{equation*}
    The  Minkowski  inequality yields that  
    \begin{equation}\label{eq-desired-bound-middble}
      \E \left[ \left( \sum_{n=0}^{\infty}   e^{- a(S^{(1)}_{n} +x)}1_{\{  S^{(1)}_{n} +x  \geq 0 \}}   \right)^{p}  \right]^{\frac{1}{p}} \leq   \sum_{j \geq 0} e^{-a j} \left( \E [  \mathtt{L}_{x,j}  ^{p} ] \right)^{\frac{1}{p}} .
    \end{equation} 
     We claim that   there exist a  constant $C_{\eqref{eq-L-x-j-p-moment}}(p)$   such that 
    \begin{equation}\label{eq-occupation-2}
      \begin{aligned}
        \E [  \mathtt{L}_{x,j}^{p} ] \leq  C_{\eqref{eq-L-x-j-p-moment}}(1)  C_{\eqref{eq-L-x-j-p-moment}}(p)  \,  e^{-k'(x-j)_{+}}  .
      \end{aligned} 
    \end{equation}
    Then substituting  \eqref{eq-occupation-2} into \eqref{eq-desired-bound-middble} we obtain that 
    \begin{align*}
      \E  \left[ \left( \sum_{n=0}^{\infty}   e^{-a(S^{(1)}_{n}+x) }1_{\{  S^{(1)}_{n} +x \geq 0 \}}   \right)^{p}  \right]  ^{\frac{1}{p}}
       & \leq   
       C_{\eqref{eq-L-x-j-p-moment}}(1)  C_{\eqref{eq-L-x-j-p-moment}}(p) \,    \left( 
        \sum_{0\le j \leq x} e^{-a j}  e^{-\frac{\kappa'}{p} (x-j)} +  \sum_{j \geq x} e^{-a j} \right) \\
        & \leq 3 e^{-\frac{\kappa'}{p} x}  C_{\eqref{eq-L-x-j-p-moment}}(1)  C_{\eqref{eq-L-x-j-p-moment}}(\kappa) \,  \sum_{j=0}^{\infty} e^{- (a-\frac{\kappa'}{p}) j }  .
    \end{align*}
 Taking $C_{\eqref{eq-lem-Sum-RW-bound}}(a,p) = [ 3\, C_{\eqref{eq-L-x-j-p-moment}}(1)  C_{\eqref{eq-L-x-j-p-moment}}(p) \,   \sum_{j \geq 0}e^{- j (a-{\kappa'}/{p})    }  ]^{p}$, the desired result follows.

    We now prove the claim \eqref{eq-occupation-2}. Let $T^{(1)}_{x,j}:= \inf\{n \geq 1: x+S^{(1)}_{n} \in [j,j+1) \}$. Then $\mathtt{L}_{x,j} >0$ if and only if $T^{(1)}_{x,j} < \infty$. By applying the strong Markov property at the stopping time $T^{(1)}_{x,j} $ we obtain that 
    \begin{align}
      & \E [  \mathtt{L}_{x,j}  ^{p} ]  = \E \left[    \left( \sum_{n \geq 0} 1_{ \{ x+S_{   n}^{(1)} \in [j,j+1) \} }  \right)^{p} ; T^{(1)}_{x,j} < \infty    \right] = \E \left[  \left( \sum_{n \geq 0} 1_{ \{ S_{ T^{(1)}_{x,j} + n}^{(1)} \in [j,j+1) \} } \right)^{p} ;  T^{(1)}_{x,j} < \infty    \right] \nonumber\\
      &= \int \E \left[  \left( \sum_{n \geq 0} 1_{ \{ S_{ T^{(1)}_{x,j} + n}^{(1)} \in [j,j+1) \} } \right)^{p} \,  \middle| \,
      T^{(1)}_{x,j} < \infty,  S_{T^{(1)}_{x,j} }= j + y  \right] 
      \mathbf{P}(T^{(1)}_{x,j}<\infty , S_{T^{(1)}_{x,j} }-j \in \d y ) \nonumber\\
      & \leq  \sup_{y \in [0,1]}   \E\left[  \left( \sum_{n \geq 0} 1_{ \{ S_{n}^{(1)} \in [-y,1-y) \} } \right)^{p} \right]
      \mathbf{P}(T^{(1)}_{x,j}<\infty  ) \leq  C_{\eqref{eq-L-x-j-p-moment}}(p)  \, 
      \mathbf{P}( \mathtt{L}_{x,j} \geq 1 ), \label{eq-L-x-j-p-moment}
    \end{align}
   where  $C_{\eqref{eq-L-x-j-p-moment}}(p):=  \E\left[  \left( \sum_{n \geq 0} 1_{ \{ S_{n}^{(1)} \in [-1,1] \} } \right)^{p} \right]< \infty$. Indeed to see that $C_{\eqref{eq-L-x-j-p-moment}}$ is finite, let 
    \begin{equation*}
      F=\sup\{ n \geq 0: |S_{n}^{(1)}| \leq 1\}. 
    \end{equation*}
     Then $ C_{\eqref{eq-L-x-j-p-moment}}(p) 
    \leq \E\left[ (1+F) ^{p} \right]=\sum_{n} (n+1)^{p} \mathbf{P}(F=n ) \leq  \sum_{n} (n+1)^{p} \mathbf{P}(|S_{n}^{(1)}|\leq 1 ) < \infty$. Indeed,  since 
    $\mathbf{P}(|S_{n}^{(1)}|\leq 1 ) \leq  \E[ e^{t-t S_{n}^{(1)} } ] = e^{t+n\psi(1+t)}$ we can select $t>0$ such that $\psi(1+t) < 0$, ensuring the summability of the series.

    Now it remains to bound $\mathbf{P}( \mathtt{L}_{x,j} \geq 1 )$. Assume that $0\le j\le x$. Applying a union bound, together with \eqref{rwschange}, we obtain that
    \begin{align*}
      \P( \mathtt{L}_{x,j} \geq 1 )&  \leq \sum_{n \geq 0} \P( S^{(1)}_{n} + x \in  [j,j+1) ) \leq \sum_{n \geq 0} \E[ e^{\kappa' S^{(\kappa)}_{n}} ;  S^{(\kappa)}_{n} + x \in  [j,j+1) ] \\
      & \leq e^{-\kappa'(x-j-1)} \sum_{n \geq 0} \P \left(   - S^{(\kappa)}_{n }  \in  [x-j,x-j+1)    \right)  .
    \end{align*}
For any $a>0$, by applying the strong Markov property at the stopping time $T_{a}^{(\kappa)}:= \inf\{ n \geq 0: -S^{(\kappa)}_{n}  \in (a-1,a]  \}$ and mimicking the argument \eqref{eq-L-x-j-p-moment} we obtain that 
     \begin{equation}
    \sup_{a>0} \sum_{n \geq 0}   \P \left(     -S^{(\kappa)}_{n} \in (a-1,a]  \right)  
     \leq   \E \left(  \sum_{n= 0}^{\infty}  1_{ \{ -S^{(\kappa)}_{n} \in [-1,1] \} }  \right)    =  C_{\eqref{eq-L-x-j-p-moment}}(1) < \infty . \label{eq-occupation-time-a-a+1}
    \end{equation}  
    For $j >x$, since  $C_{\eqref{eq-L-x-j-p-moment}}(1)  \geq 1$, it holds trivially that  $\P( \mathtt{L}_{x,j} \geq 1 ) \leq 1 \leq C_{\eqref{eq-L-x-j-p-moment}}(1)  e^{-\kappa'(x-j)_+} $.   This completes the proof.
  \end{proof}

\begin{proof}[Proof of \eqref{eq-bound-W}] 
  For  notational convenience, we set $\bar{W}_{n,x}:= \sum_{|u|=n }e^{-V(u)} \ind{\mathbf{M}_{n}  \geq -x}$, and  $a' := a - 1$ for any real number $a$.  
 Applying the spinal decomposition   (Proposition \ref{BRWchangeofp}) we have 
   \begin{align}
    & \E[ \bar{W}_{n, x} ^{\kappa+\delta}    ] \leq  \E_{\mathbf{Q^*}} [ \left( \bar{W}_{n , x}  \right) ^{\kappa+\delta'} 1_{\{\mathbf{M}_{n} \geq -x\} }  ] \notag \\
    & \leq  \E_{\mathbf{Q^*}}\left[ \left( \sum_{k=1}^{n} \sum_{z \in \Omega(w_{k}) } e^{-V(z)} 1_{\{  V(z) \geq -x\}} \bar{W}^{(z)}_{n-k , x+V(z) }  +  e^{-V(w_{n})} 1_{\{  V(w_{n}) \geq -x\}}  \right)^{\kappa+\delta'}\right] . \label{eq-boundd-W}
   \end{align}

   \underline{\textit{Case 1.}}
  Consider first the case $\kappa+\delta'\in (0,1]$. Define  $\mathcal{B}_{n}:=\sigma( V(w_{k}), \{V(z): z \in \Omega(w_{k})\}, 1 \leq k \leq n )$, where $\Omega(u)$ represents the brothers of $u$. 
  Applying the  inequality $ \E[X^{p} | \mathcal{G}] \leq \E[X| \mathcal{G}]^{p}$, valid for all $p\in (0,1]$ and  $X  \geq 0$,  to \eqref{eq-boundd-W},  then using the branching property and the fact that $\E [W_{n-k}]=1 $,
  we obtain that 
   \begin{align*} 
    &  \E[ \bar{W}_{n, x} ^{\kappa+\delta}    ]  \leq \E_{\Q^*} \left[ \E_{\Q^*} \left( \sum_{k=1}^{n}   \sum_{z \in \Omega(w_{k}) } e^{- V(z)} \ind{V(z)\geq -x}  W^{(z)}_{n-k}  +  e^{- V(w_{n})} 1_{\{ V(w_{n}) \geq -x\}}  \mid \mathcal{B}_{n} \right)^{\kappa+\delta'}  \right]  \\
     &=\E_{\Q^*} \left[  \left(  \sum_{k=1}^{n}   \sum_{z \in \Omega(w_{k}) } e^{- V(z)} \ind{V(z)\geq -x}    +  e^{- V(w_{n})}1_{\{  V(w_{n}) \geq -x \}} \right)^{\kappa+\delta'} \right] \leq  e^{(\kappa+\delta')x} \E_{\Q^*} \left[  \Sigma_{n,x} ^{\kappa+\delta' } \right],
    \end{align*}
    where for notational convenience we set $  \Delta_{k}  :=   1+ \sum_{z \in \Omega(w_{k}) } e^{-  \Delta V(z)}  $ and 
    \begin{equation}\label{eq-simga-n-x-1}
    \Sigma_{n,x}:=  \sum_{k=0}^{n}e^{- [V(w_{k})+x]} 1_{\{ V(w_{k}) \geq -x\}} \Delta_{k+1} . 
    \end{equation} 
  From the inequality $(\sum_{i} x_{i})^{p} \leq \sum_{i} x_{i}^{p}$, valid for all $p\in (0,1]$ and  $x_i  \geq 0$, we obtain that 
    \begin{align}
      \E_{\Q^*} \left[  \Sigma_{n,x} ^{\kappa+\delta' } \right] 
     & \leq   \sum_{k=0}^{n}\E_{\Q^*} \left[  e^{- (\kappa+\delta') [V(w_{k})+x]} 1_{\{ V(w_{k}) \geq -x\}} \Delta_{k+1}^{\kappa+\delta'}   \right]   \nonumber \\
       & \leq {   \mathbf{E}_{\Q^*} [ \Delta_{1}^{\kappa+\delta'} ]}    \sum_{k=0}^{n}\E_{\Q^*} \left[  e^{-  (\kappa+\delta')[ V(w_{k})+x]} 1_{\{ V(w_{k}) \geq -x\}}    \right] .  
    \end{align}  
Above,  we used the branching property which implies that $\mathbf{E}_{\Q^*} [ \Delta_{k+1}^{\kappa+\delta'}  \mid \mathcal{B}_{k} ]= \mathbf{E}_{\Q^*} [ \Delta_{1}^{\kappa+\delta'}    ] $. On the one hand,     by Assumption \ref{cond4}, we have 
\begin{equation}  
  \mathbf{E}_{\Q^*} [ \Delta_{1}^{\kappa+\delta'} ] \leq    \mathbf{E} [ (1+W_{1}) ^{\kappa+\delta}]   <\infty. 
\end{equation} 
On the other hand, it follows from Lemma \ref{lem-Sum-RW-bound}   that 
\[
      \sum_{k=0}^{n}\E_{\Q^*} \left[  e^{-  (\kappa+\delta') [V(w_{k})+x] } 1_{\{ V(w_{k}) \geq -x\}}    \right]   \leq C_{\eqref{eq-lem-Sum-RW-bound}}(\kappa+\delta',1)  e^{- \kappa' x}
\]
Combining the previous bounds, we deduce that  for any $ x\ge 0$, and $n \geq 1$
    \begin{equation}\label{eq-kappa+delta<2}
     \E[ \bar{W}_{n, x} ^{\kappa+\delta}    ] 
         \leq  \mathbf{E} [ (1+W_{1}) ^{\kappa+\delta}] C_{\eqref{eq-lem-Sum-RW-bound}}(\kappa+\delta',1)  e^{ \delta x} , 
    \end{equation}
    and we can just take $C_{\eqref{eq-bound-W}}(\delta) = \mathbf{E} [ (1+W_{1}) ^{\kappa+\delta}] C_{\eqref{eq-lem-Sum-RW-bound}}(\kappa+\delta',1) $ in the case $\kappa+\delta' \in (0,1]$. 
 \vspace{5pt}

\underline{\textit{Case 2.}} Now consider the case $\kappa+ \delta' >1$. Assume first that $\delta \in (0,1) \cap (0,\delta_0)$.  Applying  Minkowski's inequality $\E[ (\sum_{i} X_{i}) ^{p} |\mathcal{G}]\leq   ( \sum_{i} \E[ X_{i} ^{p} |\mathcal{G} ]^{1/p}  ) ^{p} $  to  \eqref{eq-boundd-W},  
   \begin{align}
    \E[ \bar{W}_{n, x} ^{\kappa+\delta}    ]  &\leq \E_{\Q^*} \left( \E_{\Q^*}\left[ \left( \sum_{k=1}^{n} \sum_{z \in \Omega(w_{k}) } e^{-V(z)} 1_{\{ V(z) \geq -x \}} W^{(z)}_{n-k,x+V(z)}   +  e^{-V(w_{n})} 1_{\{ V(w_{n}) \geq -x\}}  \right)^{\kappa+\delta'} \mid  \mathcal{B}_{n}\right] \right)   \notag\\
     & \leq \E_{\Q^*} \left[  \left( \sum_{k=1}^{n} \sum_{z \in \Omega(w_{k}) } e^{-V(z)} 1_{\{ V(z) \geq -x \}}  
       \left( \E  [  W_{n-k,x+V(z)}   ^{\kappa+\delta'} ]  \right)^{\frac{1}{\kappa+\delta' }  } +  e^{-V(w_{n})}1_{\{ V(w_{n}) \geq -x \}}  \right)^{\kappa+\delta'} \right] \notag \\
     & \leq   \E\left[ W_{\infty}^{\kappa+\delta'} \right]   e^{(\kappa+\delta') x} \E_{\Q^*} \left[ \Sigma_{n,x}^{\kappa+\delta'} \right].
     \label{eq-Minkowski-middle-bound} 
   \end{align} 
   where $\Sigma_{n,x} $ is defined in \eqref{eq-simga-n-x-1}. 
Note that, since $\delta\in (0,1)$, it follows from  Proposition \ref{Dcvg} that  $1 \leq \E\left[ W_{\infty}^{\kappa+\delta'} \right]<\infty$.  
By use of Lemma \ref{lem-p-moment} (noting that  $\Delta_{k+1}$ is independent to $\mathcal{B}_{k}$ and  has the same law as  $\Delta_{1}$) we have 
   \begin{align} 
    \E_{\Q^*} \left[  \Sigma_{n,x} ^{\kappa+\delta' } \right]  
     &\leq  (K [\kappa+\delta'])^{\kappa+\delta'} \E_{\mathbf{Q}^{*}} [ \Delta_{1}^{\kappa+\delta'} ]     
     \, \E_{\Q^*} \left[ \left( \sum_{k=0}^{n} e^{ -[V(w_{k}) +x] }1_{\{V(w_{k}) \geq -x \} } \right)^{\kappa+\delta'} \right] \notag \\
     & \leq  (K [\kappa+\delta'])^{\kappa+\delta'}   \E [ (1+W_{1})^{\kappa+\delta} ]  C_{\eqref{eq-lem-Sum-RW-bound}}(1,\kappa+\delta')  e^{- \kappa' x} \label{eq-innitial-bound-1} 
   \end{align}  
where $K$ is the absolute constant in Lemma \ref{lem-p-moment}, and we used $\E_{\mathbf{Q}^{*}} [ \Delta_{1}^{\kappa+\delta'} ] \leq \E [ (1+W_{1})^{\kappa+\delta} ] $, and applied  Lemma \ref{lem-Sum-RW-bound} to $a=1$, $p=\kappa+\delta'> \kappa'$.   
Combining with the previous bound, and using the fact that $\E_{\mathbf{Q}^{*}} [ \Delta_{1}^{\kappa+\delta'} ] \leq \E [ (1+W_{1})^{\kappa+\delta} ] $ we get 
   \begin{equation}\label{eq-bound-W-k+delta>2}
    \E[ \bar{W}_{n,x} ^{\kappa+\delta}   ]     \leq   { [K (\kappa+\delta') ]^{\kappa+\delta'} \E [ (1+W_{1})^{\kappa+\delta} ]    \E [    W_{\infty}  ^{\kappa+\delta'} ] } C_{\eqref{eq-lem-Sum-RW-bound}}(1,\kappa+\delta')  \, e^{ \delta x} . 
   \end{equation}    
   This completes the proof of \eqref{eq-bound-W} in the case where $\delta \in (0,1)\cap (0,\delta_0]$
and $\kappa+\delta' >1$.

   We  proceed by induction if $\delta_0>1$. Now, suppose that $\delta \in (1,\delta_0]$ and assume inductively that \eqref{eq-bound-W} holds for  $\delta'$. Now we apply Rosenthal's inequality (Lemma \ref{lem-Rosenthal-inequality}) to \eqref{eq-boundd-W}, since  conditionally on $ \mathcal{B}_{n}$, $\{ \bar{W}^{(z)}_{n-k,x+V(z)} , z \in \Omega_{k}: 1 \leq k \leq n\} $ are independent random variables, we get 
   \begin{align}
    & \E_{\Q^*}\left[ \left( \sum_{k=1}^{n} \sum_{z \in \Omega(w_{k}) } e^{-V(z)} 1_{\{ V(z) \geq -x \}} \bar{W}^{(z)}_{n-k,x+V(z)}   +  e^{-V(w_{n})} 1_{\{ V(w_{n}) \geq -x\}}  \right)^{\kappa+\delta'} \mid  \mathcal{B}_{n}\right] \notag \\
     & \leq [K (\kappa+\delta')]^{\kappa+\delta'}  e^{(\kappa+ \delta')x}  \bigg\{ \bigg( \sum_{k=1}^{n} \sum_{z \in \Omega(w_{k}) } e^{-[V(z)+x]} 1_{\{ V(z) \geq -x \}} 
      +  e^{-[V(w_{n})+x]}1_{\{ V(w_{n}) \geq -x \}} \bigg)^{\kappa+\delta}  \label{eq-induction-bound-1}\\
     &  \quad  +    \sum_{k=1}^{n} \sum_{z \in \Omega(w_{k}) } e^{-(\kappa+\delta') [V(z)+x]} 1_{\{ V(z) \geq -x \}} \E[ \bar{W}_{n-k,x+V(z)}^{\kappa+\delta'} ]
     +  e^{- (\kappa+\delta') [V(w_{n})+x]}1_{\{ V(w_{n}) \geq -x \}}  \bigg\} \label{eq-induction-bound-2}
   \end{align} 
   On the one hand, the summation inside the parentheses in \eqref{eq-induction-bound-1} is bounded above by $ \Sigma_{n,x} $ as defined in \eqref{eq-simga-n-x-1}.  We have just already estimated $  \E_{\Q^*} \left[  \Sigma_{n,x} ^{\kappa+\delta' } \right]$ in \eqref{eq-innitial-bound-1}.   

   On the other hand,  
   by the induction hypothesis we have $\E[ \bar{W}_{n-k,x+V(z)}^{\kappa+\delta'} ] \leq C_{\eqref{eq-bound-W}}(\delta')    e^{\delta' (x+V(z)) }$. Define $ \Delta_{k}(\kappa):= 1+ \sum_{z \in \Omega(w_{k})} e^{- \kappa \Delta V(z)}$. Then  the summation in \eqref{eq-induction-bound-2}  is bounded from above by  
\begin{align*}
&   C_{\eqref{eq-bound-W}}(\delta')    \sum_{k=1}^{n} \sum_{z \in \Omega(w_{k}) } e^{- \kappa [V(z)+x]} 1_{\{ V(z) \geq -x \}} 
  +  e^{- \kappa [V(w_{n})+x]}1_{\{ V(w_{n}) \geq -x \}}   \\
  & \leq   C_{\eqref{eq-bound-W}}(\delta')    \sum_{k=0}^{n}   e^{- \kappa [V(w_{k})+x]} 1_{\{ V(w_{k}) \geq -x \}} \Delta_{k+1}(\kappa)  =: C_{\eqref{eq-bound-W}}(\delta')   \Sigma_{n,x}(\kappa).
\end{align*} 
Moreover by using the branching property and Lemma \ref{lem-Sum-RW-bound} we obtain that 
\begin{align*}
    \E_{\Q^*} \left[  \Sigma_{n,x}(\kappa) \right]  
    &\leq    \E_{\Q^*}\left[   \Delta_{1}(\kappa)  \right] \E_{\Q^*} \left[ \sum_{k=0}^{n}   e^{- \kappa [V(w_{k})+x]} 1_{\{ V(w_{k}) \geq -x \}}  \right] \leq  \E [ (1+W_{1})^{\kappa+\delta} ] C_{\eqref{eq-lem-Sum-RW-bound}}(\kappa,1) e^{-\kappa' x} . 
\end{align*} 
 Above, we used that  by Assumption \ref{cond4}, $\E_{\mathbf{Q}^{*}} [ \Delta_{1}(\kappa) ]  \leq  \E  [ W_{1}(1+\sum_{|u|=1}e^{-\kappa V(u)} ) ] \leq \E [ W_{1}(1+ W_{1}^{\kappa}) ] =1+ \E [  W_{1}^{\kappa+1} ] \leq \E [ (1+  W_{1})^{\kappa+\delta} ]<\infty$ (recall that now $\delta\geq 1$). In conclusion we obtain that 
 \begin{align}
 & \E[ \bar{W}_{n, x} ^{\kappa+\delta}    ] \leq [K (\kappa+\delta')]^{\kappa+\delta'} e^{(\kappa+\delta')x}\left\{ \E_{\Q^*} \left[  \Sigma_{n,x} ^{\kappa+\delta' } \right] + C_{\eqref{eq-bound-W}}(\delta')  \E_{\Q^*} \left[ \Sigma_{n,x}(\kappa)\right] \right\} \notag \\
  & \leq (K [\kappa+\delta'])^{2(\kappa+\delta')}   \E [ (1+W_{1})^{\kappa+\delta} ]  \left[ C_{\eqref{eq-lem-Sum-RW-bound}}(1,\kappa+\delta')  +C_{\eqref{eq-bound-W}}(\delta')  C_{\eqref{eq-lem-Sum-RW-bound}}(\kappa,1)   \right]e^{\delta x} . \label{eq-constant-induction}
 \end{align}
 This completes the proof. 
  \end{proof}
 
  \begin{proof}[Proof of \eqref{eq-desired-bound-D}] In this proof for convenience we define 
    \begin{equation}\label{eq-sup-D-M-x}
      E_{\eqref{eq-sup-D-M-x}}(x) :=  \sup_{k \geq 1} \E \left[   |   D_{k} |^{\kappa+\delta} 1_{\{ \mathbf{M}_k \geq -x \}}  \right] .
    \end{equation} 

  \underline{\textit{Case 1.}}
 Consider first the case where $\kappa+\delta'\in (0,1]$.  Applying Lemma \ref{lem-Bahr-Esseen} to the martingale $M_{n} = D_{n}$ and the event $ A_{n}=\{ \mathbf{M}_{n} \geq -x  \} $,  we get 
\[    E_{\eqref{eq-sup-D-M-x}}(x)  \leq  4 \sum_{n=0}^{\infty}  \E \left[   |  D_{n+1} - D_{n} |^{\kappa+\delta} 1_{\{ \mathbf{M}_{n} \geq -x \}}  \right] . \]
  Note that  
\[   D_{n+1}- D_{n} = \sum_{|z|= n } e^{- V(z)}  (  -[V(z)+\psi'(1) n] (W^{(z)}_{1} -1 ) + D^{(z)}_{1} )  .\] Conditionally on $ \mathcal{F}_{n}$,  $ \{   -[V(z)+\psi'(1) n] (W^{(z)}_{1} -1 ) + D^{(z)}_{1} : |z|=n\}$ are i.i.d. centered random variables. Thus, by applying the classical Bahr-Esseen inequality, since  $\kappa+\delta \in (1,2]$,  we have
 \begin{align}
    \E \left[   |  D_{n+1} - D_{n} |^{\kappa+\delta}  \mid \mathcal{F}_{n} \right] 
  & \leq
     4 \sum_{|z|= n } e^{- (\kappa+\delta)V(z)}  \left( | V(z)+\psi'(1) n|^{\kappa+\delta}   \E | W_{1} -1 |^{\kappa+\delta} +  \E | D_{1}|^{\kappa+\delta}   \right) \nonumber\\
     & \leq  C_{\eqref{eq-mardiff-D-1}} \sum_{|z|= n } e^{-(\kappa+\delta) V(z)}  \left( 1+ | V(z)+\psi'(1) n|^{\kappa+\delta} \right) ,  \label{eq-mardiff-D-1}
 \end{align}
 where $C_{\eqref{eq-mardiff-D-1}}:= 4 \max\{\E | W_{1} -1 |^{\kappa+\delta},  \E | D_{1}|^{\kappa+\delta}\} $. Combining the previous bound with   the many-to-one Lemma \ref{ManytoOne} yields that 
 \begin{align*}
  E_{\eqref{eq-sup-D-M-x}}(x) &\leq C_{\eqref{eq-mardiff-D-1}} \sum_{n\geq 1} \E \left[ \sum_{|z|=n}e^{-(\kappa+\delta) V(z)}  \left( 1+ | V(z)+\psi'(1) n|^{\kappa+\delta} \right) 1_{\{ \mathbf{M}_{n} \geq -x \} }   \right] \\
  & \leq  C_{\eqref{eq-mardiff-D-1}}  2^{\kappa+\delta} (1+|\psi'(1)|^{\kappa+\delta})  \sum_{n \geq 0}  \E \left[ e^{- \delta S^{(\kappa)}_{n}}  \left(  | S^{(\kappa)}_{n} |^{\kappa+\delta}+ n^{\kappa+\delta}  \right) 1_{\{ S^{(\kappa)}_{n} \geq -x \} }  \right]  .
 \end{align*}
 We will partition the series into three parts and bound each part separately.
Set $t=(\kappa-1)/2>0$  so that  $\psi(\kappa-t)<0$. By applying the Markov inequality, we have, for any $p\geq 1$, 
 \begin{align}
  E^{I}(p) &:= \sum_{n \geq 0} \E \left[   \left(   | S^{(\kappa)}_{n}|^{p}+ n^{p} \right) 1_{\{ S^{(\kappa)}_{n} \geq 0\} }  \right] 
     \leq \sum_{n \geq 0} \sum_{j \geq 0} \left[  (j+1)^{p} + n^{p}  \right] \P(S^{(\kappa)}_{n} \in [j,j+1]) \nonumber\\
  & \leq  \sum_{n \geq 0} \sum_{j \geq 0} \left[  (j+1)^{p} + n^{p}  \right] e^{-t j + \psi(\kappa-t)n } =: C_{\eqref{eq-boundd-1}}(p)< \infty .  \label{eq-boundd-1}
\end{align}   
Moreover, by use of  \eqref{eq-occupation-time-a-a+1} we obtain that  
 \begin{align*}
 E^{II}(x) &:=  \sum_{n \geq 0}  \E \left[ e^{- \delta S^{(\kappa)}_{n}}    | S^{(\kappa)}_{n} |^{\kappa+\delta}   1_{\{ 0 \leq -S^{(\kappa)}_{n} \leq x \} }  \right] \leq   \sum_{1\leq j \leq x}e^{\delta j}   j^{\kappa+\delta}  \sum_{n \geq 0} \P(-S^{(\kappa)}_{n} \in (j-1,j] ) \\
 &  \leq   C_{\eqref{eq-L-x-j-p-moment}}(1)  \sum_{1\leq j \leq x} e^{\delta j}  j^{\kappa+\delta}  \leq  \frac{  C_{\eqref{eq-L-x-j-p-moment}}(1) }{1-e^{-\delta}}  \,  e^{\delta x} (1+x)^{\kappa+\delta} .
 \end{align*}  
 Finally, let $a>0$ satisfying $  a  
    \psi(\kappa-t) = -3 t $ (recall that $t=(\kappa-1)/2$). Again by applying  the  Markov inequality and  \eqref{eq-occupation-time-a-a+1}, 
      \begin{align}
       & E^{III}(x)  := \sum_{1\leq j \leq x} e^{\delta j}  \sum_{n \geq 0} n^{\kappa+\delta} \P(-S^{(\kappa)}_{n} \in (j-1,j] ) \nonumber \\
        & \leq  \sum_{1\leq j \leq x} e^{\delta j}  \sum_{n \geq a x}  n^{\kappa+\delta}  e^{t j}  e^{  \psi(\kappa-t)n}  + \sum_{1\leq j \leq x} e^{\delta j}  [1+ax ]^{\kappa+\delta} \sum_{0\le n \leq ax}   \P(-S^{(\kappa)}_{n} \in (j-1,j] ) \nonumber \\
        & \leq \frac{  e^{\delta x} }{1-e^{-\delta}}    e^{-tx}  \sum_{n\geq 0 }
        (ax+n)^{\kappa+\delta} e^{\psi(\kappa-t)n}   +  C_{\eqref{eq-L-x-j-p-moment}}(1) \,\frac{  [1+ax]^{\kappa+\delta} e^{\delta x} }{1-e^{-\delta}}   \leq C_{\eqref{eq-boundd-3}} e^{\delta x} (1+x)^{\kappa+\delta} , \label{eq-boundd-3}
      \end{align} 
where   $C_{\eqref{eq-boundd-3}}:= \frac{1}{1-e^{-\delta}}[\sup_{x\geq 0}  e^{-tx}  \sum_{n\geq a x } (n+ax)^{\kappa+\delta}e^{\psi(\kappa-t)n} + C_{\eqref{eq-L-x-j-p-moment}}(1)  a^{\kappa+\delta}]$. 
 Combining all previous bound     we finally conclude that 
 \begin{align}
  E_{\eqref{eq-sup-D-M-x}}(x)  &\leq C_{\eqref{eq-mardiff-D-1}}  2^{\kappa+\delta} (1+|\psi'(1)|^{\kappa+\delta}) \left[   E^{I}(x)+ E^{II}(x)+ E^{III}(x) \right] \notag \\
  &\leq C_{\eqref{eq-mardiff-D-1}}  2^{\kappa+\delta} (1+|\psi'(1)|^{\kappa+\delta}) \left[    C_{\eqref{eq-boundd-1}}(p)+\frac{  C_{\eqref{eq-L-x-j-p-moment}}(1) }{1-e^{-\delta}} + C_{\eqref{eq-boundd-3}}  \right]  (1+x)^{\kappa+\delta} e^{\delta x} . \label{eq-D-bound-k+delta<2}
 \end{align}

 \underline{\textit{Case 2.}} 
 Consider now the case $\kappa+\delta>2$.   
 We divide the proof into two steps, establishing the following two assertions.  
 \begin{enumerate}[(i)]
   \item There exists constants $A\geq 1$ and  $C_{\eqref{eq-bound-badname-3}} $ both depending on  $\delta, \kappa$ and the BRW such that  \begin{equation*}
      \sum_{n \geq  Ax} \E[ |D_{n+1} -D_{n}| ^{\kappa+\delta} 1_{\{\mathbf{M}_{n} \geq -x\} } ]^{\frac{1}{\kappa+\delta}} \leq  C_{\eqref{eq-bound-badname-3}}  \quad \forall \ x \geq 0.
   \end{equation*}
   \item For any $\mathfrak{m}\ge 1$, there exists a constant $C_{\eqref{eq-boundd-badname-1}}$ depending on  $\delta, \kappa$ and the BRW  such that 
     \begin{equation*}
     \E[ | D_{\mathfrak{m}}| ^{\kappa+\delta} 1_{\{\mathbf{M}_{\mathfrak{m}} \geq -x\} } ] \leq C_{\eqref{eq-boundd-badname-1}} +  ( 2 K \max\{\mathfrak{m},x\})^{\kappa+\delta} C_{\eqref{eq-bound-W}}(\delta) e^{\delta x} 
   \quad  \forall \ x \geq 0.
   \end{equation*} 
 \end{enumerate} 
 Now by applying Minkowski’s inequality and the fact that $(x+y)^p \leq 2^{p}(x^p+y^p)$ for $p, x, y \geq 0$, we  obtain the desired bound:
 \begin{align}
  E_{\eqref{eq-sup-D-M-x}}(x) 
  &  \leq  \left( \sum_{n \geq Ax } \E[ |D_{n+1} -D_{n}| ^{\kappa+\delta} 1_{\{\mathbf{M}_{n} \geq -x\} } ]^{\frac{1}{\kappa+\delta}}  + \sup_{\mathfrak{m}\le Ax}\E[ | D_{\mathfrak{m}}| ^{\kappa+\delta} 1_{\{\mathbf{M}_{\mathfrak{m}} \geq -x\} } ]^{\frac{1}{\kappa+\delta}}  \right)^{\kappa+\delta}  \nonumber\\
  & \leq 2^{\kappa+\delta}\left( C_{\eqref{eq-bound-badname-3}}^{\kappa+\delta}  + C_{\eqref{eq-boundd-badname-1}} +  ( 2 K A)^{\kappa+\delta} C_{\eqref{eq-bound-W}} (1+x)^{\kappa+\delta} e^{\delta x} \right)  \label{eq-D-bound-k+delta>2}. 
 \end{align}

 We begin by proving assertion (ii).  By employing the change of measure given by  $\frac{\d \mathbf{Q} }{\d \P} |_{\mathcal{F}_{\mathfrak{m}}}= W_{\mathfrak{m}}$ we rewrite 
 \begin{equation}\label{eq-power-D-1}
   \E \left[ |D_{\mathfrak{m}}|^{\kappa+\delta} 1_{\{\mathbf{M}_{ \mathfrak{m}} \geq -x \}} \right] = \E_{\mathbf{Q}   } \left[ \left| \frac{ D_{\mathfrak{m}}}{  W_{\mathfrak{m}} } \right|^{\kappa+\delta} |W_{\mathfrak{m}}|^{\kappa+\delta'} 1_{\{\mathbf{M}_{ \mathfrak{m}} \geq -x \}} \right] .
 \end{equation}
 By the spine decomposition (Proposition \ref{BRWchangeofp}), we have
 \begin{align*}
   \frac{ D_{\mathfrak{m}}}{  W_{\mathfrak{m}} } &= -\sum_{|u|= \mathfrak{m} }\frac{e^{-V(u)}}{W_{\mathfrak{m}} }   \left( V(u)+ \psi'(1)|u| \right)   =  \mathbf{E}_{\mathbf{Q}^*  }  \left[  - \left( V(w_{\mathfrak{m}})+ \psi'(1)\mathfrak{m} \right)  \mid \mathcal{F}_{ \mathfrak{m}} \right].
 \end{align*}
 Substituting this conditional expectation into \eqref{eq-power-D-1} and applying Jensen’s inequality yields 
 \begin{align}
 \E \left[ |D_{\mathfrak{m}}|^{\kappa+\delta} 1_{\{\mathbf{M}_{\mathfrak{m}} \geq -x \}}  \right]
    &  \leq \E_{\mathbf{Q}^* } \left[    |V(w_{\mathfrak{m}}) +\psi'(1) \mathfrak{m}  |^{\kappa+\delta}   |W_{\mathfrak{m}}|^{\kappa+\delta'} 1_{\{\mathbf{M}_{\mathfrak{m}} \geq -x \}} \right] \notag \\
   & \leq  \E_{\mathbf{Q}^* } \left[     |V(w_{\mathfrak{m}}) +\psi'(1) \mathfrak{m}  |^{\kappa+\delta}  1_{ \{ | V(w_{\mathfrak{m}})| \geq K (\mathfrak{m} \vee x)\} }   |W_{\mathfrak{m}}|^{\kappa+\delta'}  1_{\{\mathbf{M}_{\mathfrak{m}} \geq -x \}}      \right] \notag\\
   &\qquad + \left( 2 K (\mathfrak{m}\vee x) \right)^{\kappa+\delta} \E_{\mathbf{Q} } \left[      |W_{\mathfrak{m}}|^{\kappa+\delta'}  1_{\{\mathbf{M}_{\mathfrak{m}} \geq -x \}}   \right], \label{eq-assertion-2-1}
 \end{align}
 where $K \geq  1+|\psi'(1)|$ is a large constant to be specified later.  By applying inequality \eqref{eq-bound-W} to the last term  we have 
 \begin{equation}\label{eq-assertion-2-2}
  \left( 2 K (\mathfrak{m}\vee x) \right)^{\kappa+\delta} \E_{\mathbf{Q} } \left[      |W_{\mathfrak{m}}|^{\kappa+\delta'}  1_{\{\mathbf{M}_{\mathfrak{m}} \geq -x \}}      \right]  \leq   \left( 2 K (\mathfrak{m}\vee x) \right)^{\kappa+\delta} C_{\eqref{eq-bound-W}}(\delta) e^{\delta x}.
 \end{equation}   
 We now show that the first term is negligible. 
 {Set $p_0=p_0(\delta)>1$ such that $1+ (\kappa+\delta')p_0< \kappa+\delta_0$. Such $p_0$ exists since $\delta<\delta_0$.} By using   H\"{o}lder's inequality with   $q_0^{-1}:=1-p_0^{-1}$ we get that 
 \begin{align*}
   & \E_{\mathbf{Q}^* } \left[    |V(w_{\mathfrak{m}}) +\psi'(1) \mathfrak{m}  |^{\kappa+\delta} 1_{ \{ | V(w_{\mathfrak{m}})| \geq K (\mathfrak{m}\vee x) \} }   |W_{\mathfrak{m}}|^{\kappa+\delta'}  1_{\{\mathbf{M}_{\mathfrak{m}} \geq -x \}}      \right]   \\
   & \leq (1+ |\psi'(1)|)^{\kappa+\delta} \E_{\mathbf{Q}^* } \left[      ( |V(w_{\mathfrak{m}})| + \mathfrak{m}  )^{(\kappa+\delta)q_0} 1_{ \{  V(w_{\mathfrak{m}}) \geq K (\mathfrak{m}\vee x) \} } \right]^{\frac{1}{q_0}}  
    \E_{\mathbf{Q}^* } \left[  |W_{\mathfrak{m}}|^{(\kappa+\delta')p_0}  1_{\{\mathbf{M}_{\mathfrak{m}} \geq -x \}}      \right]^{\frac{1}{p_0}} \\
    &=  (1+ |\psi'(1)|)^{\kappa+\delta}  \E \left[  (| S^{(1)}_{\mathfrak{m}}|  + \mathfrak{m}   )^{(\kappa+\delta)q_0} 1_{ \{  S^{(1)}_{\mathfrak{m}}  \geq K( \mathfrak{m}\vee x) \} } \right]^{\frac{1}{q_0}}  
    \E \left[  |W_{\mathfrak{m}}|^{1+(\kappa+\delta')p_0}  1_{\{\mathbf{M}_{\mathfrak{m}} \geq -x \}}      \right]^{\frac{1}{p_0}} .
 \end{align*} 
Since we already choose $1+ (\kappa+\delta')p_0< \kappa+\delta_0$,  it follows again from \eqref{eq-bound-W} that   
 \[  \E \left[  |W_{\mathfrak{m}}|^{1+(\kappa+\delta')p_0}  1_{\{\mathbf{M}_{\mathfrak{m}} \geq -x \}}  \right]^{\frac{1}{p_0}}   \leq C_{ \eqref{eq-bound-W}}( p_{1}(\delta) ) e^{ (\kappa+\delta) (\mathfrak{m}\vee x) } . \] 
 where $ p_{1}(\delta) := 1+ (\kappa+\delta')p_0(\delta) - \kappa \leq (\kappa+ \delta)p_{0}(\delta)$. 
  Moreover according to Assumption \ref{cond4}, we can choose $t \in(0,1)$   with  $\psi(1-t) \in (0,\infty)$ 
  and then set   $K = \frac{1}{t}(\psi(1-t)+(\kappa+\delta)+1)$. Then by using of Markov's inequality we obtain that 
  \begin{align*}
     \E \left[  \left(  | S^{(1)}_{\mathfrak{m}} | +  \mathfrak{m}  \right)^{(\kappa+\delta)q_0} 1_{ \{S^{(1)}_{\mathfrak{m}} \geq K (\mathfrak{m} \vee x) \} } \right] 
    & \leq \sum_{l \geq K (\mathfrak{m}\vee x)}  (l + 1+ \mathfrak{m} )^{(\kappa+\delta)q_0} \P( S^{(1)}_{\mathfrak{m}} \in [l,l+1] )  \\ 
       & \leq      \sum_{l \geq K( \mathfrak{m}\vee x)} e^{-t l}  (l+ 1+\mathfrak{m} )^{(\kappa+\delta)q_0} e^{\psi(1-t) \mathfrak{m}   }  .
    \end{align*}
 Making a change of variable $l=K( \mathfrak{m}\vee x)+ \ell $, we see that the summation is bounded from above by  
\[
 (K+1)^{(\kappa+\delta)q_0}(\mathfrak{m}\vee x)^{(\kappa+\delta)q_0} e^{-(Kt-\psi(1-t))( \mathfrak{m} \vee x)  }   \sum_{\ell \geq 0} e^{-t \ell}  (\ell+1)^{(\kappa+\delta)q_0} . \] 
 Let $ C_{\eqref{eq-boundd-badname-1}}:=C_{ \eqref{eq-bound-W}}(p_{1}(\delta))  
 (K+1)^{(\kappa+\delta)q_0}\sup_{y \geq 1} y^{(\kappa+\delta)q_0} e^{-y}   \sum_{l \geq 0} e^{-t l}  (l+1)^{(\kappa+\delta)q_0}$. Integrating the preceding bounds, by our choice of $K$, we conclude that  
    \begin{equation}
       \E_{\mathbf{Q}^* } \left[    |V(w_{\mathfrak{m}}) +\psi'(1) \mathfrak{m}  |^{\kappa+\delta} 1_{ \{ | V(w_{\mathfrak{m}})| \geq K( \mathfrak{m} \vee x) \} }   |W_{\mathfrak{m}}|^{\kappa'+\delta}  1_{\{\mathbf{M}_{\mathfrak{m}} \geq -x \}}      \right] \leq C_{\eqref{eq-boundd-badname-1}},  \label{eq-boundd-badname-1}
    \end{equation}
Our desired assertion (ii) then follows from \eqref{eq-assertion-2-1}, \eqref{eq-assertion-2-2} and \eqref{eq-boundd-badname-1}.
\vspace{5pt}

 It remains to prove assertion (i). 
 We once again make use of the decomposition
\[   D_{n+1}- D_{n} = \sum_{|z|= n } e^{- V(z)}  \left[  -[V(z)+\psi'(1) n] (W^{(z)}_{1} -1 ) + D^{(z)}_{1} \right] ; \] 
along with the fact that,  conditionally on $ \mathcal{F}_{n}$,  $ \left\{   -[V(z)+\psi'(1) n] (W^{(z)}_{1} -1 ) + D^{(z)}_{1} \right\}$ are independent centered random variables. Applying the Marcinkiewicz-Zygmund inequality \eqref{sumindep}, we see that $ \E[ |D_{n+1}- D_{n}|^{\kappa+\delta} \mid \mathcal{F}_{n} ]$ is bounded from above by 
\[ C_{\eqref{sumindep}}(\kappa+\delta) \E \left[  \left( \sum_{|z|= n } e^{- 2V(z)}  \left[  -[V(z)+\psi'(1) n] (W^{(z)}_{1} -1 ) + D^{(z)}_{1} \right]^{2} \right)^{\frac{\kappa+\delta}{2}} \mid \mathcal{F}_{n} \right] \]
By using the inequalities     $(\sum_{i} a_{i} x_{i} )^{p} \leq (\sum_{i} a_{i}  )^{p-1}  (\sum_{i} a_{i} x_{i}^{p} )  $ and  $\sum_{i} x_{i}^{2} \leq (\sum_{i} x_{i})^{2}$ for $x_{i} \geq 0$, valid for $a_{i}, x_{i} \geq 0$ and $p\geq 1$, we can further bound the above by 
  \begin{align} 
  & C_{\eqref{sumindep}}(\kappa+\delta)  \left( \sum_{|z|= n } e^{- 2V(z)} \right)^{\frac{\kappa+\delta}{2}-1} \E \left[    \sum_{|z|= n } e^{- 2 V(z)}  \Big| [V(z)+\psi'(1) n] (W^{(z)}_{1} -1 ) - D^{(z)}_{1} \Big|^{\kappa+\delta}   \mid \mathcal{F}_{n} \right] \nonumber\\
   & \leq   C_{\eqref{eq-bound-badname-2}}  \left( \sum_{|z|= n } e^{- V(z)} \right)^{\kappa+\delta-2}   \sum_{|z|= n } e^{- 2 V(z)}  \left( |V(z)+\psi'(1) n|  + 1 \right)^{\kappa+\delta}   \label{eq-bound-badname-2}
  \end{align}
 where { $C_{\eqref{eq-bound-badname-2}} :=  C_{\eqref{sumindep}}(\kappa+\delta) 2^{\kappa+\delta} (\E[ |W_{1}-1|^{\kappa+\delta}] + \E[|D_1|^{\kappa+\delta}]) $} and we used the branching property in the last inequality.

It follows from the upper bound above of $\E[ |D_{n+1} -D_{n}| ^{\kappa+\delta} \mid \mathcal{F}_n ]$ and the spine decomposition (Proposition \ref{BRWchangeofp}) that $\frac{1}{C_{\eqref{eq-bound-badname-2}}} \E[ |D_{n+1} -D_{n}| ^{\kappa+\delta} 1_{\{\mathbf{M}_{n} \geq -x\} } ] $ is bounded from above by 
\begin{align*}
  &  \E \left[    W_{n}^{\kappa+\delta-2}  \sum_{|z|= n } e^{- 2 V(z)}  \left( |V(z)+\psi'(1) n|  + 1 \right)^{\kappa+\delta}   1_{\{\mathbf{M}_{n} \geq -x\} } \right] \\
  & =\E_{\mathbf{Q}^*} \left[   e^{-  V(w_{n})}  \left( |V(w_{n})+\psi'(1) n|  + 1 \right)^{\kappa+\delta}    W_{n}^{\kappa+\delta-2} 1_{\{\mathbf{M}_{n} \geq -x\} } \right] .\\
\end{align*}
Now by use of  Holder inequality with exponent $p=\kappa+\delta'$ and $q=\frac{\kappa+\delta'}{\kappa'+\delta'}$,  we can further bound the above by 
 \begin{align*}  
  &  \E_{\mathbf{Q}^*}  \left[   e^{- (\kappa+\delta') V(w_{n})}  \left( |V(w_{n})+\psi'(1) n|  + 1 \right)^{(\kappa+\delta)(\kappa+\delta')} 1_{\{V(w_{n}) \geq -x\}} \right] ^{\frac{1}{\kappa+\delta'}}
   \E_{\mathbf{Q}^*} \left[    W_{n}^{\kappa+\delta'} 1_{\{\mathbf{M}_{n} \geq -x\} } \right]^{ \frac{\kappa'+\delta'}{\kappa+\delta'}} \\
   &=   \E \left[   e^{- \delta S^{(\kappa)}_{n}}  \left( |S^{(\kappa)}_{n}+\psi'(1) n|  + 1 \right)^{(\kappa+\delta)(\kappa+\delta')} 1_{\{S^{(\kappa)}_{n} \geq -x\}} \right]   ^{\frac{1}{\kappa+\delta'}}
    \E \left[    W_{n}^{\kappa+\delta} 1_{\{\mathbf{M}_{n} \geq -x\} } \right]  ^{ \frac{\kappa'+\delta'}{\kappa+\delta'}},
 \end{align*}
 where  we used  that  $\{V(w_{n}), \mathbf{Q}^* \} \overset{\textrm{law}}{=}\{S^{(1)}_{n}, \P \}$ and $\E[ e^{-\kappa' S^{(1)}_{n}} f(S^{(1)}_{n}) ]= \E[f(S^{(\kappa)}_{n})]$.  
 Again, thanks to \eqref{eq-bound-W}, there holds $\E \left[    W_{n}^{\kappa+\delta} 1_{\{\mathbf{M}_{n} \geq -x\} } \right] \leq C_{\eqref{eq-bound-W}} (\delta) e^{\delta x}$. Thus, denoting $p_\delta:=(\kappa+\delta)(\kappa'+\delta)$, we obtain that  $ \sum_{n \geq Ax} \E[ |D_{n+1} -D_{n}| ^{\kappa+\delta} 1_{\{\mathbf{M}_{n} \geq -x\} } ]^{\frac{1}{\kappa+\delta}}$ is bounded from above by 
 \[{ C_{\eqref{eq-bound-badname-2}} C_{\eqref{eq-bound-W}}(\delta)  } 
  e ^{  \frac{\kappa'+\delta'}{\kappa'+\delta} \frac{\delta}{\kappa+\delta} x } \, \sum_{n \geq Ax}
    \E \left[   e^{- \delta S^{(\kappa)}_{n}}  \left( |S^{(\kappa)}_{n}|+ n\right)^{p_\delta} 1_{\{S^{(\kappa)}_{n} \geq -x\}} \right]   ^{\frac{1}{p_{\delta}}} .\] 
 Set $t= \frac{\kappa-1}{2}>0$ so that $\psi(\kappa-t)<0$. For any $p \geq 1$ by use of  the inequality $(\sum_{i} x_{i})^{1/p} \leq \sum_{i} (x_{i})^{1/p} $ for $x_{i} \geq 0$  and   Markov's inequality, it follows that 
   \begin{align*}
    & \sum_{n \geq Ax}  \E \left[   e^{- \delta S^{(\kappa)}_{n}}  \left( |S^{(\kappa)}_{n}|+ n   \right)^{p}  1_{\{S^{(\kappa)}_{n} 
    \geq -x\}} \right] ^{\frac{1}{p}} 
     \leq e^{\frac{\delta}{p} x}  \sum_{n \geq Ax}  \E   \left[   \left( |S^{(\kappa)}_{n}|+ n   \right)^{p}  1_{\{S^{(\kappa)}_{n} \geq -x\}}   \right] ^{\frac{1}{p}} \\
    & \leq e^{\frac{\delta}{p} x}  \sum_{n \geq Ax} \left( \sum_{j \geq 0} (n+j) \P(S^{(\kappa)}_{n} \in [j,j+1] )^{\frac{1}{p}}  + \sum_{0 \leq j \leq x} (n+x)  \P(-S^{(\kappa)}_{n} \in [j,j+1] )^{\frac{1}{p}} \right) \\
    & \leq e^{\frac{\delta}{p} x} \left( \sum_{n \geq Ax}  \sum_{j \geq 0} (n+j) e^{ \frac{1}{p}[  \psi(\kappa-t)n - t j ] } + \sum_{n \geq Ax}  \sum_{0 \leq j \leq x} (n+j)  e^{ \frac{1}{p}[  \psi(\kappa-t)n + t j ] } \right).  
     \end{align*}  
     Making change of variable $n- Ax + l$, the above expression can be further upper bounded by
\begin{align*}
  & e^{\frac{\delta}{p} x+ \frac{\psi(\kappa-t)}{p}Ax}  \left(  \sum_{l \geq 0}  \sum_{j \geq 0} (l+j+Ax) e^{ \frac{1}{p}[  \psi(\kappa-t)l - t j ] } + \sum_{l \geq 0}  (1+x)   [ l+(A+1)x ]   e^{ \frac{1}{p}[  \psi(\kappa-t)l + t x ] }   \right)  \\
  &\leq    
  2 e^{\frac{\delta}{p} x+ \frac{\psi(\kappa-t)}{p}Ax} (1+2Ax)^2 e^{\frac{t}{p} x} \sum_{l \geq 0}  \sum_{j \geq 0} (l+j+1) e^{ \frac{1}{p}[  \psi(\kappa-t)l - t j ] }. 
\end{align*}
Choose $A>0$ to be  sufficiently large so that $   \frac{1}{p_{\delta}}[  \psi(\kappa-t) A + t +\delta ] + \frac{\kappa'+\delta'}{\kappa'+\delta} \frac{\delta}{\kappa+\delta} \leq -1 $. Then we get 
\begin{align}
  & \sum_{n \geq Ax} \E[ |D_{n+1} -D_{n}| ^{\kappa+\delta} 1_{\{\mathbf{M}_{n} \geq -x\} } ]^{\frac{1}{\kappa+\delta}}  \nonumber \\
   & \leq  2  C_{\eqref{eq-bound-badname-2}} C_{\eqref{eq-bound-W}}    \sum_{l \geq 0}  \sum_{j \geq 0} (l+j+1) e^{ \frac{1}{p_{\delta}}[  \psi(\kappa-t)l - t j ] } \,\sup_{x \geq 0}  (1+2Ax)^2  e ^{-x } =: C_{\eqref{eq-bound-badname-3}} < \infty ,  \label{eq-bound-badname-3}
\end{align} 
from which the assertion (ii) follows. We now complete the proof.    
\end{proof}

\subsection{Tail probability of the global minimum: proof of Lemma \ref{BRWroughbd}}

\begin{proof}[Proof of Lemma \ref{BRWroughbd}]
  The upper bound is relatively easy. Observe that $\{\tM\leq -x\}$ implies that $\{\sum_{u\in\T}\ind{V(u)\leq -x<\min_{\rho<v<u}V(v)}\geq1\}$.Applying Markov’s inequality and then the many-to-one Lemma \ref{ManytoOne}, we obtain that 
  \begin{align*}
  \P(\tM\leq -x) &\leq \E\left[\sum_{k\geq1}\sum_{|u|=k}\ind{V(u)\leq -x<\min_{\rho<v<u}V(v)}\right]=\sum_{k\geq1}\E\left[e^{\kappa S^{(\kappa)}_k}; S^{(\kappa)}_k\leq -x < \min_{1\leq i\leq k-1}S^{(\kappa)}_i\right]\\
  &\leq  e^{-\kappa x} \sum_{k\geq 1}\P(S^{(\kappa)}_k\leq -x < \min_{1\leq i\leq k-1}S^{(\kappa)}_i) =e^{-\kappa x}.
  \end{align*}
  
  Fix $x \geq 0$. For the lower bound, let us introduce the following events for any $u\in\T$:
  \[
  E_u:= \left\{ V(u)\in I(x), V(u)< \min_{\rho\leq z<u}V(z) \right\} , \, \text{ and }\, F^L_u:=\left\{\sum_{k=1}^{|u|}\sum_{z\in\Omega(u_k)}e^{-\kappa(V(z)+x)}\leq L\right\}, \forall L\geq1 .
  \]
  Define 
  \[
  N(x):=\sum_{u\in\T}\ind{E_u}, \text{ and }  N_L(x):=\sum_{u\in\T}\ind{E_u}\ind{F^L_u}.
  \]
  Then it follows from the Paley-Zygmund inequality that 
  \begin{equation}\label{PaleyZygmund}
  \P(\tM\leq -x)\geq \P(N_L(x)\geq1)\geq \frac{\E[N_L(x)]^2}{\E[N_L(x)^2]}.
  \end{equation}
  We first  estimate the moments of $N_L(x)$. Note that by many-to-one Lemma \ref{ManytoOne},
  \begin{align*}
  \E[N(x)]
  =\sum_{n\geq1}\E\left[\sum_{|u|=n}\ind{V(u)<  \min\limits_{\rho\leq v<u}V(v), V(u)\in I(x)}\right]   =\sum_{n\geq1}\E\left[e^{\kappa S^{(\kappa)}_n}\ind{S^{(\kappa)}_n<\mS^{(\kappa)}_{[0,n-1]}, S^{(\kappa)}_n\in I(x)}\right] .
  \end{align*}
Using the notation in Section \ref{Renewaltheory},  we obtain that 
\[  e^{-\kappa x-\kappa}U^{(\kappa),-}_s([x,x+1))\leq \E[N(x)]\leq e^{-\kappa x}U^{(\kappa),-}_s([x,x+1)).\]
  Consequently, by \eqref{eq-renewal-bound}, there exists constants $0<c_4<c_5<\infty$ such that for any $x\geq0$,
  \begin{equation}\label{meanNx}
  c_4 e^{-\kappa x} \leq \E[N(x)]\leq c_5 e^{-\kappa x}.
  \end{equation}

Next we will bound the expectation of  $N_L^{c}(x):= N(x)-N_{L}(x)= \sum_{u\in\T}\ind{E_u}\ind{(F^L_u)^c}$.  In fact, by   part (ii) of Proposition \ref{BRWchangeofp}  for $\Q^{\kappa,*}$, we can rewrite  $\E\left[ N_L^{c}(x) \right]=\E\left[\sum_{u\in\T}\ind{E_u}\ind{(F^L_u)^c}\right]$ as 
  \begin{align*}
   \sum_{n\geq1}  \E_{\mathbf{Q}^{\kappa,*}}
  \left[e^{\kappa V(w_n)}\ind{V(w_n)< \underline{V}(w_{[0,n-1]}), V(w_n)\in I(x)}\ind{ (F^L_{w_{n}})^{c} }\right] .
  \end{align*}
  Take an absolute constant $c>0$ such that $\sum_{k=1}^{\infty}\frac{e c}{k^2}\leq 1$. Since $L\geq \sum_{k=1}^n\frac{e c L}{ k^2}=\sum_{k=1}^n\frac{e c L}{(n-k+1)^2}$, on the event $\{ (F^L_{w_{n}})^{c} \}$ there must exists some $k$ such that $\sum_{u\in\Omega(w_k)}e^{-\kappa(V(u)+x)}>\frac{e c L}{(n-k+1)^2} $. If in addition $V(w_{n}) \in I(x)$,  this implies that $    e^{-\kappa [V(w_{k-1})-V(w_{n})] }\Delta_k(\kappa) \geq \frac{cL  }{(n-k+1)^2}$, where 
  $\Delta_k^{(\kappa)}:=1+\sum_{u\in\Omega(w_k)}e^{-\kappa \Delta V(u)}$.  
  Therefore, $\E\left[ N_L^{c}(x) \right]$ is bounded from above by 
  \begin{align*}
   e^{-\kappa x} \sum_{n=1}^\infty\sum_{k=1}^n  \mathbf{Q}^{\kappa,*} \left(  \underline{V}(w_{[0,n-1]}) > V(w_n)\in I(x);   e^{-\kappa [V(w_{k-1})-V(w_{n})] }\Delta_k(\kappa) \,\geq \frac{cL  }{(n-k+1)^2}\right).
  \end{align*} 
  From the fact that $(\Delta V(w_k), \Delta_k(\kappa))_{1\leq k\leq n}$ have the same law as $(\Delta V(w_{n-k+1}), \Delta_{n-k+1}{(\kappa)})_{1\leq k\leq n}$, one sees that $(V(w_k), \Delta_k{(\kappa)})_{1\leq k\leq n}$ and $(V(w_n)-V(w_{n-k}), \Delta_{n+1-k}{(\kappa)})_{1\leq k\leq n}$ have the same law. As a consequence, by making this time-reversal, we can rewrite the series above as
  \begin{align*}
  &  e^{-\kappa x} \sum_{n=1}^{\infty} \sum_{k=1}^n \mathbf{Q}^{\kappa,*} \left( \overline{V}(w_{[1,n]})<0, V(w_n)\in I(x); _{n-k+1} e^{\kappa V(w_{n-k+1}) }\Delta_{n-k+1}(\kappa)\geq\frac{c L }{ (n-k+1)^2}\right) \\
  &=   e^{-\kappa x}  \sum_{n=1}^\infty\sum_{k=1}^n \mathbf{Q}^{\kappa,*}\left( \overline{V}(w_{[1,n]})<0, V(w_n)\in I(x); \Delta_k{(\kappa)} e^{\kappa V(w_k)}\geq \frac{c L}{ k^2}\right)\\
  & =e^{-\kappa x} \sum_{k=1}^\infty\E_{\mathbf{Q}^{\kappa,*}}\bigg[
    \mathbf{1}_{\big\{  \overline{V}(w_{[1,k]})<0, -V(w_k)\leq \frac{\ln (\Delta_k{(\kappa)} / cL) + 2 \ln k}{\kappa} \big\}} \sum_{n=k}^\infty\P_{V(w_k)}\left( \MS^{(\kappa)}_{[1,n-k]}<0, S^{(\kappa)}_{n-k}\in I(x)\right) \bigg],
  \end{align*}
  where the last equality follows from the Markov property at $w_k$. Note that $\{V(w_i); 1\leq i\leq j-1\}$ which is distributed as $(S^{(\kappa)}_i; 1\leq i\leq j-1)$, is independent of $(\Delta V(w_j),\Delta_j{(\kappa)})$. Let us introduce a new couple $(\zeta, \Delta^{(\kappa)})$ which under $\P$ is distributed as $(\Delta V(w_1),\Delta_1{(\kappa)})$ under $\Q^{\kappa,*}$ and is independent of the random walk $(S^{(\kappa)}_k)$. It follows from \eqref{RWeSbd} that 
  \begin{align*} 
  \E\left[ N_L^{c}(x) \right]  & \leq c_0 e^{-\kappa x} \sum_{k=1}^\infty\E_{\mathbf{Q}^{\kappa,*}}\bigg[
    \mathbf{1}_{\big\{  \overline{V}(w_{[1,k]})<0, -V(w_k)\leq \frac{\ln (\Delta_k{(\kappa)} / cL) + 2 \ln k}{\kappa} \big\}} \left( 1- V(w_k) \right) \bigg], \\
    & \leq c_6 e^{-\kappa x} \sum_{k=1}^\infty\E\left[\ind{\MS^{(\kappa)}_{[1,k-1]}<0, -S^{(\kappa)}_{k-1}\leq \zeta+\frac{ \ln (\Delta^{(\kappa)} / cL) + 2 \ln k }{\kappa}}(1+\ln k) \left( 1 +\ln_+ \frac{\Delta^{(\kappa)}}{cL} \right)  \right].
  \end{align*}
  Observe that for $\lambda \in(0,\kappa-1)$, we have $ \ln\E[e^{\lambda S_1^{(\kappa)}}]= \psi( \kappa-\lambda)< 0$. Thus, for any $0<b$, 
  \begin{align*}
    \P\left(\MS^{(\kappa)}_{[1,\ell]}<0,  -S^{(\kappa)}_\ell \leq b\right)\leq   \P\left( 0 <  -S^{(\kappa)}_\ell \leq b\right) \leq  e^{\lambda b}\E\left[e^{\lambda S_k^{(\kappa)}}\right] =  e^{\lambda b + \psi( \kappa-\lambda) k}.
  \end{align*} 
  It hence follows from independence that
  \begin{align*} 
  \E\left[ N_L^{c}(x) \right] & \leq  c_7 e^{-\kappa x} \sum_{k=0}^\infty (1+\ln_+ k)(1+k)^{2\lambda/\kappa} e^{ \psi( \kappa-\lambda) k}\E\left[\left(  1+\ln_+ \frac{\Delta^{(\kappa)}}{cL} \right)  \left(\frac{\Delta^{(\kappa)}}{cL}\right)^{\lambda/\kappa}e^{\lambda \zeta}\right]\\
  &\leq  c_8 e^{-\kappa x}\E\left[ \left(  1+\ln_+ \frac{\Delta^{(\kappa)}}{cL} \right)  \left(\frac{\Delta^{(\kappa)}}{cL}\right)^{\lambda/\kappa}e^{\lambda \zeta}\right]. 
  \end{align*} 
  We claim that for $\lambda \in(0,\kappa-1)$, 
  \begin{align}
    E_{\eqref{eq-D-zeta-moment}}&:=\E\left( [1+\ln_+\Delta^{(\kappa)}][\Delta^{(\kappa)}]^{\lambda/\kappa}e^{\lambda \zeta}\right) \notag\\
    &=\E_{\Q^{\kappa,*}}\left[e^{\lambda V(w_1)}\left(\sum_{z\in\Omega(w_1)}e^{-\kappa V(z)}\right)^{\lambda/\kappa}\left(1+\ln_+\left(\sum_{z\in\Omega(w_1)}e^{-\kappa V(z)}\right)\right)\right]  < \infty  \label{eq-D-zeta-moment}
  \end{align}
   As a result, $\E\left[ N_L^{c}(x) \right]=o_L(1)e^{-\kappa x}$ and hence for $L\gg 1$, $x\geq1$, 
  \begin{equation}\label{eq-1st-N-L-x}
    c_9 e^{-\kappa x}\leq \E[N_L(x)]\leq c_{10} e^{-\kappa x}.
  \end{equation} 
  Indeed,  observe that by Proposition \ref{BRWchangeofp} 
  \begin{align*} 
  E_{\eqref{eq-D-zeta-moment}} &\leq \E\left[ \left( \sum_{|u|=1}e^{-(\kappa-\lambda)V(u)} \right) \left( \sum_{|u|=1}e^{-\kappa V(u)} \right)^{\lambda/\kappa}\left[1+\ln_+\left(  \sum_{|u|=1}e^{-\kappa V(u)} \right)  \right]\right]\\
  &\leq  \E\left[\left(\sum_{|u|=1}e^{-V(u)}\right)^{\kappa-\lambda}\left(\sum_{|u|=1}e^{-  V(u)}\right)^{\lambda }\left(1+\ln_+\left(\sum_{|u|=1}e^{-\kappa V(u)}\right)\right)\right] \\
  &\leq \E\left[\left(\sum_{|u|=1}e^{-V(u)}\right)^{\kappa}\right]^{(\kappa-\lambda)/\kappa}\E\left[\left(\sum_{|u|=1}e^{-  V(u)}\right)^{\kappa}\left(1+\ln_+\left(\sum_{|u|=1}e^{-\kappa V(u)}\right)\right)^{\kappa/\lambda}\right]^{\lambda/\kappa}
  \end{align*}
  where the second inequality holds as $\kappa>\kappa-\lambda>1$; and the third one follows from H\"{o}lder's inequality. From Assumption \ref{cond4}, we conclude that $  E_{\eqref{eq-D-zeta-moment}} < \infty$.

  For the second moment, we compute that 
  \begin{align*}
   &\E[N_L(x)^2] =\E[N_L(x)]+\E\bigg[\sum_{\substack{u,v\in\T; \\ u\neq v; u\wedge v\in\{u,v\}}}\ind{E_u\cap F_u^L}\ind{E_v\cap F_v^L}\bigg]
   +
   \E\bigg[\sum_{\substack{u,v\in\T; \\ u\neq v; u\wedge v\notin\{u,v\}}}\ind{E_u\cap F_u^L}\ind{E_v\cap F_v^L}\bigg] \\
   & \leq \E[N_L(x)]+ 2 \E\bigg[\sum_{u\in\T}\ind{E_u\cap F_u^L}\left(\sum_{v: v>u}\ind{E_v}\right)\bigg]+ \E\left[\sum_{u\in\T}\ind{E_u\cap F_u^L}\left(\sum_{k=1}^{|u|}\sum_{z\in\Omega(u_k)}\sum_{v: v\geq z}\ind{E_v}\right)\right] . 
  \end{align*} 
  By using the branching property, conditionally on $\nf^{\Omega}_{|u|}:= \sigma( (u_{k}, V(u_k)), (z, V(z))_{z \in \Omega(u_{k}) } : 1\leq k\leq |u| )$, for each $u \in \cup_{n \geq 0} \mathbb{N}^{n}$, we have 
  \begin{equation*}
    \E\left[\ind{u\in\T} \ind{E_u\cap F_u^L}\left(\sum_{v: v>u}\ind{E_v}\right) \mid  \nf^{\Omega}_{|u|} \right]  \leq    \E\left[\ind{u \in \mathbb{T}} \ind{E_u\cap F_u^L}\E[N(x+V(u))\mid  V(u)]\right], 
  \end{equation*}
  and 
  \begin{align*}
    & \E\left[\sum_{u\in\T}\ind{E_u\cap F_u^L}\left(\sum_{k=1}^{|u|}\sum_{z\in\Omega(u_k)}\sum_{v: v\geq z}\ind{E_v}\right) \mid    \nf^{\Omega}_{|u|} \right] \\
    & \leq  \E\left[\ind{u \in \mathbb{T}} \ind{E_u\cap F_u^L}\left(\sum_{k=1}^{n}\sum_{z\in\Omega(u_k)}\E[N(x+V(z))\mid  V(z)]\right)\right]. 
  \end{align*} 
  In view of \eqref{meanNx}, for $L\geq1$, we obtain that 
  \begin{align*}
  \E[N_L(x)^2]& \leq   
  c_{10}e^{-\kappa x}  + 
  c_{11} \sum_{n\geq1}\E\left[\sum_{|u|=n}\ind{E_u }e^{-\kappa (V(u)+x)}\right] \\
  & \phantom{  c_{100}e^{-\kappa x}+\, } + c_{11}\sum_{n\geq1}\E\left[\sum_{|u|=n}\ind{E_u\cap F_u^L}\left(\sum_{k=1}^{n}\sum_{z\in\Omega(u_k)}e^{-\kappa (V(z)+x)}\right)\right]\\
  &\leq  c_{10}e^{-\kappa x} + c_{11} \E[N(x)]+ c_{11}L \E[N(x)]\leq c_{12} L e^{-\kappa x}.
  \end{align*}
  Therefore, for $L$ sufficiently large and fixed, we conclude from \eqref{PaleyZygmund}  \eqref{eq-1st-N-L-x} and the above estimate that $
  \P(\tM\leq -x)\geq c_{13} e^{-\kappa x} $. 
  \end{proof}

\section{Proof of Theorem \ref{BRWcvg}}\label{mainthm}

\subsection{Tightness conditioned on $\tM\leq -x$.} \label{tightness}

In this part, we give some lemmas to establish the tightness of $(\sum_{u\in\T}\delta_{V(u)-\tM},  \frac{\sqrt{\psi''(\kappa)}}{\sqrt{z}}(|u^*|-\frac{z}{\psi'(\kappa)}), \frac{D_\infty}{ze^z}, \frac{W_\infty}{e^z}, \tM+z)$ conditioned on $\tM\le -z$. The proofs will be presented in Section \ref{sec4-proof}.

Recall that $u^*$ is chosen among the youngest ones achieving $\tM$.  Let $x\mapsto b(x)$ be a non-negative function such that $b(x)=o(\sqrt{x})$ when $x\uparrow\infty$. 
Recall  
$ J(x)= [\frac{x}{\psi'(\kappa)}-b(x)\sqrt{x}, \frac{x}{\psi'(\kappa)}+b(x)\sqrt{x}  ]$
in \eqref{DefJx} and recall that $I(x)=(-x-1,-x]$.

The tightness of $\frac{\sqrt{\psi''(\kappa)}}{\sqrt{x}}(|u^*|-\frac{x}{\psi'(\kappa)})$ comes from the following lemma. 
\begin{lem}\label{generationofM}
 Under the assumption of Theorem \ref{BRWcvg}, there exists some constant $c_{2}>0$ such that for all $x$ large enough, 
 \[
 \P\left( |u^*|\notin J(x) \mid \tM\in I(x) \right)\le \exp(-c_{2} b(x)^2).
 \]
\end{lem}

About the point process $\sum_{u\in\T}\delta_{V(u)-\tM}$, we have the next lemma.
\begin{lem}\label{BRWtight}
Under the assumptions of Theorem \ref{BRWcvg}, 
 for any fixed $A>0$, as $K$ goes to infinity,
\begin{equation}\label{tightpp}
\limsup_{x\to\infty}\P\left(\sum_{u\in\T}\ind{V(u)-\tM\le A}\ge K \mid  \tM\le -x\right)=o_K(1).
\end{equation}
\end{lem}

 Recall that $\W^\tM=e^{\tM}W_\infty$ and $\D^\tM=e^{\tM}D_\infty$. Employing inequalities \eqref{eq-bound-W} and \eqref{eq-desired-bound-D},
 we could bound their conditioned moments in the  following lemma which suffices to deduce tightness.

\begin{lem}\label{bdcondmomWD}
Under the Assumptions \ref{cond1},\ref{cond2}, \ref{cond3} and \ref{cond4}, for any $\delta\in(0,1\wedge \delta_0)$,
we have
\begin{equation}\label{momWM}
\sup_{x\in\mathbb{R}_+}\E\left[(\mathcal{W}^\tM)^{\kappa+\delta}\mid \tM\le -x\right]<\infty,
\end{equation}
and
\begin{equation}\label{momDM}
\sup_{x\in\R_+}\E\left[\left(\frac{|\mathcal{D}^\tM|}{1+x}\right)^{\kappa+\delta}\mid \tM \le -x\right]<\infty.
\end{equation}
\end{lem}

\subsection{Weak convergence conditioned on $\tM\leq -x$: proof of Theorem \ref{BRWcvg}.}\label{sec-weakconv}

In this section, we will prove the following proposition, which, combined with the Lemmas in the previous section, suffices to deduce Theorem \ref{BRWcvg}. 

\begin{prop}\label{BRWmaincvg}
Under the assumptions of Theorem \ref{BRWcvg}, for any continuous and bounded functions $\phi_i: \R\rightarrow\R_+$, $i=0,1,2$, and any compactly supported continuous function $g:\R_+\to \R_+$, the following limit exists.
\begin{multline}
\lim_{x\rightarrow\infty} e^{\kappa x}\E\left[e^{-\sum_{u\in\T}g(V(u)-\tM)}\phi_0(\sqrt{\frac{\psi'(\kappa)}{x}}[|u^*|-\frac{x}{\psi'(\kappa)}])\phi_1(\W^\tM)\phi_2(\frac{\D^\tM}{|u^*|})\ind{\tM\leq -x}\right] \\
=\C(\phi_1,\phi_2,g)\E[\phi_0(G)]
\end{multline}
Here $G$ is a centered Gaussian random variable of variance $\frac{\psi''(\kappa)}{\psi'(\kappa)^2}$, and $\C(\phi_1,\phi_2,g)$ is some finite number.
\end{prop}

Note that this proposition immediately implies that as $x\to\infty$,
\[
\P(\tM\le -x)\sim c_\tM e^{-\kappa x}.
\]
And $c_\tM\in(0,1]$ because of Lemma \ref{BRWroughbd}. Moreover, we will see from the proof that conditioned on $\tM \le -x $, $(\frac{\D^\tM}{|u^*|}, \W^\tM)$ converges jointly in law to $((\psi'(\kappa)-\psi'(1))Z, Z)$ with non-negative limiting random variable $Z$.

In order to prove Proposition \ref{BRWmaincvg}, first, we make a decomposition of the martingale limits at the first generation. 
$\P$-a.s.,
\[
W_\infty=\sum_{|z|=1}e^{-V(z)}W_\infty^{(z)},\quad D_\infty =\sum_{|z|=1}e^{-V(z)}[D_\infty^{(z)}+(-V(z)-\psi'(1))W_\infty^{(z)}]
\]
where $(W_\infty^{(z)}, D_\infty^{(z)}), |z|=1$ are martingale limits associated with the subtrees rooted at $z$, respectively, which are therefore i.i.d. copies of $(W_\infty, D_\infty)$ and are independent of $(V(z), |z|=1)$. 

For any $u\in\T$ such that $|u|=n$, let $(\rho, u_1,\cdots, u_n)$ be its ancestral line. This means that $u_k$ is the ancestor of $u$ at the $k$-th generation. For any $z\in \T$,  $\Omega(z)$ is the set of all brothers of $z$, i.e.,
\begin{equation*} 
\Omega(z):=\{v\in \T: \overleftarrow{v}=\overleftarrow{z}, v\neq z\}.
\end{equation*}
Then, observe that $\P$-a.s.,
\begin{align*}
W_\infty=&\sum_{k=1}^n \sum_{z\in\Omega(u_k)}e^{-V(z)}W_\infty^{(z)}+e^{-V(u)}W_\infty^{(u)},\\
D_\infty=&\sum_{k=1}^n \sum_{z\in\Omega(u_k)}e^{-V(z)}[D_\infty^{(z)}-(V(z)+k\psi'(1))W_\infty^{(z)}]+e^{-V(u)}[D_\infty^{(u)}-(V(u)+n\psi'(1))W_\infty^{(u)}].
\end{align*}
For any integer $t$, the truncated versions are defined by
\begin{align*}
\W^{u, \le t}:=&\sum_{k=(n-t)_+}^n \sum_{z\in\Omega(u_k)}e^{V(u)-V(z)}W_\infty^{(z)}+W_\infty^{(u)},\\
\D^{u,\le t}:=&\sum_{k=(n-t)_+}^n \sum_{z\in\Omega(u_k)}e^{V(u)-V(z)}[D_\infty^{(z)}-(V(z)+k\psi'(1))W_\infty^{(z)}]+[D_\infty^{(u)}-(V(u)+n\psi'(1))W_\infty^{(u)}],
\end{align*}
with $(j)_+:=\max\{j,0\}$. Recall that  $u^*$ is chosen at random among the youngest ones attaining $\tM$. So, by taking $u=u^*$, we have
$\W^{u^*,\le t}$ and $\D^{u^*, \le t}$ which will be compared with the following terms.
\begin{align}\label{eMW}
\mathcal{W}^{\tM}=&e^{\tM}W_\infty=\sum_{k=1}^{|u^*|} \sum_{z\in\Omega(u^*_k)}e^{\tM-V(z)}W_\infty^{(z)}+W_\infty^{(u^*)},\\
\mathcal{D}^{\tM}=&e^{\tM}D_\infty=\sum_{k=1}^{|u^*|} \sum_{z\in\Omega(u^*_k)}e^{\tM-V(z)}[D_\infty^{(z)}-(V(z)+k\psi'(1))W_\infty^{(z)}]\nonumber\\
& \qquad\qquad+ [D_\infty^{(u^*)}+(-\tM-|u^*|\psi'(1))W_\infty^{(u^*)}]\label{eMD}.
\end{align}

Our goal is to establish the weak convergence of $(\W^\tM, \frac{\D^\tM}{|u^*|})$ conditioned on $\tM\le -x$. The following lemma shows that we only need to study this truncated versions. 

\begin{lem}\label{BRWrest}
Under the assumptions of Theorem \ref{BRWcvg}, for any $\delta>0$, we have
\begin{align}
\lim_{t\rightarrow\infty}\sup_{x\in\R_+}\P\left(\mathcal{W}^{\tM}-\mathcal{W}^{u^*,\leq t}\geq \delta \mid \tM\leq-x\right)=&0, \label{eq-truncate-W}\\
\lim_{t\rightarrow\infty}\sup_{x\in\R_+}\P\left(\frac{1}{1+|u^*|}|\mathcal{D}^{\tM}-\mathcal{D}^{u^*,\leq t}|\geq \delta \mid \tM\leq-x\right)=&0. \label{eq-truncate-D}
\end{align} 
\end{lem}

The proof of Lemma \ref{BRWrest} is postponed to Section \ref{sec4-proof}.  Next, we will see that conditioned on $\{\tM\le -x\}$, 
\[
\frac{\D^{u^*,\le t}}{|u^*|}- (\psi'(\kappa)-\psi'(1))\W^{u^*,\le t} 
\]
converges in probability to zero as $x\to\infty$. So, we only need to verify the joint convergence of 
\[
\left( \sum_{u\in\T}\delta_{V(u)-\tM},\, \W^{u^*,\le t} ,\, \sqrt{\frac{\psi'(\kappa)}{x}} \left[ |u^*|-\frac{x}{\psi'(\kappa)}  \right] \right)
\]
conditioned on $\{\tM\le -x\}$. This is what we state in the following Lemma \ref{BRWtruncatedcvg}.

\begin{lem}\label{BRWtruncatedcvg}
Let $t$ be a fixed integer.
Under the assumptions of Theorem \ref{BRWcvg}, for any $\delta>0$
\begin{equation}\label{approxD}
\limsup_{x\to\infty}   \P\left( \left|\frac{\D^{u^*,\le t}}{|u^*|}- (\psi'(\kappa)-\psi'(1))\W^{u^*,\le t} \right| \geq \delta  \mid  \tM\le -x \right)=0.  
\end{equation}
Moreover, for any continuous and bounded functions $\phi_i: \R\rightarrow\R_+$, $i=0,1$, and any compactly supported continuous function $g:\R_+\to \R_+$, the following limit exists. 
\begin{multline}\label{truncatedlimit}
\lim_{x\rightarrow\infty} e^{\kappa x}\E\left[e^{-\sum_{u\in\T}g(V(u)-\tM)}\phi_0\left(\sqrt{\frac{\psi'(\kappa)}{x}}[|u^*|-\frac{x}{\psi'(\kappa)}]\right) \phi_1(\mathcal{W}^{u^*,\leq t})   \ind{\tM\leq -x}\right]\\
=\C_t(\phi_1,g)\E[\phi_0(G)]
\end{multline}
Here $G$ is a centered Gaussian random variable of variance $\frac{\psi''(\kappa)}{\psi'(\kappa)^2}$. The explicit expression of $\C_t(\phi_1,g)$ is given in \eqref{Ctphig}.
\end{lem}

\begin{proof}[Proof of Lemma \ref{BRWtruncatedcvg}] 

We first prove \eqref{truncatedlimit}.
In fact, we only need to show the convergence of $e^{\kappa x}\E_\eqref{LaplacecondM}$ where $\E_\eqref{LaplacecondM}$ equals to
\begin{equation}\label{LaplacecondM}
\E\left[e^{-\sum_{u\in\T}g(V(u)-\tM)}\phi_0(\sqrt{\frac{\psi'(\kappa)}{x}}[|u^*|-\frac{x}{\psi'(\kappa)}])\phi_1(\mathcal{W}^{u^*,\leq t})\ind{\tM\in I(x)}\right],
\end{equation}
with $I(x)=(-x-1,-x]$. 

Recall that $u^*$ is chosen at random among the youngest individuals attaining $\tM$, and that  $J(x)$ is defined in \eqref{DefJx} with $1 \ll b(x) \ll \sqrt{x}$. 
By Lemma \ref{generationofM} and Lemma \ref{BRWroughbd}, we have
\begin{equation}\label{spineuptomin}
\E_\eqref{LaplacecondM} 
=\sum_{n\in J(x)}\phi_0(\sqrt{\frac{\psi'(\kappa)}{x}}(n-\frac{x}{\psi'(\kappa)}))\E_\eqref{spineuptomin}(n)+o(e^{-\kappa x}),
\end{equation} 
where 
\begin{equation*}
\E_\eqref{spineuptomin}(n):=\E\left[\frac{\sum_{|u|=n}\ind{V(u)=\tM<\tM_{n-1}, V(u)\in I(x)}\phi_1(\mathcal{W}^{u,\leq t})e^{-\sum_{z\in\T}g(V(z)-V(u))}}{\sum_{|v|=n}\ind{V(v)=\tM}}\right],
\end{equation*}
with $\tM_k=\inf_{|z|\le k}V(z)$.

Then, applying Lyons' change of measure $ \d   (\mathbf{Q}^{\kappa}_{\vert n} \otimes \mathbf{P})= W_{n}(\kappa) \d \P $ and   using Proposition \ref{BRWchangeofp2}  we have   
 \begin{equation*}
  \E_\eqref{spineuptomin} (n) =   \mathbf{E}_{\mathbf{Q}^{\kappa,*}_{\vert n} \otimes \mathbf{P}}\left(\frac{e^{\kappa V\left(w_n\right)} \phi_{1}\left(   \mathcal{W}^{w_n, \leq t}\right)}{ \sum_{|v|=n}\ind{V(v)=V(w_{n})} } e^{-\sum_{z\in\T}g(V(z)-V(w_{n}))} \ind{\left\{V\left(w_n\right)=\mathbf{M}<\mathbf{M}_{n-1}, V\left(w_n\right) \in I(x)\right\}}  
   \right) .  
 \end{equation*}
 
 The event $\{ V\left(w_n\right)=\mathbf{M}<\mathbf{M}_{n-1}, V\left(w_n\right) \in I(x)\}$ can be rewritten as 
 \begin{equation*}
   \left\{V\left(w_n\right)<\min _{0\le j \leq n-1} V\left(w_j\right) ,V\left(w_n\right) \in I(x) \right\} \cap \left\{ \mathbf{M}^{\left(w_n\right)} \geq 0 \right\} \cap \left( \cap_{j=1}^{n} A_{j} \right)
   \end{equation*}
 where  $A_{j}:=\cap_{z \in \Omega\left(w_j\right)}\{V(z)+\mathbf{M}^{(z)} \geq V\left(w_n\right) , V(z)+\mathbf{M}_{n-1-j}^{(z)}>V\left(w_n\right)  \}$ with $\tM^{(z)}:=\inf_{v\in\T_z} V(v)-V(z)$ and $\tM_k^{(z)}:=\inf_{v\in\T_z, |v|\le |z|+k}V(v)-V(z)$.

Note that $g$ is compactly supported, suppose that $\mathrm{supp}(g)\subset [0,K_g]$ for some $K_g>0$. Next,  for any integer $n\geq b\geq1$, let us define
 \begin{equation*}
  G_n\left(b\right):= \cap_{j=1}^{n-b} \left\{  \forall z \in \Omega\left(w_j\right),  V(z)+\mathbf{M}^{(z)} > V\left(w_n\right)+K_{g}; V(w_{j})>V(w_{n})+K_{g}  \right\},  
 \end{equation*} 
We claim that the following lemma holds. Its proof will be given in Section \ref{sec4-proof}.

\begin{lem}\label{badmin}
\begin{equation}\label{awayspine}
\limsup_{b\rightarrow\infty}  \limsup_{x\rightarrow\infty} \sum_{ n\geq b+1 }\mathbf{Q}^{\kappa,*}    \left(   V(w_n)\in I(x),V(w_n)<\underline{V}(w_{[0,n-1]}) ,  G_{n} (b)^c\right)=0,
\end{equation}
where $\underline{V}(w_{[0,m]}): = \min_{0\le j\le m}V(w_j) $.
\end{lem}

Note that Lemma \ref{badmin} holds also if we replace $\mathbf{Q}^{\kappa,*}$ by $\mathbf{Q}^{\kappa,*}_{\vert n} \otimes \mathbf{P}$.

Observe that  on the event $G_{n}(b)$ we have $\ind{A_{j}}=1$ for all $j \leq n-b$, and
 \begin{equation*}
  \sum_{ |v|=n} \ind{V(v)=V\left(w_n\right)} =1+ \sum_{j=n-b+1}^{n} \sum_{u \in \Omega\left(w_j\right)} \sum_{v \geq u,|v|=n} \ind{V(v)=V\left(w_n\right)} =: \mathtt{mult},
 \end{equation*}  
 in addition,
\begin{equation*}
  \sum_{z \in \mathbb{T}} g(V(z)-V(w_{n})) 
   = \sum_{j=n-b+1}^{n} g(V(w_{j})-V(w_{n}))  + \Sigma(g;n,b)+\Sigma_{w_n}(g)
\end{equation*}
where 
\[
\Sigma_{w_n}(g):=\sum_{z\in\T_{w_n}}g(V(z)-V(w_n))\textrm{ and }\Sigma(g;n,b):= \sum_{j=n-b+1}^{n} \sum_{z \in \Omega(w_{j})} \sum_{v \geq z} g(V(v)-V(w_{n})).
\]
In view of this Lemma \ref{badmin} and \eqref{spineuptomin}, we obtain that
\begin{equation}\label{GoodLaplace}
e^{\kappa x}\E_\eqref{LaplacecondM}= e^{\kappa x}\sum_{n\in J(x)} \phi_0 \left[ \sqrt{\frac{\psi'(\kappa)}{x}}(n-\frac{x}{\psi'(\kappa)}) \right] \chi_n(b)+o_{x,b}(1),
\end{equation}
where $\chi_n(b)$ is defined by
\begin{align*}
  \chi_n(b):= \E_{\Q^{\kappa,*}_{\vert n} \otimes \P}\bigg[
  \frac{e^{\kappa V(w_n)}\phi_{1}(\W^{w_n,\le t} )}{  \mathtt{mult}  }
  & \, e^{-  \sum_{j=n-b+1}^{n} g(V(w_{j})-V(w_{n})) } \,  e^{ - \Sigma(g;n,b)}   \\
  &  \times \ind{V(w_n)\in I(x), V(w_n)<\underline{V}(w_{[0,n-1]})}  e^{-\Sigma_{w_n}(g)}\ind{ \mathbf{M}^{(w_{n})} \geq 0} \prod_{j=n-b+1}^{n} \ind{A_{j}}
 \bigg] .
\end{align*}
Recall that under $\Q^{\kappa,*}_{\vert n}\otimes \P$, $(\Delta V(w_i),\sum_{u\in\Omega(w_i)}\delta_{\Delta V(u)})_{1\le i\le n}$ are i.i.d. and given $u\in \{w_n\}\bigcup\cup_{i=1}^n\Omega(w_i)$, $(V(z)-V(u))_{z\in\T_u}$ are i.i.d. and distributed as $\P$. We thus have
  \begin{align*}
   & \left\{ V\left(w_j\right),(V(u),  \{V(z)-V(u)\}_{z\in\T_u})_{u \in \Omega\left(w_j \right)} \right\}_{j=1}^{n} \\
   & \qquad \qquad \overset{\mathrm{law}}{ =}   \left\{ V\left(w_n\right)-V\left(w_{n-j}\right),(V(u), \{V(z)-V(u)\}_{z\in\T_u})_{u \in \Omega\left(w_{n-j+1}\right)}\right\}_{j=1}^{n}.
  \end{align*} 
Operating this time reversal we have the following structure (defined under $\Q^{\kappa,*}$): The spine $(w_k, V(w_k))_{0\le k\le n}$ is time-reversed random walk as above, combined with $\cup_{1\le i\le n}\cup_{u\in\Omega(w_i)}\T_u$. To the new $w_0$, we attach an extra $\P$-distributed branching random walk. Yet we will not count the descendants of $w_n$.

\begin{figure}[htbp]
\begin{center}
\includegraphics[width=0.9\textwidth]{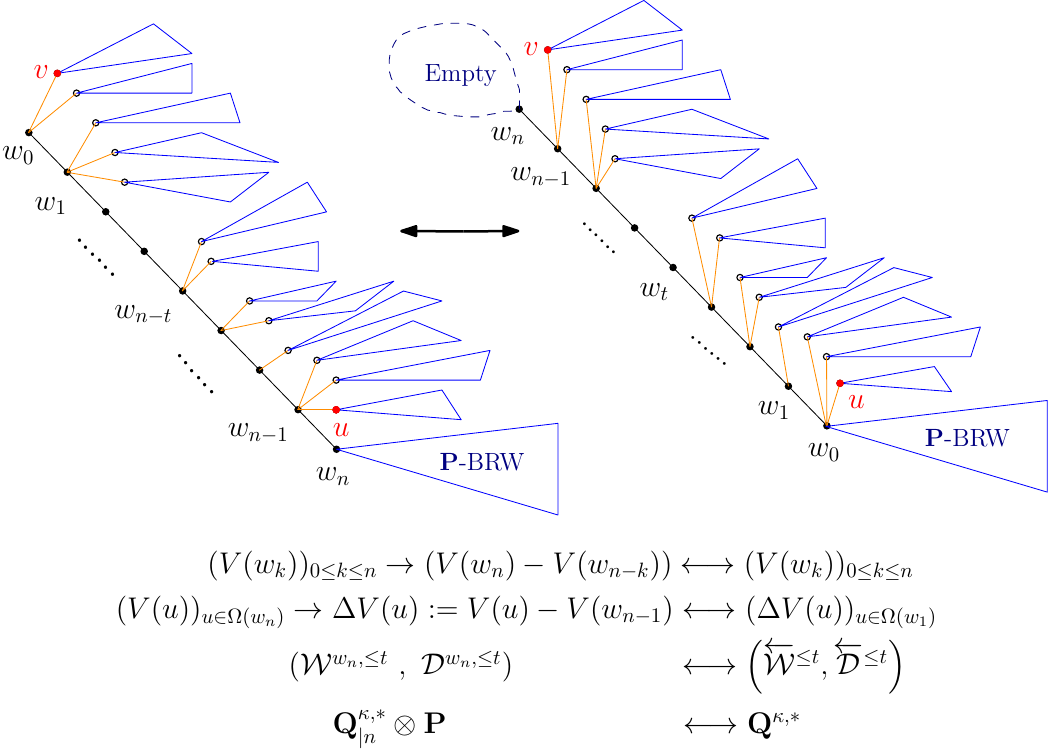}
\caption{Time-reversing tree structure}
\end{center}
\end{figure}

We thus get the following correspondence:
\begin{enumerate}[(a)]
  \item For the extra BRW rooted at $w_0$, we have similarly the random variables $W^{(w_0)}_{\infty}$, $D^{(w_0)}_{\infty}$, $\tM^{(w_0)}$, and $ \Sigma_{w_0}(g)$. Let $\Q^{\kappa,*}((W^{(w_0)}_{\infty}, D^{(w_0)}_{\infty}, \tM^{(w_0)} \Sigma_{w_0}(g))\in \cdot) =\P((W_\infty, D_\infty, \tM,\sum_{u\in\T}g(V(u)))\in\cdot)$. 
  \begin{equation*}
   \mathcal{W}^{w_n, \leq t}  = \sum_{k=n-t}^{n} e^{V\left(w_{n}\right)-V\left(w_{k-1}\right)}
  \sum_{z \in \Omega\left(w_k\right)}   e^{-\Delta V(z)}    W_{\infty}^{(z)}   
 \longleftrightarrow \sum_{\ell =1}^{t+1} e^{V(w_{\ell})} \sum_{z \in \Omega(w_{\ell })} e^{-\Delta V(z)}  W^{(z)}_{\infty}  + W^{(w_0)}_{\infty}   =:  \overleftarrow{\mathcal{W}}^{ \le t} . 
\end{equation*}
  and 
 \begin{align*}
   & \mathcal{D}^{w_n, \leq t}  = \sum_{k=n-t}^{n} e^{V\left(w_{n}\right)-V\left(w_{k-1}\right)}
   \sum_{z \in \Omega\left(w_k\right)}   e^{-\Delta V(z)} ( D_{\infty}^{(z)} -[\Delta V(z) + V(w_{k-1}) +\psi'(1) k]\, W_{\infty}^{(z)} )  \\
     & \qquad \qquad +    [ D_{\infty}^{(w_{n})}  -[V(w_{n})+\psi'(1)n]\, W_{\infty}^{(w_{n})}]  \\
  &\longleftrightarrow \sum_{\ell =1}^{t+1} e^{V(w_{\ell})} \sum_{z \in \Omega(w_{\ell })} e^{-\Delta V(z)} \left\{ D^{(z)}_{\infty}  - \left[ \Delta V(z) + V(w_{n}) - V(w_{\ell}) + \psi'(1) (n-\ell+1) \right] W^{(z)}_{\infty} \right\} \\
  & \qquad \qquad + \left\{ D^{(w_0)}_{\infty} - \left[  V(w_{n}) + \psi'(1) n \right] W^{(w_0)}_{\infty}  \right\} =:  \overleftarrow{\D}^{ \le t}  .
 \end{align*} 
 \item $\sum_{j=n-b+1}^{n} g(V(w_{j})-V(w_{n}))  \longleftrightarrow  \sum_{\ell=1}^{b}  g(-V(w_{\ell})) $, and 
\begin{align*}
 \Sigma(g;n,b) & = \sum_{j=n-b+1}^{n} \sum_{z \in \Omega(w_{j})} \sum_{v \geq z} g(  V(v)-V(z)+ \Delta V(z)+ V(w_{j-1})-V(w_{n})) \\ 
 &\qquad  \longleftrightarrow  \sum_{\ell=1}^{b}   \sum_{z \in \Omega(w_{\ell})} \sum_{v \geq z} g( V(v)-V(z)+ \Delta V(z)-V(w_{\ell}))  := \overleftarrow{\Sigma}(g;b). 
\end{align*} 
\item 
 For each $n-b \leq j\leq n$ with $\ell=n-j+1$,
 \begin{align*}
  A_{j} = & \bigcap_{z \in \Omega\left(w_j\right)}\{ \Delta V(z)+ V(w_{j-1}) +\mathtt{M}^{(z)} \geq V\left(w_n\right) , \Delta V(z)+ V(w_{j-1}) + \mathtt{M}_{n-j-1}^{(z)}>V\left(w_n\right) \} \\
   &\qquad  \longleftrightarrow \bigcap_{z \in \Omega\left(w_\ell\right)}\{ \Delta V(z)+\mathtt{M}^{(z)} \geq V\left(w_\ell \right) , \Delta V(z)+\mathtt{M}_{\ell-2}^{(z)}>V\left(w_\ell\right) \}:= \overleftarrow{A}_{\ell},
 \end{align*} 
\item Let $\T_{\ell}^{(z)}:=\{v\in\T\ :\ v\geq z,\ |v|=|z|+\ell\}$. Then 
\begin{align*}
  \mathtt{mult} & := 1+ \sum_{j=n-b+1}^{n} \sum_{z \in \Omega\left(w_j\right)} \sum_{v \geq z,|v|=n} \ind{V(v)=V\left(w_n\right)} \\
  & \qquad \quad  \longleftrightarrow 1+ \sum_{\ell=1}^{b} \sum_{z \in \Omega\left(w_\ell\right)} \sum_{v \in \T_z, |v|=|z|+\ell -1} \ind{V(v)-V(z) + \Delta V(z)=V\left(w_\ell\right)} =: \overleftarrow{\mathtt{mult}}
 \end{align*}
\end{enumerate} 
For simplicity, let 
$\Xi_{b}:=e^{- \sum_{\ell=1}^{b}  g(-V(w_{\ell}))  } e^{-\Sigma_{w_0}(g)} \ind{\mathtt{M}^{\left(w_0\right)} \geq 0} \, e^{ - \overleftarrow{\Sigma}(g;b) }\prod_{\ell=1}^{b}\ind{\overleftarrow{A}_{\ell}}  $. With this definition, we obtain
 \begin{equation*}
  \chi_{n}(b)=   \mathbf{E}_{\mathbf{Q}^{\kappa,*}}\left( \frac{ e^{\kappa V(w_n)}}{\overleftarrow{\mathtt{mult}}}  \phi_{1}(\overleftarrow{\W}^{\le t} ) \, \ind{\overline{V}(w_{[1,n]})<0, V(w_n)\in I(x)}  \,  \Xi_{b}  \right).  
 \end{equation*}

Let   $b>t$ be sufficiently large.  Conditionally on $\sigma( V(z): z \geq w_{j}, j < 2b )$ (note that $\Xi_{b}, \overleftarrow{\W}^{\le t}, \overleftarrow{\mathtt{mult}}$ are measurable to this sigma-field), 
by  the branching property at $w_{2b}$ ,
one gets that
\begin{equation*} 
\sum_{n \in J(x)} e^{\kappa x} \phi_0(\sqrt{\frac{\psi'(\kappa)}{x}}(k-\frac{x}{\psi'(\kappa)})) \chi_{n} (b)
= \E_{\Q^{\kappa,*} }\left[\frac{  \Xi_{b} \,  \phi_{1}(\overleftarrow{\W}^{\le t} )}{\overleftarrow{\mathtt{mult}}}    \ind{\overline{V}(w_{[1,2b]}<0)} \E_\eqref{mainsum}(V(w_{2b})) 
 \right],
\end{equation*}
where
\begin{equation}
\E_\eqref{mainsum}(y):= \sum_{n \in J(x)-2b} \phi_0(\sqrt{\frac{\psi'(\kappa)}{x}}(n+2b-\frac{x}{\psi'(\kappa)}))\E_{y}\left[e^{\kappa S^{(\kappa)}_{n}+\kappa x}\ind{\MS^{(\kappa)}_{[1,n]}<0, S^{(\kappa)}_n\in I(x)}\right].  \label{mainsum}
\end{equation}
By Lemma \ref{Rfunctioncvg} and dominated convergence theorem,  
\begin{equation*}
 \lim_{x \to \infty} \sum_{n \in J(x)} e^{\kappa x}  \phi_0(\sqrt{\frac{\psi'(\kappa)}{x}}(k-\frac{x}{\psi'(\kappa)})) \chi_{n} (b) =  \E[\phi_0(G) ] \E_{\Q^{\kappa,*} }\left[\frac{  \Xi_{b} \,  \phi_{1}(\overleftarrow{\W}^{\le t} )}{\overleftarrow{\mathtt{mult}}}    \ind{\overline{V}(w_{[1,2b]}<0)} \mathbf{c}(V(w_{2b})) 
\right].  
\end{equation*}
Going back to \eqref{GoodLaplace}. Letting $x\to\infty$ and then $b\to\infty$ yields \eqref{truncatedlimit} with
\begin{equation}\label{Ctphig}
\C_t(\phi_1,  g):=\E[\phi_0(G) ]  \lim_{b\to\infty}\E_{\Q^{\kappa,*} }\left[\frac{  \Xi_{b} \,  \phi_{1}(\overleftarrow{\W}^{\le t} )}{\overleftarrow{\mathtt{mult}}}    \ind{\overline{V}(w_{[1,2b]}<0)} \mathbf{c}(V(w_{2b})) 
\right] . 
\end{equation}
The existence of this limit is ensured by the monotonicity of $\chi_{n}(b)$ in $b$ and  Lemma \ref{badmin}. This completes the proof of \eqref{truncatedlimit}.

Next we prove \eqref{approxD}.    Let  $J(x)$ be defined in \eqref{DefJx} with $ b(x) \equiv b > 0$.   
As in the proof of \eqref{truncatedlimit}, by use of Lemma \ref{generationofM} and Lemma \ref{BRWroughbd}, and then by change of measure and time reversing, we get that
\begin{align}
  & \mathbf{P}_{\eqref{approxD-1}} := e^{\kappa x}\P \left( \left|\frac{\D^{u^*,\le t}}{|u^*|}- (\psi'(\kappa)-\psi'(1))\W^{u^*,\le t} \right| \geq \delta , \tM \in I(x) \right)  \label{approxD-1} \\
   & \leq  \sum_{n\in J(x)} \E_{\mathbf{Q}^{\kappa,*}_{\vert n} \otimes \P}\left[ e^{\kappa V(w_{n})+\kappa x}   \ind{\left|\frac{1}{n}\D^{w_{n},\le t}- (\psi'(\kappa)-\psi'(1))\W^{w_{n},\le t} \right|>\delta ;  \, \underline{V}(w_{[0,n-1]}) > V(w_n)\in I(x)}  \right] + o_{b}(1) \nonumber \\
   & \leq   \sum_{n\in J(x)} \mathbf{Q}^{\kappa,*} \left( \left|\frac{1}{n} \overleftarrow{\D}^{\le t}- (\psi'(\kappa)-\psi'(1))\overleftarrow{\W}^{\le t} \right|>\delta;   \overline{V}(w_{[1,n]})<0, V(w_n)\in I(x)  \right) + o_{b}(1) . \nonumber
\end{align}
Note that on the event $\{ V(w_{n}) \in I(x)\}$ we have 
\begin{align*}
& \left|\frac{1}{n} \overleftarrow{\D}^{\le t}- (\psi'(\kappa)-\psi'(1))\overleftarrow{\W}^{\le t} \right| 
  \\
& \lesssim  \frac{1}{n}\left( \sum_{\ell =1}^{t+1} e^{V(w_{\ell})} \sum_{z \in \Omega(w_{\ell })} e^{-\Delta V(z)}   
    \left\{  |D^{(z)}_{\infty}|+ ( | V(w_{\ell}) - \Delta V(z)|+ b\sqrt{x} ) W^{(z)}_{\infty}  \right\} +   | D^{(w_0)}_{\infty} | +  b\sqrt{x} W^{(w_0)}_{\infty} \right)  
\end{align*} 
 Conditionally on $\mathcal{B}_{n}= \sigma\{ (w_k, V(w_k))_{0\le k\le n}, (V(z))_{z\in \cup_{1\le k < n}c(w_k)} \}  $,   
applying the Markov property we have  
 \begin{align}
  &  \mathbf{Q}^{\kappa,*} \left( \left|\frac{1}{n} \overleftarrow{\D}^{\le t}- (\psi'(\kappa)-\psi'(1))\overleftarrow{\W}^{\le t} \right|>\delta     \mid \mathcal{B}_{n} \right)  \nonumber \\
  & \leq \min \left\{ 1, \frac{1}{\delta} \, \E_{\mathbf{Q}^{\kappa,*}} \left(   \left|\frac{1}{n} \overleftarrow{\D}^{\le t}- (\psi'(\kappa)-\psi'(1))\overleftarrow{\W}^{\le t} \right|    \mid \mathcal{B}_{n} \right)  \right\}\nonumber \\
  & \lesssim  \min \left\{ 1,  \frac{1}{\delta n} \sum_{\ell =1}^{t+1} e^{V(w_{\ell})} \sum_{z \in \Omega(w_{\ell })} e^{-\Delta V(z)}   
 \left(  b \sqrt{x}+ | V(w_{\ell})|+|\Delta V(z)|  \right)   +\frac{b\sqrt{x}}{\delta n} \right\}  \label{eq-bound-diff-DW}
 \end{align} 
Let $\Delta_{\ell}:= \sum_{z \in \Omega(w_{\ell})} e^{-\Delta V(z)} (|\Delta V(z)|+1)$. Since $\max_{1 \leq \ell \leq t+1} e^{V(w_{\ell})} \Delta_{\ell} $ is finite random variable, we can find $f(b)>1$ such that 
\begin{equation}\label{Deffb}
 \lim_{b \to \infty} \max_{1 \leq \ell \leq t+1 } b\, \mathbf{Q}^{\kappa,*}  \left(  e^{V(w_{\ell})} \Delta_{\ell} > f(b)  \right)  = 0 . 
\end{equation}
We then    bound  \eqref{eq-bound-diff-DW} by
 \begin{align*} 
  &\min \left\{ 1,  \frac{b \sqrt{x}}{\delta n} \sum_{\ell =1}^{t+1} e^{V(w_{\ell}) }  \left( 1+ |V(w_{\ell})|  \right) \Delta_{\ell} + \frac{b\sqrt{x}}{n}  \right\}  \\
  &\lesssim   \sum_{\ell=1}^{t+1} \ind{ e^{V(w_{\ell})}  \Delta_{\ell} > f(b)    } + \frac{ b f(b)    \sqrt{x} }{\delta n } \sum_{\ell =1}^{t+1} ( 1+|V(w_{\ell})|) . 
 \end{align*}
Going back to \eqref{approxD-1}, we thus obtain that
 \begin{align*}
  \mathbf{P}_{\eqref{approxD-1}} \lesssim \sum_{n \in J(x)} &  \sum_{\ell=1}^{t+1} \E_{\mathbf{Q}^{\kappa,*}}  \left[   \P_{ V (w_{\ell}) } \left( S^{(\kappa)}_{n-\ell} \in I(x) \right)  ; e^{V(w_{\ell})} \Delta_{\ell} > f(b)    \right]  \\
  & + \sum_{n \in J(x)} \frac{b f(b) \sqrt{x} }{\delta n } \sum_{\ell =1}^{t+1} \E_{\mathbf{Q}^{\kappa,*}} \left[   ( 1+|V(w_{\ell})|)  \P_{ V (w_{\ell}) } \left( S^{(\kappa)}_{n-\ell} \in I(x) \right) \right]  +o_b(1).  
 \end{align*}
By  the local limit theorem \cite[Corollary 1]{Stone1967},  that  is a constant $c_3>0$ depending only on the random walk  such that for all $x \in \mathbb{R}$,
 \begin{equation*}
  \sup_{x \in \mathbb{R}} \mathbf{P} ( S^{(\kappa)}_{n} \in [x-1,x] )  \leq   \frac{c_3}{\sqrt{n}} . 
 \end{equation*} 
Therefore, 
\begin{align*}
  \mathbf{P}_{\eqref{approxD-1}}  & \lesssim  \sum_{\ell=1}^{t+1} \sum_{n\in J(x)} \frac{1}{\sqrt{n-\ell}} \mathbf{Q}^{\kappa,*} \left(  e^{V(w_{\ell})}  \Delta_{\ell} > f(b)   \right)  
    +  \sum_{n \in J(x)} \frac{b f(b) \sqrt{x} }{\delta n } \frac{1}{\sqrt{n-\ell}} +o_b(1) \\
    & \overset{x \to \infty}{\longrightarrow}   \sqrt{\psi'(\kappa) } \sum_{\ell=1}^{t+1} b\,\mathbf{Q}^{\kappa,*} \left(  e^{V(w_{\ell})}  \Delta_{\ell} > f(b)   \right)  +o_b(1) \overset{b \to \infty}{\longrightarrow}  0  
\end{align*}
and the desired result follows.  
\end{proof}

\subsection{Proofs of Lemmas in Sections \ref{tightness} and \ref{sec-weakconv}}
 \label{sec4-proof}

\begin{proof}[Proof of Lemma \ref{generationofM}] 
Recall that $J(x):=[\frac{x}{\psi'(\kappa)}-b(x)\sqrt{x}, \frac{x}{\psi'(\kappa)}+b(x)\sqrt{x}]$ with $1\ll b(x)\ll \sqrt{x}$.
 By Lemma \ref{BRWroughbd}, it suffices to show that there is some constant $c_{2}>0$ such that 
      \begin{equation*}
      e^{\kappa x}  \P \left( |u^{*}| \notin  J(x),  \mathbf{M} \in I(x) \right) \leq  e^{-c_{2} b(x)^{2}} .
      \end{equation*}
 Set $n_\pm:=\frac{x}{\psi'(\kappa)}\pm b(x)\sqrt{x}$. By many-to-one Lemma \ref{ManytoOne},
 \begin{align*}
  &e^{\kappa x}  \P \left( |u^{*}| \notin  J(x),  \mathbf{M} \in I(x) \right) \\
  &\leq  e^{ \kappa x } \P\left(\exists |u|\le n_-, V(u)\le -x\right) +  e^{ \kappa x }\sum_{n\ge n_+}\P\left(\exists |u|=n, V(u)\in I(x)\right)\\
  &\leq  \P\left( \mS^{(\kappa) }_{[1,n_-]}\le -x\right) + \sum_{n\ge n_+} \P(S_n^{(\kappa)}\in I(x)) .
 \end{align*}

      On the one hand, as $ \ln \mathbf{E}(e^{-\lambda S^{(\kappa)}_1 }) = \psi(\kappa+\lambda)\in(0,\infty)$ for $\lambda\in(0,\delta_0)$, by Doob's inequality for submartingale $e^{-\lambda S_n^{(\kappa)}}$,
      \begin{align*}
     \P\left( \mS^{(\kappa) }_{[1,n_-]}\le -x\right) =   \P\left( \max_{1\le k\le n_-} e^{ -\lambda S_k^{(\kappa)}} \ge e^{ \lambda x}\right)
     \le  e^{-\lambda x +\psi(\kappa+\lambda) n_- }.
      \end{align*}
    Note that $\psi(\kappa + \lambda )= \lambda \psi'(\kappa ) + O(\lambda^2) $.  We could take $\lambda = \delta \frac{b_x}{\sqrt{x}}$ with some small $\delta>0$ so that
      \[
        \P\left( \mS^{(\kappa) }_{[1,n_-]}\le -x\right)  \lesssim e^{- \frac12 \delta \psi'(\kappa) b(x)^2}.
      \]
      
      On the other hand, for $n\ge n_+$ and $\lambda>0$, one sees that
      \begin{align*}
      \P(S_n^{(\kappa)}\in I(x)) &=  \E\left[ e^{-\lambda S_n^{(\kappa)} + \lambda S_n^{(\kappa)} }\ind{ S_n^{(\kappa)} \in I(x) }\right] \\
      &\le  e^{\lambda (x+1) + n \psi(\kappa - \lambda )} \P\left( S_n^{(\kappa - \lambda )} \in I(x) \right).
      \end{align*}
      Again, note that $\psi(\kappa-\lambda)= -\lambda \psi'(\kappa) + O(\lambda^2)$ for $\lambda \in (0, \kappa')$. For $n\ge \frac{x}{ \psi'(\kappa) } (1+\epsilon ) $ with some sufficiently small and fixed $\epsilon>0$, we could take $\lambda \in (0, \kappa' ) $ so that with some $\epsilon_0>0$,
      \[
      \P(S_n^{(\kappa)}\in I(x)) \le e^{\lambda (x+1) + n \psi(\kappa - \lambda )} \lesssim e^{ -\epsilon_0 n }.
      \]
      This implies that
      \[
      \sum_{n\ge \frac{x}{ \psi'(\kappa) } (1+\epsilon )}   \P(S_n^{(\kappa)}\in I(x))  \lesssim e^{ - \epsilon_0  \frac{x}{ \psi'(\kappa) } (1+\epsilon )}.
      \]
      It remains to consider $n_+ \le n <  \frac{x}{ \psi'(\kappa) } (1+\epsilon ) $. Recall that $ 1\ll b(x) \ll \sqrt{ x } $. Then we could take $\lambda = \lambda_x = \delta \frac{b(x)}{\sqrt{ x } }$ with some fixed small $\delta>0$ so that
      \[
      \lambda ( x + 1) + n \psi( \kappa - \lambda ) \le - \lambda_x \psi'(\kappa ) \left(n- \frac{x}{\psi'(\kappa) } \right) /2.
      \]
      It then follows that
      \begin{align*}
      \sum_{n_+\le n \le \frac{x}{ \psi'(\kappa) } (1+\epsilon ) }  \P(S_n^{(\kappa)}\in I(x)) \le  \sum_{n_+\le n \le \frac{x}{ \psi'(\kappa) } (1+\epsilon ) } e^{- \lambda_x \psi'(\kappa ) (n- \frac{x}{\psi'(\kappa) }) /2 } \P\left( S_n^{(\kappa - \lambda_x )} \in I(x) \right)
      \end{align*}
      By use of Berry-Essen inequality, one sees that for any $x\in\R$,
      \[
      \P\left( S_n^{(\kappa - \lambda_x )} \in I(x) \right) \le 3 \frac{\E[ 1 + |S_1^{(\kappa - \lambda_x)}|^3 ]}{ (Var(S_1^{( \kappa - \lambda_x )}))^{3/2}} \frac{1}{\sqrt{n}}.
      \]
      Uniformly for $x\gg1$, we could bound $ \frac{\E[ 1 + |S_1^{(\kappa - \lambda_x)}|^3 ]}{ (Var(S_1^{( \kappa - \lambda_x )}))^{3/2}} $ by some constant. We thus end up with
      \begin{align*}
       \sum_{n_+\le n \le \frac{x}{ \psi'(\kappa) } (1+\epsilon ) }  \P(S_n^{(\kappa)}\in I(x)) \lesssim & \sum_{n_+\le n \le \frac{x}{ \psi'(\kappa) } (1+\epsilon ) } e^{- \lambda_x \psi'(\kappa ) (n- \frac{x}{\psi'(\kappa) }) /2 } \frac{1}{\sqrt{n}} \\
       \lesssim & \frac{1}{\sqrt{ x} } \frac{\sqrt{x}}{b(x)} e^{- \psi'(\kappa) \delta b(x)^2/2 } \lesssim e^{ - \psi'(\kappa) \delta b(x)^2/2}.
      \end{align*}
      This concludes Lemma \ref{generationofM}. 
\end{proof}

\begin{proof}[Proof of Lemma \ref{BRWtight}]

By Markov inequality and Lemma \ref{BRWroughbd}, one sees that for any $x\ge A>0$,
\begin{align*}
\P\left(\sum_{u\in\T}\ind{V(u)-\tM\le A}\ge K\vert \tM\le -x\right)&\lesssim  e^{\kappa x} \frac1K \E\left[\sum_{u\in\T}\ind{V(u)\le \tM+A}\ind{\tM\le -x}\right]\\
& \lesssim \frac{1}{K}e^{\kappa x}\sum_{n\ge0}\E\left[\sum_{|u|=n}\ind{V(u)\le -x+A}\right],
\end{align*}
By Many-to-One Lemma \ref{ManytoOne}, we thus get that
\begin{align}\label{bdpp}
\P\left(\sum_{u\in\T}\ind{V(u)-\tM\le A}\ge K\vert \tM\le -x\right)\lesssim \frac{1}{K}e^{\kappa x}\sum_{n\ge0}\E\left[e^{\kappa S^{(\kappa)}_n}\ind{S_n^{(\kappa)}\le -x+A}\right].
\end{align}
By use of the descending ladder process of $(S^{(\kappa)}_n)_{n\ge0}$, we get that
\begin{align*}
&\sum_{n\ge0}\E\left[e^{\kappa S^{(\kappa)}_n}\ind{S_n^{(\kappa)}\le -x+A}\right]=\E\left[\sum_{k\ge0}\ind{\hat{H}_k^{(\kappa),-}\le -x+A}\sum_{j=\hat{\tau}_k^{(\kappa),-}}^{\hat{\tau}_{k+1}^{(\kappa),-}-1}e^{\kappa S_j^{(\kappa)}}\ind{S_j^{(\kappa)}\le -x+A}\right]\\
=&\E\left[\sum_{k\ge0}e^{\kappa \hat{H}_k^{(\kappa),-}}\ind{\hat{H}_k^{(\kappa),-}\le -x+A}\E\left[\sum_{j=0}^{\hat{\tau}_1^{(\kappa),-}-1}e^{\kappa S_j^{(\kappa)}}\ind{S_j^{(\kappa)}\le r}\right]\bigg\vert_{r=-x+A-\hat{H}_k^{(\kappa),-}}\right].
\end{align*}
Observe that for any $r\ge0$,
\begin{align*}
\E\left[\sum_{j=0}^{\hat{\tau}_1^{(\kappa),-}-1}e^{\kappa S_j^{(\kappa)}}\ind{S_j\le r}\right]&=\sum_{j\ge0}\E\left[e^{\kappa S_j^{(\kappa)}}\ind{\mS_j^{(\kappa)}\ge0, S_j^{(\kappa)}\le r}\right]\\
&=  \int_0^r e^s U_w^{(\kappa), + }(\d s) \lesssim e^r. 
\end{align*} 
It then follows that
\begin{align*}
\sum_{n\ge0}\E\left[e^{\kappa S^{(\kappa)}_n}\ind{S_n^{(\kappa)}\le -x+A}\right]& \lesssim \E\left[\sum_{k\ge0}e^{\kappa \hat{H}_k^{(\kappa),-}}\ind{\hat{H}_k^{(\kappa),-}\le -x+A} e^{-x+A-\hat{H}_k^{(\kappa),-}}\right]\\
& =e^{-x+A}\int_{x-A}^\infty e^{-\kappa' y} U^{(\kappa),-}_s(\d y)\lesssim e^{-\kappa(x-A)}.
\end{align*}
Going back to \eqref{bdpp}, we obtain that for any fixed $A>0$,
\[
\limsup_{x\to\infty}\P\left(\sum_{u\in\T}\ind{V(u)-\tM\le A}\ge K\vert \tM\le -x\right)\lesssim \frac1K e^{\kappa A}=o_K(1)
\] 
as required. 
\end{proof}

\begin{proof}[Proof of Lemma \ref{bdcondmomWD}]
  In fact, it follows from \eqref{eq-bound-W} and Fatou's lemma that
  \[
 \E[ W_{\infty} ^{\kappa+\delta} \ind{\mathbf{M}  \geq -x } ] \lesssim e^{\delta x } .
 \]
  Combining it with Lemma \ref{BRWroughbd} gives that
\begin{equation*}
\begin{aligned}
& \sup_{x\in\mathbb{R}_+} \mathbf{E}\left[\left(\mathcal{W}^{\mathbf{M}}\right)^{\kappa+\delta} \mid \mathbf{M} \leq-x\right]   \lesssim  \sup_{x\in\mathbb{R}_+} e^{\kappa x} \sum_{\ell=0}^{\infty} \mathbf{E}\left[\left(e^{\mathbf{M}} W_{\infty}\right)^{\kappa+\delta} ; \mathbf{M} \in I(x+\ell)\right] \\
 & \lesssim  \sup_{x\in\mathbb{R}_+} \sum_{\ell \geq 0} e^{-\kappa \ell}e^{-\delta(x+\ell)} \mathbf{E}\left[ \left(W_{\infty}\right)^{\kappa+\delta} ; \mathbf{M} \in I(x+\ell)\right] \lesssim \sum_{\ell \geq 0} e^{-\kappa \ell} < \infty,
  \end{aligned}
 \end{equation*}
which proves \eqref{momWM}.
An   immediate adaptation of the above argument proves \eqref{momDM}. 
\end{proof}

\begin{proof}[Proof of Lemma \ref{BRWrest}]
  We upper bound the difference by the  summation of absolute values
\begin{align*}
  & \mathcal{W}^{\tM}-\mathcal{W}^{u^*,\leq t}= \sum_{k= 0}^{|u^{*}|-t} \sum_{z\in \Omega(u^{*}_k)}e^{V(u^{*})-V(z)} W_\infty^{(z)} =: |\mathcal{W}|_{t} , \\
& \left|\mathcal{D}^{\mathbf{M}}-\mathcal{D}^{u^*, \leq t} \right| \leq \sum_{k=1}^{|u^{*}|-t}  
  \sum_{z \in \Omega\left(u^{*}_k\right)}    e^{V(u^{*})-V(z)}  \left[ |D_{\infty}^{(z)} | + |(V(z)+\psi'(1) k) |\, W_{\infty}^{(z)} \right]=: |\mathcal{D}|_{t} .
 \end{align*}
      By Lemma \ref{BRWroughbd}, it suffices to show that for any given $\epsilon >0$, 
      \begin{equation*}
       \lim_{t \to \infty} \sup _{x \in \mathbb{R}_{+}} e^{\kappa x} \left[ \P \left(  \frac{1}{|u^{*}|} |\mathcal{D}|_{t} >  \epsilon,  \mathbf{M} \in I(x) \right) + \P \left(|\mathcal{W}|_{t} >  \epsilon,  \mathbf{M} \in I(x) \right)  \right]= 0 .
      \end{equation*}
     As $\{ \mathbf{M} \in I(x)\} \subset  \cup_{n \geq 1} \{ |u^*|= n, V(u^*) \in I(x)\}$, one sees that 
      \begin{align}
         &  e^{\kappa x} \P \left(  \frac{1}{|u^{*|}} |\mathcal{D}|_{t} >  \epsilon,  \mathbf{M} \in I(x) \right)
         \leq e^{\kappa x}  \E \left[  \sum_{n \geq t+1} \sum_{|u|=n} \ind{ u = u^*, V(u) \in I(x)} \ind{ \frac{1}{n} | \mathcal{D} |_{t}   >  \epsilon  } \right]\nonumber\\
        & \le    \sum_{n \geq t+1} \E_{\mathbf{Q}^{\kappa, *}_{\vert n} \otimes \P} \left[ e^{\kappa V(w_{n}) + \kappa x }    \ind{V(w_n)<\min_{0\le j\le n-1} V(w_j), V(w_{n}) \in I(x)} \ind{ \frac{1}{n} | \mathcal{D} |_{t}   >  \epsilon }  \right] ,\label{badDbd}
        \end{align}
where, in \eqref{badDbd} we made a change of measure  $\d (\mathbf{Q}^{\kappa}_{\vert n} \otimes \P )= W_{n}(\kappa) \d \P$ and used Proposition \ref{BRWchangeofp2}.
 Denote again by $\mathcal{B}_{n}:=\sigma((w_k,V(w_{k})), (V(z): z \in \Omega(w_{k})), 1 \leq k \leq n )$ the information of the spine and their brothers. By the Markov inequality and  the branching property we obtain that, on the event $w_n=u^*$,  
       \begin{align*}
        & \mathbf{Q}^{\kappa, *}_{\vert n} \otimes \P \left( \frac{1}{n} | \mathcal{D} |_{t}   >  \epsilon \mid \mathcal{B}_{n}  \right)  
         \leq \min \left\{ 1, \  \E_{  \mathbf{Q}^{\kappa,*}_{\vert n} \otimes \P } \left[ \frac{1}{\epsilon n}  | \mathcal{D} |_{t} \mid \mathcal{B}_{n} \right] \right\} \\
        & \leq \min \left\{ 1, \ \frac{1}{\epsilon n}      \sum_{k=1}^{n-t}    \sum_{z \in \Omega\left(w_k\right)}    e^{V(w_{n})-V(z)} \left(   \E  |D_{\infty}  | + |(V(z)+\psi'(1) k) |\, \E W_{\infty}  \right)\right\} \\
        & \lesssim  \min \left\{ 1, \ \frac{1}{\epsilon n}      \sum_{k=1}^{n-t}      e^{V(w_{n})-V(w_{k-1})}   \sum_{z \in \Omega\left(w_k\right)} e^{-\Delta V(z)} \left( 1+ | \Delta V(z)| + | V(w_{k-1}) | +k  \right)\right\}
       \end{align*}
       Let $\Delta_{k}:=  \sum_{z \in \Omega(w_{k})} e^{-\Delta V(z)} (1+|\Delta V(z)|)$. Then we get that 
       \begin{align*}
        & (\mathbf{Q}^{\kappa,*}_{\vert n} \otimes \P) \left( \frac{1}{n} | \mathcal{D} |_{t}   >  \epsilon \mid \mathcal{B}_{n}  \right) \lesssim  \min \left\{ 1, \ \frac{1}{\epsilon n}      \sum_{k=0}^{n-t}      e^{V(w_{n})-V(w_{k})} (n+ | V(w_{k})|  ) \Delta_{k+1} \right\}  \\
         & \leq   \frac{1}{\epsilon}      \sum_{k=0}^{n-t}      e^{[V(w_{n})-V(w_{k})]/2 } \left( 1+ \frac {| V(w_{k})| }{n} \right)  +     \sum_{k=0}^{n-t}   \ind{  \Delta_{k+1} >e^{ [V(w_{k})-V(w_{n}) ]/2} } .
       \end{align*}  
    Plugging the above  into \eqref{badDbd} yields that $e^{\kappa x} \P \left(  \frac{1}{|u^{*|}} |\mathcal{D}|_{t} >  \epsilon,  \mathbf{M} \in I(x) \right)$ is bounded from above by 
    \begin{align*}
 & \sum_{n \geq t+1} \E_{\mathbf{Q}^{\kappa,*}_{\vert n} \otimes \P} \left[ e^{\kappa V(w_{n})+\kappa x}    1_{\left\{V(w_{n}) <\min_{0\le j<n}V(w_j), V(w_{n}) \in I(x)\right\}}   (\mathbf{Q}^{\kappa,*}_{\vert n} \otimes \P) \left( \frac{1}{n} | \mathcal{D} |_{t}   >  \epsilon \mid \mathcal{B}_{n}  \right)    \right] \\
    & \leq 
    \frac{1}{\epsilon }  \sum_{n \geq t+1}    \E_{\mathbf{Q}^{\kappa,*}_{\vert n} \otimes \P} \left[ \sum_{k=0}^{n-t}     e^{[V(w_{n})-V(w_{k})]/2 }  \left( 1+ \frac {| V(w_{k})| }{n} \right)    \ind{ V(w_{n}) <\min_{0\le j<n}V(w_j), V(w_{n}) \in I(x) }   \right]    
      \\
    &\quad  + \sum_{n \geq t+1} \sum_{k=0}^{n-t} (\mathbf{Q}^{\kappa,*}_{\vert n} \otimes \P) \left(  V(w_{n}) <\min_{0\le j<n}V(w_j), V(w_{n}) \in I(x) , \Delta_{k+1} > e^{[V(w_{k})-V(w_{n}) ]/2 } \right).
    \end{align*}
Following the same  argument  as above, we get that $ e^{\kappa x} \P \left(  |\mathcal{W}|_{t} >  \epsilon,  \mathbf{M} \in I(x) \right) $ is bounded by 
\begin{align*} 
  & \frac{1}{\epsilon }  \sum_{n \geq t+1}    \E_{\mathbf{Q}^{\kappa,*}_{\vert n} \otimes \P} \left[ \sum_{k=0}^{n-t}   e^{[V(w_{n})-V(w_{k})]/2 }      \ind{ V(w_{n}) <\min_{0\le j<n}V(w_j), V(w_{n}) \in I(x)}  \right]   \\
  &\quad  + \sum_{n \geq t+1} \sum_{k=0}^{n-t} (\mathbf{Q}^{\kappa,*}_{\vert n} \otimes \P) \left(  V(w_{n}) <\min_{0\le j<n}V(w_j), V(w_{n}) \in I(x) ,  \Delta_{k+1} >e^{ [V(w_{k})-V(w_{n}) ]/2 } \right).
 \end{align*} 
Notice that we can replace the probability measure  $\mathbf{Q}^{\kappa,*}_{\vert n} \otimes \P $ by $\mathbf{Q}^{\kappa,*} $ in the above formula. Therefore it is sufficient to prove  that  
 \begin{equation}\label{eq-sum-1}
      \sup _{x \in \mathbb{R}_{+}}  \sum_{n \geq t+1}      \E_{\mathbf{Q}^{\kappa,*} } \left[  \sum_{k=0}^{n-t}     e^{[V(w_{n})-V(w_{k})]/2 } \left( 1+ \frac {| V(w_{k})| }{n} \right)        \ind{ \min\limits_{0\leq j<n} V(w_{j}) > V(w_{n}) \in I(x)   }  \right] \overset{ t \to \infty}{\longrightarrow} 0  
    \end{equation}
     and 
     \begin{equation}\label{eq-sum-2}
      \sup _{x \in \mathbb{R}_{+}}  \sum_{n \geq t+1} \sum_{k=0}^{n-t} \mathbf{Q}^{\kappa,*} \left( \min\limits_{0\leq j<n} V(w_{j}) > V(w_{n}) \in I(x) , 2\ln \Delta_{k+1} >  V(w_{k})-V(w_{n})   \right) \overset{ t \to \infty}{\longrightarrow} 0  .
     \end{equation}
    
     Firstly we show \eqref{eq-sum-1}. Denote the summation  in \eqref{eq-sum-1} by $\Sigma_{\eqref{eq-sum-1}}$.
     Observe that 
    \begin{align*} 
      \Sigma_{\eqref{eq-sum-1}}  = \sum_{n \geq t+1} \E  \left[  \ind{  \min\limits_{0\leq j<n} S^{(\kappa)}_{j}> S^{(\kappa)}_{n} \in I(x) } \sum_{k=0}^{n-t}       e^{[S^{(\kappa)}_{n} -S^{(\kappa)}_{k}]/2}  \left( 1+ \frac{|S^{(\kappa)}_{k}|}{n}     \right)  \right] .
    \end{align*}
  By reversing the time, we get that
    \begin{align}
      \Sigma_{\eqref{eq-sum-1}}  &=  \sum_{n \geq t+1} \E  \left[  \ind{  \max\limits_{1\leq j\leq n} S^{(\kappa)}_{j}< 0 , S^{(\kappa)}_{n} \in I(x) } \sum_{k=0}^{n-t}   e^{S^{(\kappa)}_{n-k}/2}  \left( 1+ \frac{|S^{(\kappa)}_{n-k}-S^{(\kappa)}_{n}| }{n} \right)   \right] \nonumber\\
      &\leq   \sum_{n \geq t+1} \E  \left[  \ind{  \max\limits_{1\leq j\leq n} S^{(\kappa)}_{j}< 0 , S^{(\kappa)}_{n} \in I(x) } \sum_{l=t}^{n}   e^{S^{(\kappa)}_{l} /2 }  \left( 1+ |S^{(\kappa)}_{l}| + \frac{ |S^{(\kappa)}_{n}| }{n} \right)   \right] \nonumber\\
      &\le      \sum_{l \geq t}  \E  \left[  \left( 1+ |S^{(\kappa)}_{l}| \right)  e^{S^{(\kappa)}_{l}/2}  \ind{  \max\limits_{1\leq j\leq l} S^{(\kappa)}_{j}< 0  } \sum_{j=0}^{\infty} \mathbf{P}_{S^{(\kappa)}_{l}} \left( \max_{0 \leq i \leq j} S^{(\kappa)}_{i}<0, S^{(\kappa)}_{j} \in I(x) \right)  \right]  \nonumber\\
      & \quad + \sum_{l \geq t}  \E  \left[  e^{S^{(\kappa)}_{l}/2}  \ind{  \max\limits_{1\leq j\leq l} S^{(\kappa)}_{j}< 0  } \sum_{j=0}^{\infty} \mathbf{E}_{S^{(\kappa)}_{l}} \left(   \frac{ |S^{(\kappa)}_{j}| }{j+l} \ind{ \max_{0 \leq i \leq j} S^{(\kappa)}_{i}<0, S^{(\kappa)}_{j} \in I(x) } \right)  \right] . \label{upbd-eq-sum-1}
    \end{align} 
    On the one hand, \eqref{upperbd} shows that for every $a>0$,
    \begin{equation}\label{eq-bound-Rxa}
      \sum_{j=0}^{\infty} \mathbf{P}_{-a} \left( \max_{0 \leq i \leq j} S^{(\kappa)}_{i}<0, S^{(\kappa)}_{j} \in I(x) \right) \le c_0 (1+a).
    \end{equation}
    On the other hand if $x \leq 2a$, we infer from \eqref{upperbd} that   
    \begin{equation*}
      \sum_{j=0}^{\infty} \mathbf{E}_{-a} \left(   \frac{ |S^{(\kappa)}_{j}| }{j+l} 1_{ \{ \max_{0 \leq i \leq j} S^{(\kappa)}_{i}<0, S^{(\kappa)}_{j} \in I(x)\} } \right) \lesssim (1+a)^{2}. 
    \end{equation*}
    Choose a small  $\lambda$.  If $x \geq 2a$ and $j \leq \lambda x$, 
     $\mathbf{P}_{-a}( S^{(\kappa)}_{j} \in I(x) ) =  \mathbf{P}(S^{(\kappa)}_{j} \in I(x-a)) \leq \mathbf{P}(-S^{(\kappa)}_{j} \geq x/2) \leq \E[ e^{-\delta_{o} S^{(\kappa)}_{j} -\delta_{o} x /2 }] =e^{\psi(\kappa+\delta_{o})j-\delta_o x/2} \leq e^{-\delta_o x/4}$ by taking $\delta_o\in(0,\delta_0)$ and $\lambda = \frac{\delta_o}{ 4\psi(\kappa + \delta_o)}$.  Hence  
    \begin{align*}
      & \sum_{j=0}^{\infty} \mathbf{E}_{-a} \left(   \frac{ |S^{(\kappa)}_{j}| }{j+l} 1_{ \{ \max_{0 \leq i \leq j} S^{(\kappa)}_{i}<0, S^{(\kappa)}_{j} \in I(x)\} } \right) \\
      &  \leq  (1+x) \sum_{j \leq \lambda x } \mathbf{P}_{-a} \left(      \max_{0 \leq i \leq j} S^{(\kappa)}_{i}<0, S^{(\kappa)}_{j} \in I(x)  \right) 
      +  
      \frac{1}{\epsilon \lambda }\sum_{j \geq \lambda x }^{\infty} \mathbf{P}_{-a} \left(  \max_{0 \leq i \leq j} S^{(\kappa)}_{i}<0,        S^{(\kappa)}_{j} \in I(x)  \right) \\
      & \leq (1+x)  e^{-\delta_o x/4 }  + \frac{1}{\epsilon \lambda} c_0 (1+a ) \lesssim (1+a).
    \end{align*} 
    Applying these upper bounds in \eqref{upbd-eq-sum-1} yields that
    \begin{equation*}
     \sup_{x\in\R_+} \Sigma_{\eqref{eq-sum-1}} \lesssim  \sum_{l \geq t}  \E  \left[ (1+|S^{(\kappa)}_{l}|)^{2}  e^{S^{(\kappa)}_{l}/2}  1_{\left\{  \max\limits_{1\leq j\leq l} S^{(\kappa)}_{j}< 0 \right\}} \right] \overset{t \to \infty}{\longrightarrow } 0
    \end{equation*}
    since
    \begin{equation*}
      \sum_{l \geq 0}  \E  \left[ (1+|S^{(\kappa)}_{l}|)^{2}  e^{S^{(\kappa)}_{l}/2}  1_{\left\{  \max\limits_{1\leq j\leq l} S^{(\kappa)}_{j}< 0 \right\}} \right] = \int_0^\infty (1+x)^{2}  e^{-x/2} U_s^{(\kappa),-}(\d x ) < \infty .
    \end{equation*}
    
    We now prove \eqref{eq-sum-2}. Denote the summation  in \eqref{eq-sum-2} by $\Sigma_{\eqref{eq-sum-2}}$.  Notice that $( V (w_{j}) -V (w_{j-1}), \Delta_{j})$ are i.i.d., again by time-reversing, we have
    \begin{align*}
      \Sigma_{\eqref{eq-sum-2}}  
      &= \sum_{n \geq t+1} \sum_{k=0}^{n-t} \mathbf{Q}^{\kappa,*} \left( \max\limits_{1\leq j\leq n} V(w_{j}) <0, V(w_{n}) \in I(x) , 2 \ln \Delta_{n-k} > - V(w_{n-k})   \right) \\
      &= \sum_{n \geq t+1} \sum_{k=t}^{n} \mathbf{Q}^{\kappa,*} \left( \max\limits_{1\leq j\leq n} V(w_{j}) <0, V(w_{n}) \in I(x) , 2 \ln \Delta_{k} > - V(w_{k})   \right)  
    \end{align*}
    Applying the Markov property at time $k$ and by \eqref{upperbd}, it follows  that
    \begin{equation*}
     \begin{aligned}
      \Sigma_{\eqref{eq-sum-2}} &\le \sum_{k=t}^{+\infty}\mathbf{E}_{\mathbf{Q}^{\kappa,*}}\left(\ind{\max\limits_{1 \leq j\leq k}V(w_j)<0, -V(w_k) < 2 \ln{\Delta}_k   }\sum_{n\geq0}\mathbf{P}_{V(w_k)}\left(\max_{1 \leq j \leq n}S^{(\kappa)}_j<0,S^{(\kappa)}_n\in I(x)\right)\right) \\
        &\lesssim \sum_{k=t}^{+\infty}\mathbf{E}_{\mathbf{Q}^{\kappa,*}}  \left(\ind{\max\limits_{1 \leq j\leq k}V(w_j)<0, -V(w_k) < 2 \ln{\Delta}_k   } (1- V(w_k))\right) \\
        &\lesssim \sum_{k=t}^{+\infty} \mathbf{E}_{\mathbf{Q}^{\kappa,*}} \left(\ind{\max\limits_{1 \leq j\leq k}V(w_j)<0, -V(w_k) < 2 \ln{\Delta}_k   }  (1+2\ln{\Delta}_k)\right).
        \end{aligned} 
    \end{equation*}
     Let  $(\Delta,\zeta)$ be a random variable independent of everything and distributed as $(\Delta_1,V(w_1))$. Recall that we set $X:= \sum_{|u|=1} e^{-V(u)} (|V(u)|+1 )$ and notice  that $W_{1}(\kappa) =\sum_{|u|=1}e^{-\kappa V(u)}  \leq X^{\kappa}$.  Consequently, 
    \begin{equation}
     \sup_{x\in\R_+}\Sigma_{\eqref{eq-sum-2}} \lesssim \sum_{k=t+1}^{+\infty} \mathbf{E}_{\mathbf{Q}^{\kappa,*}} \left[(1+2\ln \Delta)\mathbf{P}\left(\max_{1 \leq j \leq k-1}S^{(\kappa)}_j<0, -S^{(\kappa)}_{k-1}\leq  2 \ln \Delta+ \zeta \mid \Delta, \zeta \right) \right] 
    \end{equation}
    Observe from renewal theory that
    \begin{align*}
    & \sum_{k=1}^{+\infty} \mathbf{E}_{\mathbf{Q}^{\kappa,*}} \left[(1+2\ln \Delta)\mathbf{P}\left(\max_{1 \leq j \leq k-1}S^{(\kappa)}_j<0, -S^{(\kappa)}_{k-1}\leq  2 \ln \Delta+ \zeta \mid \Delta, \zeta \right) \right]  \\
    &\lesssim \mathbf{E}_{\mathbf{Q}^{\kappa,*}} \left[ (1+2\ln \Delta) (1+ 2 \ln \Delta+ \zeta )   \right] \leq  \mathbf{E}_{\mathbf{Q}^{\kappa,*}} \left[ (1+2\ln X ) (1+ 2 \ln X+ |V(w_{1})|)   \right]   \\
    &\leq   \mathbf{E}  \left[  \sum_{|z|= 1} (1+ |V(z)|)e^{ -\kappa V(z) } (1+ 2\ln X)^2  \right] \lesssim  \mathbf{E}  \left[    X^{\kappa} \ln^{2} X \right]  <\infty  
    \end{align*}
      by our assumption \ref{cond4}. This suffices to deduce \eqref{eq-sum-2}.  
\end{proof}
 
\begin{proof}[Proof of Lemma~\ref{badmin}]
  Observe that if $G_{n}\left(b\right)$ does not hold then there is $j \leq n-b$ and $z \in \Omega(w_{j})$ satisfying $\mathbf{M}^{(z)} \leq -[ V(z)-V\left(w_n\right)-K_{g}]$. Conditionally on the spine $\mathcal{B}_{n}=\sigma((w_k,V(w_{k})), (V(z): z \in \Omega(w_{k})), 1 \leq k \leq n )$ by the branching property and Lemma \ref{BRWroughbd}  we have  
 \begin{align*}
   \mathbf{Q}^{\kappa,*}    \left(    G_{n} (b)^c \mid \mathcal{B}_{n}   \right) &\leq \min \left\{ 1,   \sum_{j=1}^{n-b} \sum_{z \in \Omega(w_{j})} e^{-\kappa [V(z)-V(w_{n})]}   \right\} =   \min \left\{ 1,   \sum_{j=1}^{n-b}   e^{\kappa [ V(w_{n}) -V(w_{j-1}) ]} \Delta_{j}(\kappa)   \right\} \\
  & \leq  \sum_{j=1}^{n-b}   e^{\kappa [ V(w_{n}) -V(w_{j-1}) ]/2}  + \sum_{j=1}^{n-b} \ind{  2\ln \Delta_{j}(\kappa) > - \kappa [ V(w_{n}) -V(w_{j-1}) ]  }
 \end{align*}
 where   $\Delta_{j}(\kappa):= \sum_{z \in \Omega(w_{j})} e^{- \kappa \Delta V(z) }  $. Then it suffices to show that as $b \to \infty$, 
 \begin{equation*}
  \sup_{x\in \R_+} \sum_{n \geq b+1}  \sum_{j=1}^{n-b} \E_{\mathbf{Q}^{\kappa,*}} \left[   e^{\kappa [ V(w_{n}) -V(w_{j-1}) ]/2} \ind{  V(w_n)\in I(x),V(w_n)<\underline{V}(w_{[0,n-1]})  } \right]\to  0
 \end{equation*}
 and 
 \begin{equation*}
  \sup_{x\in\R_+}   \sum_{n \geq b+1}  \sum_{j=1}^{n-b}  \mathbf{Q}^{\kappa,*} \left(   V(w_n)\in I(x),V(w_n)<\underline{V}(w_{[0,n-1]}) , 2\ln \Delta_{j}(\kappa) > - \kappa [ V(w_{n}) -V(w_{j-1}) ]   \right)\to  0
 \end{equation*} 
These two convergences follow from the same arguments as in  the proof of  \eqref{eq-sum-1} and \eqref{eq-sum-2}. We  omit  the details.
  \end{proof}

\section{Proofs of Theorems \ref{thmtailWM} and   \ref{CondBRW}}

\subsection{Tail probabilities: Theorem \ref{thmtailWM}.}
This section is devoted to proving Theorem \ref{thmtailWM}  which is direct consequences of Theorem \ref{BRWcvg}.

Recall that $\P(\tM\le -x)\sim c_\tM e^{-\kappa x}$ with some constant $c_\tM>0$. Moreover, Theorem \ref{BRWcvg} shows that conditioned on $\{\tM\le -x\}$, $e^{-x}W_\infty$ converges weakly to $e^UZ$ with $U$ and $Z$ independent. Here $U$ has exponential distribution of parameter $\kappa$ and $Z\ge 0$.

\begin{lem}\label{BRWtailM}
Under the assumptions \ref{cond1},\ref{cond2}, \ref{cond3} and \ref{cond4}, we have
\begin{equation}
\limsup_{\epsilon\downarrow0}\limsup_{x\rightarrow\infty}x^{\kappa}\P(W_\infty\leq \epsilon x, e^{-\tM}\geq x)=0.
\end{equation}
\end{lem}
\begin{rem}
Lemma \ref{BRWtailM} implies that  $\P(Z>0)=1$. And Lemma \ref{bdcondmomWD} shows that $Z\in L^\kappa$.
\end{rem}

The proof of Lemma~\ref{BRWtailM} will be given in the end of this subsection. 

\begin{proof}[Proof of Theorem \ref{thmtailWM}]
For any $a>0$ and $x>1$, observe that
\begin{align*}
&x^\kappa\P(W_\infty\geq ax, \tM\leq -\ln x)\\
&=\P(e^{-\tM-\ln x}\mathcal{W}^\tM\geq a\vert \tM+\ln x\leq 0)\P(\tM\le -\ln x)x^{\kappa}.
\end{align*}
It then follows from Theorem \ref{BRWcvg} that
\begin{equation*}
\lim_{x\rightarrow\infty}x^\kappa\P(W_\infty\geq ax, \tM\leq -\ln x) =\P(e^UZ\ge a)c_\tM=:\gamma(a),
\end{equation*}
with 
\begin{equation}\label{gammaa}
\gamma(a)=c_\tM \E\left[\left(1\wedge \frac{Z}{a}\right)^{\kappa}\right], \forall a>0.
\end{equation}
Notice that by Lebesgue's dominated convergence theorem, $\gamma$ is continuous on $\R_+^*$. It is easy to see that $\gamma(0+)=c_\tM$ and that
\[
\lim_{a\downarrow 0+}\gamma(\frac1a)a^{-\kappa}=c_\tM\E[Z^\kappa].
\]

Next, we stem from Theorem \ref{BRWcvg} that conditioned on $\{\tM\le -x\}$, $\frac{D_\infty}{xe^x}$ converges weakly to $\frac{e^UY}{\psi'(\kappa)}$ with $U$ and $Y=(\psi'(\kappa)-\psi'(1))Z$ independent. One thus sees that for any $a>0$,
\begin{align*}
\lim_{x\to\infty}x^\kappa \P(D_\infty \ge  x\ln x, \tM\le -\ln(a x))&= c_\tM a^{-\kappa}\P(e^UY\ge \frac1a \psi'(\kappa))\\
& =c_\tM\E\left[\left(\frac1a \wedge \frac{Y}{\psi'(\kappa)}\right)^\kappa\right].
\end{align*}
By use of  
 \eqref{eq-desired-bound-D}, taking $\epsilon_0 < \frac{1}{2} \delta_0$, we get  that for $\delta>0$, 
\begin{align*}
\P(D_\infty \ge x\ln x, e^{-\tM}< \delta x )& \leq  \frac{\E \left[  |D_{\infty}|^{\kappa+\epsilon_0} 1_{ \{\mathbf{M} \geq -\ln(\delta x) \}} \right]}{(x \ln x)^{\kappa+\epsilon_0}} \\
& \lesssim  \frac{ (\delta x)^{\epsilon_0} (1+\ln\ln x)^{\kappa+\epsilon_0 } }{(x \ln x)^{\kappa+\epsilon_0}}= 
 o(x^{-\kappa}).
\end{align*}

As a result, we obtain that as $x\to\infty$,
\[
x^\kappa\P(D_\infty \ge x\ln x)\to c_\tM\frac{\E[Y^\kappa]}{(\psi'(\kappa))^\kappa}=c_\tM [\frac{\psi'(\kappa)-\psi'(1)}{\psi'(\kappa)}]^\kappa \E[Z^\kappa].
\]
Let 
\[
c_D:=c_\tM [\frac{\psi'(\kappa)-\psi'(1)}{\psi'(\kappa)}]^\kappa \E[Z^\kappa].
\]
Then $c_D\in(0,\infty)$ as $Z>0$ a.s. and $\E[Z^\kappa]<\infty$. By taking $z=x\ln x$, we get \eqref{tailD}.
\end{proof}

\begin{proof}[Proof of Lemma~\ref{BRWtailM}]
  For small $\varepsilon\in(0,1)$, let us consider $P_\varepsilon(x):=\mathbf{P}(W_\infty\leq\varepsilon e^{x},\mathbf{M}\in I(x))$ for $x \geq  1$ and prove that
  \begin{equation*}
    \lim_{\varepsilon\downarrow0} \limsup_{x \to \infty} e^{\kappa x}P_{\varepsilon}(x)=0. 
  \end{equation*}  
  Then the desired result follows from the dominated convergence theorem (as  $e^{\kappa x}P_{\varepsilon} $ is bound  by Lemma \ref{BRWroughbd})
  \begin{align*}
    \limsup_{\epsilon\downarrow0} \limsup_{x \to \infty} e^{\kappa x}\P(W_\infty\leq \epsilon e^{x}, \mathbf{M} \leq -x)=  \limsup_{\epsilon\downarrow0} \limsup_{x \to \infty}  \sum_{j \geq 0} e^{\kappa x} P_{\epsilon}(x+j) = 0.
  \end{align*} 
  
   Recall that  $J(x)$ is defined in \eqref{DefJx} with $1 \ll b(x) \ll \sqrt{x}$. 
   Lemma \ref{generationofM} shows that $|u^*|$ stays in $J(x)$ with high probability. Thus we have 
    \begin{align*}
     &  e^{\kappa x}P_{\varepsilon}(x) \leq\sum_{n\in J(x)}e^{\kappa x} \mathbf{E}\left[\sum_{|u|=n}\mathbf{1}_{\{V(u)=\mathbf{M}<\mathbf{M}_{n-1},V(u)\leq-x\}}\mathbf{1}_{\{W_\infty\leq\varepsilon e^x\}}\right] + o(1) \\
      & \leq  \sum_{n\in J(x)}  \mathbf{E}_{\mathbf{Q}^{\kappa,*}_{\vert n} \otimes \P} \left[ e^{\kappa V(w_{n})+ \kappa x}  \mathbf{1}_{\{V(w_{n} ) < \underline{V}(w_{[1,n-1]}),V(w_{n})\in I(x) \}} \mathbf{1}_{\left\{\sum_{j=1}^n\sum_{z\in\Omega (w_j)}e^{-V(z)}W_\infty^{(z)}\leq\varepsilon e^{x} \right\}}  \right] +o(1)
    \end{align*} 
   where in the last inequality we made a  change of measure  (Proposition \ref{BRWchangeofp2}).   Note that $q:=\mathbf{P}(W_\infty>0)>0$ and $ \mathbf{P}(\min_{|u|=1} V(u) \leq K )>0$ for every $K>0$ sufficiently large. Consequently, for any fixed $b>0$, 
   \begin{align*}
    & e^{\kappa x} P_{\epsilon}(x) \leq  \sum_{n\in J(x)}  \mathbf{E}_{\mathbf{Q}^{\kappa,*}_{\vert n} \otimes \P} \left[   \mathbf{1}_{\{V(w_{n} ) < \underline{V}(w_{[1,n-1]}),V(w_{n})\in I(x) \}} \prod_{j=n-b}^{n} \ind{ \not\exists z \in \Omega(w_{j}), \Delta V(z)\leq K \text{ and } W^{(z)}_{\infty}>0}    \right] \\
    &+   \sum_{n\in J(x)}  \mathbf{E}_{\mathbf{Q}^{\kappa,*}_{\vert n} \otimes \P} \left[   \mathbf{1}_{\{V(w_{n} ) < \underline{V}(w_{[1,n-1]}),V(w_{n})\in I(x) \}}  \ind{ W^{+}_{\infty } \min\limits_{n-b\leq j\leq n} e^{V(w_{n})-V(w_{j-1})-K} \leq \epsilon} \right]  +o(1) \\
    &=: E(x, b,K; \mathrm{extinction}) + E(x,b,K; \mathrm{survival}) + o(1)
   \end{align*}
   where $W^{+}_{\infty}$ is distributed as $\mathbf{P}(W_{\infty} \in \cdot \mid W_{\infty} >0 )$ and is independent of $\mathcal{B}_{n}$. By time reversing we have 
   \begin{align*}
    E(x,b,K; \mathrm{survival}) & = \sum_{n \in J(x)} \P \left[  S^{(\kappa)}_{n} \in I(x) , S^{(\kappa)}_{n} < \underline{S}^{(\kappa)}_{[1,n-1]}, W^{+}_{\infty} \min_{n-b\leq j \leq n}\exp (S^{(\kappa)}_{n}-S^{(\kappa)}_{j-1} ) \leq \epsilon e^{K} \right] \\
    &=  \sum_{n \in J(x)} \P \left[  S^{(\kappa)}_{n} \in I(x) , \overline{S}^{(\kappa)}_{[1,n]} <0, W^{+}_{\infty} \exp ( \underline{S}^{(\kappa)}_{[1,b+1]}) \leq \epsilon e^{K} \right],
   \end{align*}
  where under $\P$, $W^{+}_{\infty}$ is independent to $S^{(\kappa)}_{n}$. Applying Markov property at time $b+1$ gives that
  \begin{align*}
    E(x,b,K; \mathrm{survival}) & \leq \E \left[    \sum_{n\geq 0}\P_{S^{(\kappa)}_{b+1}} \left(   S^{(\kappa)}_{n} \in I(x) , \overline{S}^{(\kappa)}_{[1,n]} <0 \right)  ;   W^{+}_{\infty} \exp ( \underline{S}^{(\kappa)}_{[1,b+1]}) \leq \epsilon e^{K}   ,  \overline{S}^{(\kappa)}_{[1,b+1]} <0  \right] \\
    & \lesssim  \E \left[ 1+|S^{(\kappa)}_{b+1}|;   W^{+}_{\infty} \exp ( \underline{S}^{(\kappa)}_{[1,b+1]}) \leq \epsilon e^{K}   ,  \overline{S}^{(\kappa)}_{[1,b+1]} <0  \right] \\
    & \leq   \E \left( 1+|S^{(\kappa)}_{b+1}| \right) \P \left(    W^{+}_{\infty} \leq \epsilon e^{2 K} \right)   +    \E \left( 1+|S^{(\kappa)}_{b+1}| ;\underline{S}^{(\kappa)}_{[1,b+1]} \leq -K \right) 
  \end{align*}
  where in the second inequality we used  \eqref{eq-bound-Rxa}. Then, using the fact that  $\mathbf{E} (1+|S_{b+1}^{(\kappa)}|)^2  \lesssim (1+b)^{2}$ (as $ \mathbf{E} (S_{1}^{(\kappa)})^2 <\infty $) , we see that
  \begin{equation*}
    \lim_{\varepsilon\downarrow0} \limsup_{x \to \infty}   E(x,b,K; \mathrm{survival}) \lesssim \frac{(1+b)^2}{1+K} .
  \end{equation*}
  
  Now it remains to bound $E(x, b,K; \mathrm{extinction})$. Write $\overline{V}(w_{[1,n]}):=\max_{1\le j\le n}V(w_j)$.
   First operating a time reversal and then applying Markov property at time $b+1$,  we get that 
   \begin{align*}
   E(x, b,K; \mathrm{extinction}) &  =\sum_{n \in J(x)} \E_{\mathbf{Q}^{\kappa,*}_{\vert n} \otimes \P} \left[   \ind{ V(w_{n}) \in I(x), \overline{V}(w_{[1,n]}) < 0 } \prod_{j=1}^{b+1} \ind{ \not\exists z \in \Omega(w_{j}), \Delta V(z) \leq K \text{ and } W^{(z)}_{\infty} >0 }  \right] \\
    & = \E_{\mathbf{Q}^{\kappa,*}_{\vert n} \otimes \P} \left[  \prod_{j=1}^{b+1} \prod_{z \in \Omega(w_{j})} \left( 1- q \ind{\Delta V(z) \leq K} \right) \ind{ \overline{V}(w_{[0,b+1]}) < 0 }  \right.  \\
    & \qquad  \qquad   \times  \left.  \sum_{n \in J(x)} \P_{V(w_{b+1})}  \left(   V(w_{n-b-1}) \in I(x), \overline{V}(w_{[1,n-b-1]}) < 0  \right)  \right]. 
   \end{align*}
  First using  inequality $(1-x) \leq e^{-x }$ and    \eqref{eq-bound-Rxa} and then by Cauchy-Schwarz inequality  we get that 
  \begin{align*}
    E(x, b,K; \mathrm{extinction}) 
    & \lesssim 
    \E_{\mathbf{Q}^{\kappa,*} } \left[  e^{- q \sum_{j=1}^{b+1} \sum_{z \in \Omega(w_{j})} \ind{ \Delta V(z) \leq K } }\left(  1 + | V(w_{b+1}) |  \right)\ind{ \overline{V}(w_{[0,b+1]}) < 0 }  \right] \\
    & \leq  \E_{\mathbf{Q}^{\kappa,*} } \left[  e^{- 2 q \sum_{j=1}^{b+1} \sum_{z \in \Omega(w_{j})}   \ind{ \Delta V(z) \leq K } } \right]^{1/2} \E\left[ \left(  1 + | S^{(\kappa)}_{b+1}   |  \right)^{2} \ind{ \overline{S}^{(\kappa)} _{[1,b+1]} <0 }  \right]^{1/2} . 
  \end{align*}
  Observe that $ \left( \sum_{z \in \Omega(w_{j})}   \ind{ \Delta V(z) \leq K }  : 1 \leq j \leq b+1\right) $ are i.i.d., 
  \begin{align*}
    \E_{\mathbf{Q}^{\kappa} } \left[  e^{- 2 q \sum_{j=1}^{b+1} \sum_{z \in \Omega(w_{j})}   \ind{ \Delta V(z) \leq K } } \right]^{1/2} = \E_{\mathbf{Q}^{\kappa} } \left[  e^{- 2 q  \sum_{z \in \Omega(w_{1})}   \ind{   V(z) \leq K } } \right]^{(b+1)/2} =: c_{K}^{(b+1)/2} . 
  \end{align*}
  Notice that  $c_{K}$ is decreasing in $K$. 
  Provided that $\P( \min_{|z|=1} V(z) \leq K_0)>0$,  for large $K$ we have 
  \begin{equation*}
    c_{K} \leq  c_{K_0} = \E  \left[ \sum_{|u|=1} e^{-\kappa V(u)}  e^{- 2 q  \sum_{|z|=1, z \neq u}   \ind{   V(z) \leq K_0 } } \right] < \E  \left[ \sum_{|u|=1} e^{-\kappa V(u)}    \right]=1 . 
  \end{equation*}  
  Since  $\mathbf{E} (1+|S_{b+1}^{(\kappa)}|)^2  \lesssim (1+b)^{2}$, we get that  
  \begin{equation*}
    \lim_{\varepsilon\downarrow0} \limsup_{x \to \infty}  E(x, b,K; \mathrm{extinction})  \lesssim    (1+b) c_{K_0}^{(b+1)/2}
  \end{equation*}
  
   In summary we have 
   \begin{align*}
    \lim_{\varepsilon\downarrow0} \limsup_{x \to \infty}  e^{\kappa x} P_{\epsilon}(x) \lesssim \frac{(1+b)^2}{1+K} +   (1+b) c_{K_0}^{(b+1)/2}.
   \end{align*}
   Letting $K \to \infty $ first then letting $ b \to \infty$  yields the desired result. 
  \end{proof}

\subsection{Weak convergence conditioned on $W_\infty \ge x$: Proof of Theorem \ref{CondBRW}.}  

\noindent Recall that $\P(W_\infty\ge x)\sim C_0 x^{-\kappa}$. Let us first prove that $C_0=c_\tM\E[Z^\kappa]$ here. Note that for any $A\ge 1$,
\begin{align*}
\P(W_\infty\ge x)&=\P(\mathcal{W}^\tM\ge xe^{\tM})\\
&=\P(\tM+\ln x\le \ln \mathcal{W}^\tM\le A)+\P(\tM+\ln x\le \ln \mathcal{W}^\tM, \ln\mathcal{W}^\tM>A).
\end{align*}
For the second term on the right hand side, observe that
\begin{align*}
&\P(\tM+\ln x\le \ln \mathcal{W}^\tM, \ln\mathcal{W}^\tM>A)\\
&=\sum_{\ell=0}^\infty \P(\tM+\ln x \le \ln \mathcal{W}^\tM,  A+\ell<\ln\mathcal{W}^\tM\le A+\ell+1)\\
&\le \sum_{\ell\ge 0}\P(\tM \le A+\ell+1-\ln x, \mathcal{W}^\tM \ge e^{A+\ell}),
\end{align*}
which by Lemma \ref{BRWroughbd} is bounded by
\begin{align*}
&  \sum_{\ell\ge 0} e^{-\kappa \ln x+\kappa (A+\ell+1)}\P(\mathcal{W}^\tM \ge e^{A+\ell}\vert \tM\le A+\ell+1-\ln x)\\
& \le  x^{-\kappa}\sum_{\ell\ge 0}e^{\kappa (A+\ell+1)}e^{-(\kappa+\delta)(A+\ell)}\sup_{t\in\mathbb{R}_+}\E[(\mathcal{W}^\tM)^{\kappa+\delta}\vert \tM\le -t]\\
& \lesssim x^{-\kappa}e^{-\delta A},
\end{align*}
where the last inequality is obtained by \eqref{momWM}. This implies that
\[
\limsup_{A\to\infty}\limsup_{x\to\infty}x^\kappa \P(\tM+\ln x\le \ln \mathcal{W}^\tM, \ln\mathcal{W}^\tM>A)=0.
\]
On the other hand, by writing $t=\ln x-A$, we have 
\[
\P(\tM+\ln x\le \ln \mathcal{W}^\tM\le A)=\P(\tM\le -t)\P(\ln \mathcal{W}^\tM-(\tM+t)\ge A \ge \ln \mathcal{W}^\tM\vert \tM\le -t)
\]
By use of Theorem \ref{BRWcvg}, 
\[
\lim_{A\to\infty}\lim_{x\to\infty}x^\kappa\P(\tM+\ln x\le \ln \mathcal{W}^\tM\le A)=\lim_{A\to\infty} c_\tM \E[Z^\kappa\ind{\ln Z\le A}]=c_\tM \E[Z^\kappa].
\]
It then follows that
\[
\lim_{x\to\infty} x^\kappa\P(W_\infty\ge x)=c_\tM \E[Z^\kappa].
\]
As $C_0=\lim_{x\to\infty}x^\kappa\P(W_\infty\ge x)$, we conclude that
\[
C_0=c_\tM\E[Z^\kappa].
\]
Recall that $\lim_{a\to0+}\gamma(\frac1a)a^{-\kappa}=c_\tM \E[Z^\kappa]$. Therefore, we obtain that
\begin{equation}\label{largeWsmallM}
  \begin{aligned}
&\limsup_{\varepsilon\downarrow0}\limsup_{x\to\infty}x^\kappa\P(W_\infty\ge x, \M_e< \varepsilon x) \\
&= \lim_{\varepsilon\downarrow0}   \lim_{x\to\infty}x^\kappa  \left[  \P(W_\infty\ge x )    - \P(W_\infty\ge x, \M_e \geq  \varepsilon x)  \right] = c_\tM \E[Z^\kappa] - \lim_{\varepsilon\downarrow0} \epsilon^{-\kappa} \gamma(\frac{1}{\epsilon}) =0.
\end{aligned}
\end{equation} 

Now we are ready to prove Theorem \ref{CondBRW}.

\begin{proof}[Proof of Theorem \ref{CondBRW}]
Note that for any bounded and continuous function $h:\mathbb{R}^3\to\mathbb{R}_+$, 
\begin{align*}
&\E[e^{-\sum_{u\in\T}g(V(u)-\tM)}h(\frac{W_\infty}{x}, \frac{D_\infty}{x\ln x},\tM+\ln x)\ind{W_\infty \ge x}] \\
&=\E[\exp\{-\sum_{u\in\T}g(V(u)-\tM)\}h(\frac{W_\infty}{x}, \frac{D_\infty}{x\ln x}, \tM+\ln x)\ind{W_\infty \ge x, \M_e\ge \varepsilon x}]+o_{x,\varepsilon}(1) x^{-\kappa}\\
&=\P(\tM\le -t)\E[e^{-\sum_{u\in\T}g(V(u)-\tM)}h(\frac{W_\infty}{x}, \frac{D_\infty}{x\ln x},\tM+\ln x)\ind{W_\infty \ge x}\mid \tM\le -t]+o_{x,\varepsilon}(1) x^{-\kappa},
\end{align*}
where we used \eqref{largeWsmallM} in the first equality and we set $t=\ln(\varepsilon x)$  in the second. By use of Theorem \ref{BRWcvg}, 
\begin{align*}
&\lim_{x\to\infty}\E[\exp\{-\sum_{u\in\T}g(V(u)-\tM)\}h(\frac{W_\infty}{x}, \frac{D_\infty}{x\ln x}, \tM+\ln x)\ind{W_\infty \ge x}\mid \tM\le -t]\\
&=\E\left[e^{-\int g(x)\mathcal{E}_\infty(\d x)}h(\varepsilon e^U Z, \varepsilon \frac{e^UY}{\psi'(\kappa)}, -U-\ln\varepsilon)\ind{e^U Z\ge \frac{1}{\varepsilon}}\right]\\
&=\E\left[\int_0^\infty \kappa e^{-\kappa u}e^{-\int g(x)\mathcal{E}_\infty(\d x)}h(\varepsilon e^u Z, \varepsilon \frac{e^uZ(\psi'(\kappa)-\psi'(1))}{\psi'(\kappa)},-u-\ln\varepsilon)\ind{e^u Z\ge \frac{1}{\varepsilon}}\d u\right]
\end{align*}
which by change of variables $r=u+\ln Z+\ln \varepsilon$ is equal to
\[
\E\left[\int_{(\ln Z+\ln \varepsilon)_+}^\infty \kappa e^{-\kappa r} Z^\kappa \varepsilon^\kappa e^{-\int g(x)\mathcal{E}_\infty(\d x)}h(e^r, \frac{\psi'(\kappa)-\psi'(1)}{\psi'(\kappa)}e^r,\ln Z-r)\d r\right].
\]
Note that $Z^\kappa\in L^1(\P)$. By \eqref{tailM}, letting $x\to\infty$ and then $\varepsilon\to0+$, we thus get that
\begin{align*}
&\lim_{x\to\infty}x^\kappa\E[\exp\{-\sum_{u\in\T}g(V(u)-\tM)\}h(\frac{W_\infty}{x}, \frac{D_\infty}{x\ln x}, \tM+\ln x)\ind{W_\infty \ge x}]\\
&=  c_\tM \E\left[Z^\kappa \int_0^\infty \kappa e^{-\kappa r}  e^{-\int g(x)\mathcal{E}_\infty(\d x)}h(e^r, \frac{\psi'(\kappa)-\psi'(1)}{\psi'(\kappa)}e^r, \ln Z-r)\d r\right].
\end{align*}
This implies that
\begin{align*}
&\lim_{x\to\infty}\E\left[\exp\{-\sum_{u\in\T}g(V(u)-\tM)\}h(\frac{W_\infty}{x}, \frac{D_\infty}{x\ln x}, \tM+\ln x) \mid  W_\infty \ge x\right]\\
&= \frac{1}{\E[Z^\kappa]}\E\left[Z^\kappa \int_0^\infty \kappa e^{-\kappa r}  e^{-\int g(x)\mathcal{E}_\infty(\d x)}h(e^r, \frac{\psi'(\kappa)-\psi'(1)}{\psi'(\kappa)}e^r, \ln Z-r)\d r\right]
\end{align*}
since $\P(W_\infty \ge x)\sim C_0 x^{-\kappa}$ with $C_0=c_\tM \E[Z^\kappa]$. This suffices to conclude Theorem \ref{CondBRW}.
\end{proof}

\appendix

\section{}

\subsection{Proofs of Lemmas in Section \ref{RW}}
\label{App}

\begin{proof}[Proof of Lemma \ref{Rfunctioncvg}] $ $

\noindent \textbf{Proof of \eqref{RWsumecvg}}. It suffices to show that there exists some constant $C_\kappa>0$ such that
 \begin{align*} 
  & \lim_{x \to \infty} e^{\kappa x} \sum_{|\psi'(\kappa) n - x | \leq b(x) \sqrt{x}}   \phi_0 \left( \frac{\sqrt{\psi'(\kappa)}}{\sqrt{x}} \left[ n- \frac{x}{\psi'(\kappa)} \right] \right) \mathbf{E}_{-a} \left[  e^{\kappa S^{(\kappa)}_{n}} ; \max_{1 \leq j \leq n} S^{(\kappa)}_{j}<0, S^{(\kappa)}_{n} \in I(x) \right] \\
  &=  C_{\kappa} U^{(\kappa),+}_w([0,a))\int_{-b}^{b} \phi_0 \left(z \right) 
  \frac{1}{  \sqrt{2 \pi \psi''(\kappa) / \psi'(\kappa)^{2} }} 
  e^{- \frac{   z^{2}   }{2 \psi''(\kappa) /  \psi'(\kappa)^{2} }}   \d z . 
  \end{align*}
  where $ U^{(\kappa), +}_w$ is the renewal measure corresponding to the weak ascending ladder process of $(S_n^{(\kappa)})_{n\ge0}$ and $b:= \lim_{x \to \infty} b(x) \in [0,\infty]$.  
  Applying Fubini's theorem, we have 
  \begin{align*}
  &  \sum_{n\in J(x)} \phi_0(\sqrt{\frac{\psi'(\kappa)}{x}}(n-\frac{x}{\psi'(\kappa)}))\E_{-a}\left[e^{\kappa S^{(\kappa)}_n+\kappa x}\ind{\MS^{(\kappa)}_{[1,n]}<0, S^{(\kappa)}_n\in I(x)}\right] \\
    &=  e^{\kappa x} \, \int_{-\infty}^{-x} \kappa e^{\kappa u} \sum_{n \in J(x)} \phi_0 \left( \frac{\sqrt{\psi'(\kappa)}}{\sqrt{x}} \left[ n- \frac{x}{\psi'(\kappa)} \right] \right)  \mathbf{P}_{-a}\left(S^{(\kappa)}_{n} >u, \max_{1 \leq j \leq n} S^{(\kappa)}_{j} <0, S^{(\kappa)}_n \in I(x) \right) \d u \\
    &=    \int_{0}^{1} \kappa e^{-\kappa \lambda} \sum_{n \in J(x)} \phi_0 \left( \frac{\sqrt{\psi'(\kappa)}}{\sqrt{x}} \left[ n- \frac{x}{\psi'(\kappa)} \right] \right)  \mathbf{P}_{-a}\left(  \max_{1 \leq j \leq n} S^{(\kappa)}_{j} <0, S^{(\kappa)}_{n} \in(-x-\lambda,-x]\right) \d \lambda \\
   &  \qquad +    e^{-\kappa } \sum_{n \in J(x)} \phi_0 \left( \frac{\sqrt{\psi'(\kappa)}}{\sqrt{x}} \left[ n- \frac{x}{\psi'(\kappa)} \right] \right)  \mathbf{P}_{-a}\left(  \max_{1 \leq j \leq n} S^{(\kappa)}_{j} <0, S^{(\kappa)}_{n} \in(-x-1,-x]\right).
  \end{align*}

Note that $ \E[S_1^{(\kappa)}]=-\psi'(\kappa)<0$ and   $ \mathrm{Var}(S^{(\kappa)}_{1})= \psi''(\kappa)$. Let $p(\kappa)=\mathbf{P} \left(  \sup_{j\ge 1}S^{(\kappa)}_{j} <0 \right) $. Then $p(\kappa)\in (0,1)$. Now, we claim that for any $\lambda\in(0,1]$,
  \begin{equation}\label{Localrenewal}
      \mathbf{P}_{-a}\left(  \max_{1 \leq j \leq n} S^{(\kappa)}_{j} <0, S^{(\kappa)}_n \in(-x-\lambda,-x]\right)  = [1+o(1)] p(\kappa) U^{(\kappa), +}_w([0,a)) \frac{\lambda}{  \sqrt{2 \pi \psi''(\kappa) n}} e^{- \frac{(x-\psi'(\kappa) n)^{2} }{2 \psi''(\kappa) n}},
  \end{equation}
  where  $o(1)$ is uniformly in $n \in J(x)$.  As a consequence,
\begin{align*}
  & \sum_{n \in J(x)} \phi_0 \left( \frac{\sqrt{\psi'(\kappa)}}{\sqrt{x}} \left[ n- \frac{x}{\psi'(\kappa)} \right] \right)  \mathbf{P}_{-a}\left(  \max_{1 \leq j \leq n} S^{(\kappa)}_{j} <0, S^{(\kappa)}_{n} \in(-x-\lambda,-x]\right) \\
  &\sim   p(\kappa) U^{(\kappa), +}_w([0,a))\sum_{n \in J(x)} \phi_0 \left( \frac{\sqrt{\psi'(\kappa)}}{\sqrt{x}} \left[ n- \frac{x}{\psi'(\kappa)} \right] \right)    \frac{\lambda}{  \sqrt{2 \pi \psi''(\kappa) n}} e^{- \frac{(x-\psi'(\kappa) n)^{2} }{2 \psi''(\kappa) n}} \\
  &\sim p(\kappa) U^{(\kappa), +}_w[0,a) \int_{x/\psi'(\kappa) -b(x) \sqrt{x}}^{ x/\psi'(\kappa)+b(x)\sqrt{x}} \phi_0 \left( \frac{\sqrt{\psi'(\kappa)}}{\sqrt{x}} \left[ z- \frac{x}{\psi'(\kappa)} \right] \right)    \frac{\lambda}{  \sqrt{2 \pi \psi''(\kappa) z}} e^{- \frac{(x-\psi'(\kappa) z)^{2} }{2 \psi''(\kappa) z}} \d z \\
  & \xrightarrow{x\to\infty } p(\kappa) U^{(\kappa), +}_w([0,a))  \int_{-b}^{b} \phi_0 \left(z \right) 
     \frac{\lambda}{  \sqrt{2 \pi \psi''(\kappa)  }} 
     e^{- \frac{   z^{2}   }{2 \psi''(\kappa) /  \psi'(\kappa)^{2} }}   \d z .
\end{align*}
 Then applying the dominated convergence theorem we get 
 \begin{align*}
 & \lim_{x \to \infty}  \sum_{n\in J(x)} \phi_0(\sqrt{\frac{\psi'(\kappa)}{x}}(n-\frac{x}{\psi'(\kappa)}))\E_{-a}\left[e^{\kappa S^{(\kappa)}_n+\kappa x}\ind{\MS^{(\kappa)}_{[1,n]}<0, S^{(\kappa)}_n\in I(x)}\right]  \\
 &=  \left(\int_{0}^{1} \kappa \lambda e^{-\kappa \lambda}  \d \lambda + e^{-\kappa}\right)  \frac{ p(\kappa)}{\psi'(\kappa)} U^{(\kappa), +}_w([0,a))  \int_{-b}^{b} \phi_0 \left(z \right) 
 \frac{1}{  \sqrt{2 \pi \psi''(\kappa)/\psi'(\kappa)^{2}  }} 
 e^{- \frac{   z^{2}   }{2 \psi''(\kappa) /  \psi'(\kappa)^{2} }}   \d z ,
 \end{align*}
 which completes the proof with $C_{\kappa}:= \left(\int_{0}^{1} \kappa \lambda e^{-\kappa \lambda}  \d \lambda + e^{-\kappa}\right)  \frac{ p(\kappa)}{\psi'(\kappa)}$. 
 
 Now it suffices to show the claim \eqref{Localrenewal}.  
 In fact, take $t>0$ with $\psi(\kappa-t)<0$ and $\epsilon\in(0,1/10)$. By Markov's inequality we see that
   \begin{align*}
  & \left| \mathbf{P}_{-a}\left(  \max_{1 \leq j \leq n} S^{(\kappa)}_{j} <0, S^{(\kappa)}_{n} \in(-x-\lambda,-x]\right) -  \mathbf{P}_{-a}\left(  \max_{1 \leq j \leq n^{\epsilon}} S^{(\kappa)}_{j} <0, S^{(\kappa)}_{n} \in(-x-\lambda,-x]\right) \right|  \\
   & \leq   \sum_{j=n^{\epsilon}}^{n} \mathbf{P}_{-a}\left( S^{(\kappa)}_{j} \geq 0  \right) \leq  \sum_{j=n^{\epsilon}}^{n}  \E [ e^{ t S^{(\kappa)}_{j}} ] = \sum_{j=n^{\epsilon}}^{n} e^{\psi(\kappa-t) j} \lesssim e^{- t n^{\epsilon}} .
   \end{align*}
   Using again Markov's inequality, we have
   \begin{align*}
    & \left|   \mathbf{P}_{-a}\left(  \max_{1 \leq j \leq n^{\epsilon}} S^{(\kappa)}_{j} <0, S^{(\kappa)}_{n} \in(-x-\lambda,-x]\right)- \mathbf{P}_{-a}\left(  \max_{1 \leq j \leq n^{\epsilon}} S^{(\kappa)}_{j} <0,  S^{(\kappa)}_{n^{\epsilon}}> -n^{2\epsilon}, S^{(\kappa)}_{n} \in(-x-\lambda,-x]\right) \right|  \\
     & \leq  \mathbf{P}_{ - a } \left( -S^{(\kappa)}_{n^{\epsilon}} > n^{2\epsilon}   \right) \leq e^{- t n^{2\epsilon}}  \E \left[  e^{-t S^{(\kappa)}_{n^{\epsilon}}} \right] = e^{- t n^{ 2 \epsilon} +  \psi(t+\kappa)n^{\epsilon}} \lesssim e^{- t n^\epsilon }.
     \end{align*}
 By the Markov property and local limit theorem (\cite[Corollary 1]{Stone1967})
\begin{align*}
 &  \mathbf{P}_{-a}\left(  \max_{1 \leq j \leq n^{\epsilon}} S^{(\kappa)}_{j} <0, S^{(\kappa)}_{n^{\epsilon}}>-n^{2\epsilon}, S^{(\kappa)}_{n} \in(-x-\lambda,-x]\right)   \\
 &= \mathbf{E} \left[\mathbf{P}_{S^{(\kappa)}_{n^{\epsilon}}}\left( S^{(\kappa)}_{n-n^{\epsilon}} \in(-x-\lambda,-x]\right) ; S^{(\kappa)}_{n^{\epsilon}} \in( -n^{2\epsilon} + a ,a),\max_{1 \leq j \leq n^{\epsilon}} S^{(\kappa)}_{j}< a \right] \\
 &= [1+o(1)]  \frac{\lambda}{  \sqrt{2 \pi \psi''(\kappa) n}} e^{- \frac{(x-\psi'(\kappa) n)^{2} }{2 \psi''(\kappa) n}} \mathbf{P} \left( S^{(\kappa)}_{n^{\epsilon}} \in( -n^{2\epsilon}+ a ,a)  ,\max_{1 \leq j \leq n^{\epsilon}} S^{(\kappa)}_{j}< a \right).
\end{align*}
It follows from the previous arguments that
\[
\left|\mathbf{P} \left( S^{(\kappa)}_{n^{\epsilon}} \in( -n^{2\epsilon}+ a ,a)  ,\max_{1 \leq j \leq n^{\epsilon}} S^{(\kappa)}_{j}< a \right) -  \mathbf{P} \left(  S^{(\kappa)}_{j} < a \text{ for all } j \geq 1 \right) \right|\lesssim e^{- t n^\epsilon }.
\]
It remains to compute $P(a):= \mathbf{P} \left(  S^{(\kappa)}_{j} < a \text{ for all } j \geq 1 \right)  $. 
Recall that $\tau^{(\kappa),+}:=\inf\{ k\ge 1: S_k^{(\kappa)} \ge 0 \}$. Observe that for all $a>0$,
\[
P(a) = P(0) +  \int_{[0, a)} P(a - r ) \P( S_{\tau^{(\kappa),+}} \in \d r).
\]
where $P(0) : = \P( \tau^{(\kappa),+} = \infty ) = p(\kappa)$. So the renewal theory shows that 
\[
P( a ) = P( 0 ) U^{(\kappa),+}_w[0,a).
\]
We thus complete the proof. 

\textbf{Proof of \eqref{RWeSbd}.}
Let us consider 
\begin{equation}
R(x,a):=\sum_{j\geq0}\P_{-a}\left(\MS^{(\kappa)}_{[1,j]}<0, S^{(\kappa)}_j> -x\right), \forall a,x \geq0.
\end{equation}
Apparently, when $a=0$, 
\begin{equation}\label{Ra0}
R(x,0)=U_s^{(\kappa),-}[0,x).
\end{equation}
Note that $(S^{(\kappa)}_i)_{1\leq i\leq j}$ has the same distribution as $(S^{(\kappa)}_j-S^{(\kappa)}_{j-i})_{1\leq i\leq j}$. As a consequence, for any $a>0$,
\begin{align*}
R(x,a)=&\sum_{j\geq0}\P(\MS^{(\kappa)}_{[1,j]}<a, S^{(\kappa)}_j> a-x)\\
=&\sum_{j\geq0}\P\left(S^{(\kappa)}_j< a+\mS^{(\kappa)}_{[0,j-1]}, S^{(\kappa)}_j> a-x\right)
\end{align*}
Let $( \hat{ \tau }_n^{ (\kappa), -}, \hat{H}_n^{ (\kappa ), -})_{n\ge0}$ be the strict descending ladder process. Then
\[
U_s^{ (\kappa), -} (\d x) = \E\left[ \sum_{n\ge 0 } \ind{ \hat{H}_n^{(\kappa),-1} \in \d x}\right].
\]
In addition, the  renewal measure associated with the weak ascending ladder process is 
\[
U_w^{(\kappa), + }( \d x):=\E\left[ \sum_{k=0}^{ \hat{ \tau }_1^{ (\kappa), -}- 1} \ind{ S_k^{(\kappa)} \in \d x} \right].
\]
As the descending ladder process is proper, this implies that
\begin{align*}
R(x,a)=&\sum_{n\geq0}\E\left[\sum_{j=\hat{\tau}_n^{(\kappa),-}}^{\hat{\tau}_{n+1}^{(\kappa),-}-1}\ind{S^{(\kappa)}_j< a+\mS^{(\kappa)}_{[0,j-1]}, S^{(\kappa)}_j> a-x}\right]\\
=&\sum_{n\geq 0} \E\left[\ind{\hat{H}_n^{(\kappa),-}\geq -x}\sum_{j=\hat{\tau}_n^{(\kappa),-}}^{\hat{\tau}_{n+1}^{(\kappa),-}-1}\ind{a-x< S^{(\kappa)}_j<a +\hat{H}_n^{(\kappa),-}}\right],
\end{align*}
which by Markov property at time $\tau_n^{(\kappa),-}$ equals to
\begin{align*}
&\sum_{n\geq 0} \E\left[\ind{\hat{H}_n^{(\kappa),-}\geq-x}U_w^{(\kappa),+}\left((a-x-\hat{H}_n^{(\kappa),-}, a)\right)\right]\\
=&\sum_{n\geq 0} \E\left[\ind{\hat{H}_n^{(\kappa),-}\geq -x}U_w^{(\kappa),+}\left([0, a)\right)\right]-\sum_{n\geq 0} \E\left[\ind{\hat{H}_n^{(\kappa),-}\geq-x}U_w^{(\kappa),+}\left([0, a-x-\hat{H}_n^{(\kappa),-}]\right)\right]
\end{align*}
This means that for any $a>0$ and $x\geq0$, 
\begin{equation}
R(x,a)= U_w^{(\kappa),+}\left( [0, a)\right)R_s^{(\kappa),-}(x) -\int_{[(x-a)_+,x]} U_w^{(\kappa),+}\left([0, a-x+u]\right) U_s^{(\kappa),-}(du).
\end{equation}
As $U_w^{(\kappa),+}\left(\R\right)=C_w^{(\kappa),+}\in(0,\infty) $, one sees that there exists some constant $c_0>0$ such that for any $x\geq0$,
\begin{equation}\label{upperbd}
\sum_{k\geq0}\P_{-a}\left[\MS^{(\kappa)}_{[1,k]}<0, S^{(\kappa)}_k\in I(x)\right]=R(x+1,a)-R(x,a)\leq c_0 (1+a), \forall a>0.
\end{equation}
Note that it also holds for $a=0$ by \eqref{Ra0}. This suffices to conclude \eqref{RWeSbd}.
\end{proof}

\subsection{Convergence of $D_n$: proof of Proposition \ref{Dcvg}.}

For any $p\in(1,\kappa)$, the $L^p$ boundedness of the additive martingale $W_n$ has been proved in \cite[Theorem 2.1]{Liu00}. Although our argument here is not simpler than the one in \cite[Theorem 2.1]{Liu00}, 
to be self-contained, we will use a unified approach to show both $W_n$ and $D_n$ are bounded in $L^p$. 

We will use the following Marcinkiewicz-Zygmund inequality  (see  \cite[Theorem 2 in Section 10.3]{CT03}). For independent centered random variables $\{X_i; 1\le i\le m\}$ and $p\in [1,\infty)$, we have
\begin{equation}\label{sumindep}
\E\left[ \left|\sum_{i=1}^m X_i\right|^p\right]\leq  C_{\eqref{sumindep}}(p) \, \E\left[\left(\sum_{i=1}^m X_i^2\right)^{\frac{p}{2}}\right] , 
\end{equation}
where  $ C_{\eqref{sumindep}}(p) :=  2^p \left\lceil\frac{p}{2}\right\rceil^{p / 2} $.

\begin{proof}[Proof of Proposition \ref{Dcvg}]
Recall that $W_1=\sum_{|u|=1}e^{-V(u)}$ and $D_1=-\sum_{|u|=1}(V(u)+\psi'(1))e^{-V(u)}$ and that
\[
\E[W_1]=1, \textrm{ and } \E[D_1]=0.
\]
Then \eqref{hyp1+} implies that 
\begin{equation}\label{bdmomWD}
\E[|W_1-1|^\alpha]<\infty, \textrm{ and } \E[|D_1|^\alpha]<\infty
\end{equation}
for any $\alpha\in (0,\kappa+\delta_0)$.

Recall that for any $u\in\T$,  $c(u)$ denotes the set of children of $u$ and $\overleftarrow{u}$ denotes the parent of $u$  when $u\neq \rho$. The displacement of $u$ is
\(
\Delta V(u)=V(u)-V(\overleftarrow{u}).
\)
Then for any $u\in\T$, we see that under $\P$, the couple
\[
D_1^{(u)}:=\sum_{z\in c(u)}(-\Delta V(z)-\psi'(1))e^{-\Delta V(z)}\textrm{ and } W_1^{(u)}:=\sum_{z\in c(u)}e^{-\Delta V(z)}
\]
has the same law as $(D_1,W_1)$.

  Observe that   the martingale difference of $W_n$ has the form 
\begin{equation*}
W_{n+1}-W_n= \sum_{|u|=n}e^{-V(u)}\left[W_1^{(u)} -1\right], \forall n\ge0,
\end{equation*}
where, conditioned on $\mathcal{F}_n$, for $|u|=n$, $W_1^{(u)} -1$ are independent centered random variables. Similarly, notice that
for the  martingale difference of $D_n$, we have
\begin{align*}
D_{n+1}-D_n&=\sum_{|u|=n}\sum_{z\in c(u)}(-V(u)-\psi'(1)n-\Delta V(z)-\psi'(1))e^{-V(u)-\Delta V(z)}-D_n\\
&=\sum_{|u|=n}e^{-V(u)}\left[D_1^{(u)}+(-V(u)-\psi'(1)n)[W_1^{(u)}-1]\right], \forall n\ge0,
\end{align*}
where, conditioned on $\mathcal{F}_n$, for $|u|=n$, $D_1^{(u)}+(-V(u)-\psi'(1)n)[W_1^{(u)}-1] $
are independent centered random variables. Applying Marcinkiewicz-Zygmund inequality \eqref{sumindep}, we get that for any $p\in(1, \kappa)$,
\begin{equation} \label{W-MDbd}
  \E\left[|W_{n+1}-W_n|^p \mid \mathcal{F}_n\right] 
  \lesssim_p  \E\left[\left(\sum_{|u|=n} e^{-2V(u)}\left[W_1^{(u)}-1\right]^2\right)^{p/2} \mid \nf_n\right] ,
  \end{equation}
  and 
\begin{align}
\E\left[|D_{n+1}-D_n|^p \mid \mathcal{F}_n\right]  \lesssim_p\E\left[\left(\sum_{|u|=n} e^{-2V(u)}\left(D_1^{(u)}+(-V(u)-\psi'(1)n)[W_1^{(u)}-1]\right)^2\right)^{p/2} \mid \nf_n\right] \label{D-MDbd} .
\end{align}

\textit{Proof of the $L^p$ boundedness of $W_n$.} When $p \leq 2$ and $p \in (1,\kappa)$, by the fact that $(\sum_i x_i)^{p/2}\le \sum_i x_i^{p/2}$ for any $x_i \geq 0$  and by \eqref{bdmomWD},  \eqref{W-MDbd} becomes 
\begin{equation}\label{W-MDbd1}
 \E\left[|W_{n+1}-W_n|^p \mid \mathcal{F}_n\right] 
\lesssim_p  \E\left[\sum_{|u|=n}e^{-pV(u)} |W_1^{(u)}-1|^p   \mid \nf_n\right] 
\lesssim_p  \sum_{|u|=n}e^{-pV(u)}  .
\end{equation}
Taking expectation we get that $
 \E[|W_{n+1}-W_n|^p] \lesssim_{p}    e^{n\psi(p)} $. As $\psi(p)<0$ for $p \in (1,\kappa)$, it is clear that $W_n$ is a $L^p$-bounded martingale. If $\kappa \leq 2$ we have completed the proof. 
 
 In the case $\kappa >2$, fix any $p \in (2,\kappa)$. Let $\delta \in (0, 1- \frac{p-1}{\kappa-1})$ be a small constant such that $p= 2+K\delta$ for some positive integer $K$. Let $p_k:= 2+  k \delta$ for $0 \leq k \leq K$. It is sufficient to show that  $\sup_{n} \mathbf{E} |W_n|^{p_{k}}<\infty$ implies that $\sup_{n} \mathbf{E} |W_n|^{p_{k+1}} <\infty$ for all $ k\leq K-1$. By use of the fact that $(\sum_i a_i x_i)^{p/2}\le (\sum_i a_i)^{p/2-1}\sum_i a_i x_i^{p/2}$ for any $a_i, x_i\ge0$ and by \eqref{bdmomWD} and \eqref{W-MDbd}, we have 
 \begin{align}\label{W-MDbd2}
 \E\left[|W_{n+1}-W_n|^{p_{k+1}} \mid \mathcal{F}_n\right]  
 &\lesssim_p    \left(\sum_{|u|=n}e^{-2V(u)}\right)^{p_{k+1}/2-1}\sum_{|u|=n} e^{-2V(u)}\E\left[ \left| W_1^{(u)}-1 \right|^{p_{k+1}} \mid \nf_n\right]\nonumber\\
  & \lesssim_p\left(\sum_{|u|=n}e^{-V(u)}\right)^{p_{k+1}-2}\sum_{|u|=n} e^{-2V(u)} 
 \end{align}
 Taking expectation and by Lyons' change of measure, one sees that 
 \begin{align}\label{W-DifBound-k}
 \E\left[|W_{n+1}-W_n|^{p_{k+1}}\right] & \lesssim_p  \E_{\Q^{1,*}}\left[W_n^{p_{k+1}-2} e^{-V(w_n)} \right] \nonumber \\
 &
  \lesssim_p  \E_{\Q^{1,*}}\left[W_n^{p_{k}-1}\right]^{\frac{p_{k+1}-2}{p_{k}-1}}\E_{\Q^{1,*}}\left[e^{- \frac{p_{k}-1}{1-\delta} V(w_n)} \right]^{ \frac{1-\delta}{p_{k}-1}}.
 \end{align}
 where the last line comes from H\"{o}lder's inequality and the fact that $\frac{p_{k+1}-2}{p_{k}-1}+ \frac{1-\delta}{p_{k}-1} = 1$, $ \frac {p_{k}-1}{1-\delta} >1 $. 
 Notice that  on the one hand  as $W_n$ is $L^{p_k}$ bounded,
 \begin{equation}\label{eq-W-pk-bounded}
 \E_{\Q^{1,*}}\left[W_n^{p_k-1}\right]= \E\left[ W_n^{p_{k}}\right] \le \sup_{n\ge0}\E[W_n^{p_{k}}]<\infty.
 \end{equation}
 On the other hand, observe that $ \frac {p_{k}-1}{1-\delta} < \frac{p-1}{1-\delta} < \kappa-1$   and compute that 
 \begin{equation}\label{W-value-spine}
  \E_{\Q^{1,*}}\left[e^{- \frac{p_{k}-1}{1-\delta} V(w_n)} \right] = \mathbf{E} \left[ \sum_{|u|=n}  e^{- (1+\frac{p_{k}-1}{1-\delta} )V(u)} \right] = e^{n \psi (1+\frac{p_{k}-1}{1-\delta} )}
 \end{equation}
 and $ \psi (1+\frac{p_{k}-1}{1-\delta} )<0$. It follows from \eqref{W-DifBound-k}, \eqref{eq-W-pk-bounded} and \eqref{W-value-spine} that  $\sup_{n} \mathbf{E} |W_n|^{p_{k+1}} <\infty$.

\textit{Proof of the $L^p$ boundedness of $D_n$.}
When $p\le 2$, by the fact that $(\sum_i x_i)^{p/2}\le \sum_i x_i^{p/2}$ for any $x_i \geq 0$  and by \eqref{bdmomWD},  \eqref{D-MDbd} becomes 
\begin{align}\label{D-MDbd1}
&\E\left[|D_{n+1}-D_n|^p \mid \mathcal{F}_n\right]\nonumber\\
&\lesssim_p \E\left[\sum_{|u|=n}e^{-pV(u)}\left(|D_1^{(u)}|^p +|V(u)+\psi'(1)n|^p |W_1^{(u)}-1|^p \right) \mid \nf_n\right]\nonumber\\
&\lesssim_p \sum_{|u|=n}e^{-pV(u)} (1+|V(u)|^p+n^p).
\end{align}
Note that for any sufficiently small $\varepsilon>0$ such that $1<p\pm\varepsilon<\kappa$, we have
\[
e^{-pV(u)} (1+|V(u)|^p+n^p) \lesssim_p e^{-(p+\varepsilon)V(u)}+e^{-(p-\varepsilon)V(u)}+(1+n^p)  e^{-p V(u)}.
\]
Using it in \eqref{D-MDbd1} and taking expectation yield that 
\begin{align*}
&\E[|D_{n+1}-D_n|^p]\\
&\lesssim_{p}    (1+n^p)\E\left[\sum_{|u|=n} e^{-p V(u)}\right]+\E\left[\sum_{|u|=n} e^{-(p+\varepsilon) V(u)}\right]+\E\left[\sum_{|u|=n} e^{-(p-\varepsilon) V(u)}\right]\\
& =   (1+n^p) e^{n\psi(p)}+e^{n\psi(p+\varepsilon)}+e^{n\psi(p-\varepsilon)}.
\end{align*}
As $\psi(p)<0$ and $\psi(p\pm\varepsilon)<0$, we get that
\[
\sup_{n\ge0}\| D_n\|_p\le \sum_{n\ge0}\E[|D_{n+1}-D_n|^p]^{1/p}<\infty.
\]
This shows that $D_n$ is $L^p$-bounded martingale for any $p\in(1,\kappa\wedge 2)$.

When $p\ge 2$, by use of the fact that $(\sum_i a_i x_i)^{p/2}\le (\sum_i a_i)^{p/2-1}\sum_i a_i x_i^{p/2}$ for any $a_i, x_i\ge0$ and by \eqref{bdmomWD}, \eqref{D-MDbd} becomes 
\begin{align}\label{D-MDbd2}
&\E\left[|D_{n+1}-D_n|^p \mid \mathcal{F}_n\right]\nonumber\\
&\lesssim_p  \left(\sum_{|u|=n}e^{-2V(u)}\right)^{p/2-1}\sum_{|u|=n} e^{-2V(u)}\E\left[ \left|D_1^{(u)}+(-V(u)-\psi'(1)n)[W_1^{(u)}-1] \right|^{p} \mid \nf_n\right]\nonumber\\
&\lesssim_p  \left(\sum_{|u|=n}e^{-V(u)}\right)^{p-2}\sum_{|u|=n} e^{-2V(u)}(1+n^p + |V(u)|^p).
\end{align}
Taking expectation and by Lyons' change of measure, one sees that
\begin{align}
&\E\left[|D_{n+1}-D_n|^p\right]\lesssim_p  \E_{\Q^{1,*}}\left[W_n^{p-2} e^{-V(w_n)}(1+n^p+|V(w_n)|^p)\right]\nonumber\\
&\lesssim_p  \E_{\Q^{1,*}}\left[W_n^{p-1}\right]^{\frac{p-2}{p-1}}\E_{\Q^{1,*}}\left[e^{-(p-1)V(w_n)}(1+n^{p(p-1)}+|V(w_n)|^{p(p-1)})\right]^{\frac{1}{p-1}}.
\end{align}
where the last line comes from H\"{o}lder's inequality. On the one hand, as $W_n$ is $L^p$ bounded,
\begin{equation*}
\E_{\Q^{1,*}}\left[W_n^{p-1}\right]= \E\left[ W_n^p\right] \le \sup_{n\ge0}\E[W_n^p]<\infty.
\end{equation*}
On the other hand, for arbitrary small $\epsilon>0$ such that $p\pm\epsilon\in(1,\kappa)$,
\begin{align*}
&\E_{\Q^{1,*}}\left[e^{-(p-1)V(w_n)}(1+n^{p(p-1)}+|V(w_n)|^{p(p-1)})\right]\\
&=\E\left[ \sum_{|u|=n}e^{-pV(u)}(1+n^{p(p-1)}+|V(u)|^{p(p-1)})\right]
\lesssim_p (1+n^p) e^{n\psi(p)}+e^{n\psi(p+\varepsilon)}+e^{n\psi(p-\varepsilon)}.
\end{align*}
Again, as $\psi(p)<0$ and $\psi(p\pm\varepsilon)<0$, we end up with
\[
\sum_n \E\left[|D_{n+1}-D_n|^p\right]^{1/p}\lesssim_p \sum_n (1+n^p)^{\frac{1}{p(p-1)}} e^{n\frac{\psi(p)}{p(p-1)}}+e^{n\frac{\psi(p+\varepsilon)}{p(p-1)}}+e^{n\frac{\psi(p-\varepsilon)}{p(p-1)}}<\infty.
\]
 This suffices to conclude Proposition \ref{Dcvg}.
\end{proof}

\subsection{Martingale inequalities}

  \subsubsection{Bahr-Esseen type inequality}

 Let $f:[0, \infty) \rightarrow[0, \infty)$ be a nondecreasing convex function with $f(0)=0$ such that
  \begin{equation}\label{eq-Bahr-Esseen-condition}
    f(|z+w|)+f(|z-w|) \leq 2(f(|z|)+f(|w|)) \quad \text { for all } z, w 
  \end{equation} 
  For example if $g(x):=f(\sqrt{x})$ is concave on $(0, \infty)$, then \eqref{eq-Bahr-Esseen-condition} holds.

\begin{lem}\label{lem-Bahr-Esseen}
  Let $\left(M_n,\mathcal{F}_{n} \right)_{n \geq 0}$ be a  martingale with $M_0=0$ a.s. and set $\zeta_n:=M_{n}-M_{n-1}$ for $n \geq 1$. Let $(A_{n})_{n \geq 0}$ be a sequence of event such that 
  \begin{equation*}
    A_{n} \in \mathcal{F}_{n} \text{ and }  A_{n} \subset A_{n-1} \text{ for all }  n\geq 1. 
  \end{equation*}
  Then 
\begin{equation*}
\mathbf{E}\left[f\left(\left| M_n\right|\right) 1_{A_{n}}\right] \leq 4 \sum_{k=1}^n \mathbf{E}\left[f\left(\left|\zeta_k\right|\right)1_{A_{k-1}} \right] \text{ for all } n \geq 1.
\end{equation*}
\end{lem}

\begin{proof}
 The proof closely follows the proof of Lemma A.1 in \cite{IKM20}.  Without loss of generality assume that $\mathbf{E}\left[f\left(\left|\zeta_n\right|\right)\right]< \infty$ for all $n\geq 1$.

Denote by $\zeta_{n}^*$ a random variable such that $\zeta_{n}$ and $\zeta_{n}^*$ are i.i.d. conditionally given $\mathcal{F}_{n-1}$. Then 
\begin{equation*} 
\left|M_{n-1}+\zeta_{n}\right|   =\left|\mathbf{E}\left[M_{n-1}+\zeta_{n}-\zeta_{n}^* \mid M_{n-1}, \zeta_{n}\right]\right|  \leq \mathbf{E}\left[\left|M_{n-1}+\zeta_{n}-\zeta_{n}^*\right| \mid M_{n-1}, \zeta_{n}\right].
\end{equation*}
By using that $f$ is increasing and  Jensen's inequality for conditional expectation  
\begin{align*}
\E[ f(|M_{n}|)1_{A_{n}} ] &= \mathbf{E}\left[f\left(\left|M_{n-1}+\zeta_{n}\right|\right)1_{A_{n}} \right] 
\leq  \mathbf{E}\left[f\left(\left|M_{n-1}+\zeta_{n}\right|\right)1_{A_{n-1}} \right]\\
&
\leq \mathbf{E}\left[ f\left(\mathbf{E}\left[\left|M_{n-1}+\zeta_{n}-\zeta_{n}^*\right| \mid M_{n-1}, \zeta_{n}\right]\right) 1_{A_{n-1}}  \right]   \leq \mathbf{E}\left[f\left(\left|M_{n-1}+\zeta_{n}-\zeta_{n}^*\right|\right) 1_{A_{n-1}}\right].
\end{align*}
Further, as   $\zeta_{n}$ and $\zeta_{n}^*$ are i.i.d. conditionally on $\mathcal{F}_{n-1}$.
\begin{equation*}
\begin{aligned}
& \mathbf{E}\left[f\left(\left|M_{n-1}+\zeta_{n}-\zeta_{n}^*\right|\right) 1_{A_{n-1}} \right] =\mathbf{E}\left[\mathbf{E}\left[f\left(\left|M_{n-1}+\zeta_{n}-\zeta_{n}^*\right|\right) \mid \mathcal{F}_{n-1}\right] 1_{A_{n-1}} \right] \\
& =\mathbf{E}\left[\mathbf{E}\left[f\left(\left|M_{n-1}-\left(\zeta_{n}-\zeta_{n}^*\right)\right|\right) \mid \mathcal{F}_{n-1}\right] 1_{A_{n-1}} \right] 
  =\mathbf{E}\left[ f\left[\left|M_{n-1}-\left(\zeta_{n}-\zeta_{n}^*\right)\right|  1_{A_{n-1}} \right]\right]
\end{aligned}
\end{equation*}
Thus  an appeal to \eqref{eq-Bahr-Esseen-condition} thus yields
\begin{align*}
&\mathbf{E}\left[f\left(\left|M_{n-1}+\zeta_{n}-\zeta_{n}^*\right|\right) 1_{A_{n-1}}\right] \\ 
& = \frac{1}{2} \left( \mathbf{E}\left[f\left(\left|M_{n-1}+\zeta_{n}-\zeta_{n}^*\right|\right) 1_{A_{n-1}}\right] + \mathbf{E}\left[ f\left[\left|M_{n-1}-\left(\zeta_{n}-\zeta_{n}^*\right)\right|  1_{A_{n-1}} \right]\right]  \right) \\
& 
\leq \mathbf{E}\left[f\left(\left|M_{n-1}\right|\right)  1_{A_{n-1}} \right]+\mathbf{E}\left[f\left(\left|\zeta_{n}-\zeta_{n}^*\right|\right)  1_{A_{n-1}} \right]
\end{align*} 
Another application of \eqref{eq-Bahr-Esseen-condition} yields
\begin{equation*}
  \mathbf{E}\left[f\left(\left|\zeta_{n}-\zeta_{n}^*\right|\right) 1_{A_{n-1}} \right] \leq 4 \mathbf{E}\left[f\left(\left|\zeta_{n}\right|\right)  1_{A_{n-1}} \right].
\end{equation*}
Finally we get 
\begin{equation*}
  \E[ f(|M_{n}|)1_{A_{n}} ] \leq \mathbf{E}\left[f\left(\left|M_{n-1}\right|\right)  1_{A_{n-1}} \right]+  4 \mathbf{E}\left[f\left(\left|\zeta_{n}\right|\right)  1_{A_{n-1}} \right]. 
\end{equation*}
Hence the desired result follows. 
\end{proof}

\subsubsection{Rosenthal's inequality}

First we state a generalization of Rosenthal  inequality for nonnegative random variables. 

\begin{lem}[{\cite[Theorem 5.1]{Hitczenko90}}]\label{lem-Rosenthal-inequality}
   For any adapted sequence $\left(X_k,\mathcal{F}_k\right)_{k \geq 1}$ of nonnegative random variables the following inequality is true:
\begin{equation*}
   \mathbf{E}\left( \sum_{k\geq 1} X_k  \right)^p 
    \leq \left( K_0 p \right)^p 
    \left[ \mathbf{E}\left(   \sum_{k \geq 1} \mathbf{E}[ X_k | \mathcal{F}_{k-1}] \right)^p   + \E\left(   \sup_{k \geq 1}  X_{k}^p \right) \right], \quad p \geq 1,
\end{equation*} 
  for some absolute constant $K_0$.
\end{lem}

The following lemma is an easy consequence of the Rosenthal  inequality. 

\begin{lem}\label{lem-p-moment}
Let $(A_{n}), (B_{n})$ be two   nonnegative $\mathcal{F}_{n}$-adapted processes. Let $p>1$.  Assume that $B_{n}$ is independent of $\mathcal{F}_{n-1}$ and has the same law as $B_{1}$ for all $n$ with $ \E(B_1^p)<\infty$,
  then for $n \geq 2$
\begin{equation*}
  \E \left( \sum_{k=1}^{n-1} A_{k} B_{k+1} \right)^{p}  \leq (Kp)^p \E[B_1^p] \ \E \left( \sum_{k=1}^{n-1} A_{k}  \right)^{p}  
\end{equation*}
for some absolute constant $K$.
\end{lem}

\begin{proof}
Apply the Rosenthal  inequality above with $X_{k+1}= A_{k} B_{k+1}$. Since  $\mathbf{E}[ X_{k+1} | \mathcal{F}_{k}]= A_{k} \mathbf{E} (B_{1})$, we have 
\begin{equation*}
  \E \left( \sum_{k=1}^{n-1} A_{k} B_{k+1} \right)^{p}  \leq   \left( K_0 p \right)^p 
  \left[  (\E B_{1} )^p  \mathbf{E}\left(   \sum_{k =1}^{n-1}  A_k \right)^p   +  \E\left(  \sup_{1 \leq k \leq n-1}A^p_{k} B_{k+1}^p \right) \right] . 
\end{equation*} 
Observe that  $(\E B_{1} )^p \leq \E [B_{1} ^p] $ and  
\begin{equation*}
  \E\left(  \sup_{1 \leq k \leq n-1}A^p_{k} B_{k+1}^p \right) \leq \sum_{k=1}^{n-1}   \E\left(  A^p_{k} B_{k+1}^p \right) =  \E [B_{1}^p ] \  \sum_{k=1}^{n-1}   \E\left(  A_{k}^p \right) \leq  \E[B_1^p] \ \E \left( \sum_{k=1}^{n-1} A_{k}  \right)^{p}  
\end{equation*} 
the desired result follows. 
\end{proof}

\bibliographystyle{alpha}

\bibliography{biblio}

\end{document}